\documentclass[11pt]{amsart}

\usepackage[dvipsnames]{xcolor}
\usepackage{amssymb,verbatim}
\usepackage{amsmath,amsfonts,enumitem}
\usepackage[mathscr]{euscript} % make a nice \Net
\usepackage{amsthm}
\usepackage{url}
\usepackage{graphicx} %to include figure
\usepackage{enumitem} %to use nice enumeration with changed indent
\usepackage{hyperref} % to make links for lemmas and theorems
\hypersetup{colorlinks} %so the links look nicer
\usepackage{bbm} % for the characteristic function 1
\usepackage{bm} % for bold greek letter
\usepackage{mathrsfs} %to use \mathscr
\usepackage{cancel} %to make sout to cross things out

\usepackage[colorinlistoftodos]{todonotes}

\RequirePackage{cleveref}
\usepackage{hypcap}
\hypersetup{colorlinks=true, citecolor=darkblue, linkcolor=darkblue}
\definecolor{darkblue}{rgb}{0.0,0,0.7}
\newcommand{\darkblue}{\color{darkblue}}

\definecolor{darkred}{rgb}{0.68,0,0}
\newcommand{\darkred}{\color{darkred}}

\definecolor{darkgreen}{rgb}{0,.38,0}
\newcommand{\darkgreen}{\color{darkgreen}}

\newcommand{\defn}[1]{\emph{\darkblue #1}}
\newcommand{\defna}[1]{\emph{\darkred #1}}

\newcommand{\defng}[1]{\emph{\darkgreen #1}}

\makeatletter
\def\th@plain{%
	\thm@notefont{}% same as heading font
	\itshape % body font
}
\def\th@definition{%
	\thm@notefont{}% same as heading font
	\normalfont % body font
}
\makeatother

\newtheorem{thm}{Theorem}[section]

\newtheorem*{claim*}{Claim}

\newtheorem{prop}[thm]{Proposition}
\newtheorem{conj}[thm]{Conjecture}
\newtheorem{conv}[thm]{Convention}

\newtheorem{op}[thm]{Open Problem}

\theoremstyle{definition}
\newtheorem{ex}[thm]{Example}
\newtheorem{rem}[thm]{Remark}

\numberwithin{figure}{section}
\numberwithin{equation}{section}

\def\wh{\widehat}

\def\sq{\square}

\def\zz{\mathbb Z}
\def\nn{\mathbb N}
\def\cc{\mathbb C}

\def\rr{\mathbb R}
\def\qqq{\mathbb Q}

\def\ff{\mathbb F}

\def\SS{\mathbb S}

\def\pp{\mathbb P}

\def\ov{\overline}

\def\sm{\smallsetminus}

\def\Ga{\Gamma}
\def\Om{\Omega}

\def\la{\lambda}
\def\ga{\gamma}
\def\si{\sigma}
\def\de{\delta}

\def\al{\alpha}
\def\be{\beta}
\def\om{\omega}

\def\vp{\varphi}

\def\cC{\mathcal C}
\def\cB{\mathcal B}

\def\cA{\mathcal A}
\def\cd{\mathcal D}

\def\cT{\mathcal T}

\def\ssu{\subset}
\def\sss{\supset}
\def\wt{\widetilde}
\def\<{\langle}
\def\>{\rangle}

\def\PM{\text{{\rm PM}}}
\def\sign{\text{\rm sign}}

\def\GL{ {\text {\rm GL} } }

\def\rG{ {\text {\rm F} } }

\def\Cat{ {\text {\rm Cat} } }
\def\Aut{ {\text {\rm Aut} } }

\def\rg{\overline{g}}

\def\rT{{\text {\rm T} } }

\def\rK{{\text{\text{K}}}}

\def\St{{\text {\rm Stab} } }

\def\0{{\mathbf 0}}

\def\ED{{\mathcal{E}}}

\def\.{\hskip.06cm}
\def\ts{\hskip.03cm}

\def\lra{\leftrightarrow}

\def\bx{{\textbf{\textit{x}}}}
\def\by{{\textbf{\textit{y}}}}
\def\bz{{\textbf{\textit{z}}}}

\def\ba{\textbf{\textit{a}}}
\def\bb{\textbf{\textit{b}}}

\def\bK{\textbf{\textit{K}}}

\def\La{\Lambda}

\def\di{{\small{\ts\diamond\ts}}}

\def\CI{{\mathcal{C}}}
\def\BG{{\mathcal{B}}}
\def\Mon{{\mathcal{M}}}
\def\Obli{{\mathcal{O}}}

\newcommand{\maj}{\mathrm{maj}}
\newcommand{\jdt}{\mathrm{jdt}}

\newcommand{\PP}{\operatorname{PP}}

\newcommand{\SSYT}{\operatorname{SSYT}}
\newcommand{\SYT}{\operatorname{{\rm SYT}}}
\newcommand{\SVT}{\operatorname{{\rm SVT}}}
\newcommand{\IT}{\operatorname{{\rm IT}}}
\newcommand{\LR}{\operatorname{{\rm LR}}}
\newcommand{\RHT}{\operatorname{{\rm RHT}}}
\newcommand{\Cor}{\mathrm{Cor}}
\newcommand{\RC}{\operatorname{{\text{RC}}}}
\newcommand{\Sch}{\operatorname{\mathfrak{S}}}

\def\aN{\textrm{N}}

\def\.{\hskip.06cm}
\def\ts{\hskip.03cm}

\def\nin{\noindent}

\def\SP{{\textsf{\#P}}}
\def\SPr{{{{\sf \#P}}}}
\def\SEXP{{\textsf{\#EXP}}}

\def\NP{{\textsf{NP}}}
\def\coNP{{\textsf{coNP}}}
\def\PSPACE{{\textsf{PSPACE}}}

\def\PPA{{\textsf{PPA}}}
\def\GapP{{\textsf{GapP}}}
\def\GapPr{{\sf{GapP}}}
\def\NPr{{{\sf NP}}}
\def\FPr{{{\sf FP}}}
\def\coNPr{{{\sf coNP}}}

\def\PH{{\textsf{PH}}}
\def\PHr{{{\sf PH}}}

\def\CeqP{{{\sf C_=P}}}

\def\CCF{{\textsf{CCF}}}

\def\coNP{{\textsf{coNP}}}
\def\FP{{\textsf{FP}}}
\def\UP{{\textsf{UP}}}
\def\coUP{{\textsf{coUP}}}
\def\poly{{\textsf{P}}}
\def\Pr{{{\sf P}}}

\def\GapP{{\textsf{GapP}}}

\newcommand{\inv}{\operatorname{{\rm inv}}}

\DeclareMathOperator{\Ec}{\mathcal{E}} %The set of all linear extensions
\title{What is a combinatorial interpretation?}
\date{\today}

\thanks{\thinspace ${\hspace{-.54cm}}^\star$Department of Mathematics,
UCLA, Los Angeles, CA 90095. \ \, Email: \.
\texttt{(pak@)math.ucla.edu}.
}

\author[\ Igor Pak]{Igor Pak$^\star$}
\begin{document}

\begin{abstract}
In this survey we discuss the notion of \emph{combinatorial
interpretation} in the context of Algebraic Combinatorics and related
areas.  We approach the subject from the Computational Complexity
perspective.  We review many examples, state a workable definition,
discuss many open problems, and present recent results on the subject.
\end{abstract}
	
\maketitle
	
%\vskip.2cm

\section{Introduction}

\medskip

\subsection{What numbers?}\label{ss:intro-numbers}
Traditionally, \defna{Combinatorics} \ts works with numbers.  Not with structures,
relations between the structures, or connections between the relations ---
just numbers.  These numbers tend to be nonnegative integers, presented
in the form of some exact formula or disguised as probability.  More
importantly, they always count the number of some combinatorial objects.

This approach, with its misleading simplicity, led to a long series
of amazing discoveries, too long to be recounted here.  It turns out
that many interesting combinatorial objects satisfy some formal
relationships allowing for their numbers to be analyzed.
More impressively, the very same combinatorial objects
appear in a number of applications across the sciences.

Now, as structures are added to Combinatorics, the nature of the numbers
and our relationship to them changes.  They no longer count something
explicit or tangible, but rather something ephemeral or esoteric,
which can only be understood by invoking further results in the area.
Even when you think you are counting something combinatorial,
it might take a theorem or a even the whole theory to realize
that what you are counting is well defined (see e.g.~$\S$\ref{ss:seq-knots}).

This is especially true in \defna{Algebraic Combinatorics} \ts where the numbers
can be, for example, dimensions of invariant spaces, weight multiplicities or
Betti numbers.  Clearly, all these numbers are nonnegative integers,
but as defined they do not count anything per se, at least in the
most obvious or natural way.

\smallskip

\subsection{Why combinatorial interpretations?}\label{ss:intro-CI}
This brings us to the popular belief that one should always look
for a combinatorial interpretation (see~$\S$\ref{ss:finrem-hist}).
As we see it, there are two reasons for this.

The first reason is clear:
when you know what you are counting you have access to a
large toolkit already developed in Enumerative Combinatorics and areas
further afield.  Essentially, the explicit combinatorial objects serve
as a common playground where the areas can meet and be understood
(see~$\S$\ref{ss:Main-examples}).

\smallskip

The second, deeper reason, is largely based on the hope is that
a combinatorial interpretation would reveal some structures hidden
in the algebraic objects they are working with.  One can think of
combinatorial interpretations as projections --- you gather
enough projections and hope the whole structure emerges.
Consequently, when a combinatorial interpretation is found
it is hard to tell if it points to a new structure without
further study.

\smallskip

From our perspective, the first reason is terrific
and can bring a lot of activity as new combinatorial
objects rise to prominence in the areas spurred by applications.
Meanwhile, the second reason is unfortunate and indicates that
the area does not have a workable definition of a
``combinatorial interpretation''.  This
brings us to the following difficult question.

% When one \defng{proves} \ts that
% there is no combinatorial interpretation (this is rare), they refuse to
% engage long enough to understand the statement and meekly suggest that
% there might be ``another kind of combinatorial interpretation''.
% This may fall short of a willful ignorance, but the signs are
% unmistakable.

\smallskip

\subsection{What do we mean by a combinatorial interpretation?}\label{ss:intro-def}
Well, this is what this survey is about.  The short answer is $\ts \SP$,
a computational complexity class which we discuss
at length.

But before we proceed, let us make a trivial comment.  The first step
towards building a theory is admitting the need for a formal definition.
Without that, the negative results are impossible to state while the
positive results ring hollow and cannot be fully appreciated for the
miracle that they are.

According to Popper's philosophy, a belief needs to be \defna{disprovable}
in order to be \defna{scientific} \cite{Pop62}.  Those who have
unquestionable faith in the existence of combinatorial interpretations
for all problems they care about, might want to take this lesson to heart.

\smallskip

\subsection{There is no there there}\label{ss:intro-NOCI}
We argue that many (all?) long-standing open problems on combinatorial
interpretations in Algebraic Combinatorics can be formalized and resolved.
We believe that few if any of them will have a solution of the kind that
people in the area are looking for.

We aim to give a negative solution to many of the combinatorial interpretation
problems using a formal definition we mentioned above.
The goal of this paper is to advance this as part of a larger
project.  Until recently, this seemed overly ambitious and beyond the reach.
Hopefully, this survey will leave you more optimistic.

\smallskip

\subsection{Why bother?}\label{ss:intro-bother}
Given that until recently the notion of ``combinatorial interpretation''
had been informal, why set on a quixotic journey?  Let us frame the
question in much broader terms, from the perspective of
\defn{Computational Combinatorics}.
There are two fundamental questions we want to address in our study:

\smallskip

\qquad
$(1)$ \.
How do you prove that given numbers \emph{have} a \ts ($\SP$) \ts combinatorial interpretation?

%\smallskip
\qquad
$(2)$ \.
How do you prove that they \emph{do not}?

\smallskip

\nin
Note that we are not so much interested in \emph{finding} \ts
an explicit combinatorial interpretation, just proving membership
in~$\SP$ suffices for our purposes.
Mostly, we are interested in development of new tools
coming both from Combinatorics and Computational Complexity,
to resolve the questions above.

In the last few decades, the area of Algebraic Combinatorics did a
great job proving relevance and applicability beyond its boundaries.
From this point of view, the open problems of finding combinatorial
interpretations for numbers such as the \defng{Kronecker} and
\defng{Schubert coefficients} are key benchmarks.  Resolving them
in either direction would be an important achievement in the
whole Mathematics.

\smallskip

\subsection{Why computational complexity?}\label{ss:intro-CC}
In other words, is there perhaps a more suitable and easier to
understand language in which the problem can be phrased?
Perhaps, in terms of integer points in convex polytopes or
tilings in the plane?  Indeed, both type of combinatorial
interpretations do frequently appear in Algebraic Combinatorics
and may seem like the natural place to start.

This is both the easiest and the hardest question to answer.  The short
answer is this: ``Computational complexity provides the broadest and
the most flexible language''.  In fact, prior to converging
to~$\SP$, we tried a handful of approaches including the ones above.
Since the types of ``combinatorial interpretations'' they gave were
rather constrained, we reasoned that negative results would be
easier to obtain.  Eventually we discarded all such formulations
as unconvincing and resigned to the fundamental difficulty and
its ever conditional nature of Computational Complexity.

As we see it now, the Computational Complexity is truly
foundational for the whole of Mathematics, and allows
one to ask questions on the deeper level.  While we constrain
ourselves with the problem at hand, we refer to~\cite{Wig19}
for the general picture.

\smallskip

\subsection{What to expect from this survey}\label{ss:intro-WTE}
We assume that the reader is familiar with standard notation, results
and ideas in Algebraic Combinatorics, see e.g.\ \cite{Mac95,Man01,Sag01,Sta99}.
After some hesitation, we decided to assume the same about
Computational Complexity.  We realize this might be unreasonable,
and we will provide plenty of examples, but in the 21st century,
a research survey is probably not the best place to include a long
list of basic definitions.
% so as not to compete with {\tt Wikipedia}.

We do include a quick overview of basic complexity notions
(Section~\ref{s:basic}), but stop short of getting technical.
For more details, we recommend \cite{MM} which is a fun read,
\cite{AB09,Pap} for standard textbooks on the subject, and
\cite{Aa} for a very readable survey.

% What we do is survey many combinatorial interpretations in the area
% and how they can be viewed from the computational point of view.
% We separate those which are routine from those which should be
% viewed as minor miracles.  We also discuss a more general class
% of problems where combinatorial interpretations are of interest
% and our recent work on the subject.

\subsection{Structure of the paper}\label{ss:intro-structure}
In Section~\ref{s:basic}, we introduce the computational complexity
language, followed by Section~\ref{s:Main} with many motivating examples.
In Section~\ref{s:seq}, we discuss many examples from Enumerative Combinatorics,
the original motivation for combinatorial interpretations (cf.~$\S$\ref{ss:finrem-hist}).
In Sections~\ref{s:sub} and~\ref{s:LE} we survey inequalities in Probabilistic
Combinatorics and Order Theory, respectively.  The selection of results is
somewhat biased and reflects some of our own interests.  The goal is prepare
the reader for more difficult problems later on.

In Sections~\ref{s:tab}--\ref{s:Schubert} is the core part of the survey.
Here we discuss many functions in
Algebraic Combinatorics centered around four subjects: Young tableaux,
$S_n$~characters, Kronecker and Schubert coefficients.  These should be
read in order, as they build on top of each other.  Sections~\ref{s:Kron}
and~\ref{s:Schubert} have polemical portions at the end, which some readers
might disagree with.

The last part of the survey presets our effort to organize the earlier material
and make digestible conclusions.  In Section~\ref{s:Sym}, we discuss results and
bijections in Algebraic Combinatorics centered around the \defng{LR~rule}
and the \defng{RSK correspondence}, which we
consider crucially important to the subject.  In Section~\ref{s:no}, we discuss
our recent work \cite{IP22} which develops tools to prove nonexistence of
combinatorial interpretations.

In Section~\ref{s:addendum}, we give an annotated list of $\SP$-completeness
and $\SP$-hardness results that we omitted earlier to avoid the confusion.
These last three sections are independent from each other and should be
accessible to experts in the area who skipped earlier section.
We conclude with proofs postponed from earlier sections (Section~\ref{s:proofs}),
and some final remarks (Section~\ref{s:finrem}).

\smallskip

\subsection*{Notation}
Let $\{0,1\}^n$ denote the set of sequences of zeros and ones of length~$n$,
and let $\{0,1\}^\ast$ the set of all such sequences of finite length.
We use \ts $\nn=\{0,1,2,\ldots\}
$ \ts and \ts $[n]=\{1,\ldots,n\}$.  The rest is
pretty self-explanatory.
\bigskip

\section{Basic Computational Combinatorics}\label{s:basic}

For the purposes of this survey, we will make numerous
shortcuts and imprecise statements, largely sacrificing standard
definitions and rigor for the sake of clarity and conciseness.  We also make
our focus quite a bit more narrow than it could be.  We beg
forgiveness to the experts.

\smallskip

\subsection{Combinatorial objects}\label{ss:basic-obj}
The notion of a \defna{combinatorial object} can be viewed as follows.
A \defn{word} is a binary sequence $x\in \{0,1\}^\ast$.
The \defn{size} of $x$ is the length~$|x|$.   In other words,
combinatorial objects of size~$N$ are encoded by words
of length~$N$, which in turn correspond to integers
\ts $0\le a <2^N$.\footnote{We realize this is not how some
combinatorialists think of combinatorial objects.}  For example,
a simple graph on $n$ vertices can be viewed as
a word of length $\binom{n}{2}$.

Note that we view \emph{combinatorial object} as more than
an abstract concept because the definition includes the
\defn{presentation}.  For example, a simple graph on $n$ vertices
and $m$ edges can also be presented as a list of edges.
The resulting word would be of size $\Theta(m\log n)$.  This makes some algorithms faster
and other slower, but since $m=O(n^2)$ the change is at most polynomial
for general graphs, so we can ignore it.

On the other hand, the presentation can make a lot of difference for
problems in Algebraic Combinatorics.  For example, a partition \ts
$\la=(4,3,1)\vdash 8$ \ts can be written in \defn{binary} \ts $(100,11,1)$ \ts
or in \defn{unary} \ts $(1111,111,1)$.  The binary presentation is more compact, so
partitions $\la=(\la_1,\ldots,\la_\ell)\vdash n$ with $\ell=O(1)$ have
size $O(\log n)$.

On the other hand, for many problems in Algebraic Combinatorics
the unary presentation of size $O(n)$ is more appropriate.
For example if your problem involves \defng{self-conjugate}
partitions $\la=\la'\vdash n$, we have $\ell(\la)\ge \sqrt{n}$,
so the binary presentation still requires poly$(n)$ space.\footnote{Occasionally,
people use \defng{Frobenius coordinates} \ts specifically to
avoid this issue for partitions with bounded \defng{Durfee size}. }
Similarly, the unary presentation is more natural when one works
with \defng{Young diagrams}.  In summary, every time the problem
involves partitions, one should \emph{always} state whether partitions
are given in unary or in binary.

\smallskip

\subsection{Classes and functions}\label{ss:basic-class}
The notion of a \defna{combinatorial class} can be viewed
as follows.  Define the \defn{language} as a subset of
binary words: $L\subseteq\{0,1\}^*$.  For example,
we can consider the language of words which encode all
Hamiltonian graphs.  Similarly, we can consider the
language of unary encodings of Young tableaux counted by the
LR~rule, i.e.\ for every triple of partitions $(\la,\mu,\nu)$,
there are exactly $c^\la_{\mu\nu}$ words in the corresponding
language.

We write $\overline{L} := \{0,1\}^*\setminus L$ to denote the
\defn{complement} of~$L$.  There is some ambiguity here depending
on the presentation, but in principle $\overline{L}$ can be very
large.  For example, the complement to the language of
Hamiltonian graphs includes both non-Hamiltonian graphs and all
words corresponding to non-graphs.

We usually think of a \defna{combinatorial function} as a function
counting combinatorial objects.  This can be viewed as special case
of the following general definition.  Let $f: \{0,1\}^\ast\to \zz$ an
integer function. Suppose $f(x) = 0$ for $x\notin L$, and $f(x)\ge 1$
for $x\in L$.  We then say that $f$ is \defn{supported on~$L$}.

For example, the number of Hamiltonian cycles is a function supported
on Hamiltonian graphs.  Similarly, the LR~coefficients is a function
\ts $f: (\la,\mu,\nu) \to c^\la_{\mu\nu}$ \ts supported on triples
of partitions where LR~coefficients are nonzero.

Note that the language can be defined in a variety of ways: by an
explicit function~$f$, by a Turing machine, by a formal grammar
by an explicit mathematical definition, or by an abstract logical
construction.  It is important to keep in mind that complexity of~$L$
can reflect complexity of its definition, but that does not always hold.

For example, the set of
exponents~$n$ for which the \defng{Fermat Last Theorem} (FLT) holds, naturally
correspond to a  language \ts $L=\{3,4,5,\ldots\}$.  From the
mathematical point of view, proving that $n\in L$ becomes increasingly
more difficult as $n$ grows, but as a language $L$ is rather simple now
that FLT has been proved.

\smallskip

\subsection{Decision, search and counting problems}\label{ss:basic-problems}
A \defn{decision problem} is a computational problem defined by
the membership in the language.  For example, \textsc{Hamiltonicity} is a
problem whether the language of Hamiltonian graphs contains a
word corresponding to given graph.  Similarly, the
non-vanishing of LR~coefficients problem \ts $c^\la_{\mu\nu}>^?0$
is a problem whether a triple of partitions $(\la,\mu,\nu)$ written in unary
is in the language given as a support of the LR~coefficients.

A \defn{search problems} is similar to the decision problem and asks
not only to decide $x\in^? L$, but also to \defng{verify} the answer by
providing a \defn{witness}.  We formalize this notion below, but for now
let us think of this qualitatively rather than quantitatively.

For example, the search problem associated with \textsc{Hamiltonicity}
ask to find a Hamiltonian cycle, since that implies that Hamiltonicity
by definition.  On the other hand, for \textsc{Non-Hamiltonicity}
there is no natural witness as nonexistence of a Hamiltonian cycle
cannot be easily characterized (for a very good reason, see below).
The verifier can simply list all $2^m$ subgraphs and supplant
a proof that none are a Hamiltonian cycle.

While seemingly harder, in some cases it is possible to use the decision
problem as a black box to solve the search problem, by applying
it repeatedly for smaller instances. For example, given graph~$G$,
if $G$ has a Hamiltonian cycle, check if so does $G\sm e$.
Continue removing edges until eventually some edge cannot
be removed.  Suppose now $G\sm e$ does not have Hamiltonian
cycles.  This means that $e$ is an edge in \emph{every} Hamiltonian
cycle in~$G$.  This reduced the problem to finding a Hamiltonian cycle
in~$G/e$ obtained from~$G$ as contraction by~$e$.  Proceed in this fashion
until the whole Hamiltonian cycle is constructed.

Given a search problem, a \defn{counting problem} is a problem of
computing a function~$f$ given by the number of witnesses the verifier
can accept. So if $x\notin L$, i.e.\ the decision problem
has a negative answer, the function $f(x):=0$.  Otherwise,
the function $f(x)\ge 1$ for all $x\in L$, i.e.\ function~$f$ is
supported on~$L$.

For example, the number of Hamiltonian cycles is a function supported
on Hamiltonian graphs which naturally arises that way.  Similarly,
but less obviously, the LR-coefficients $c^\la_{\mu\nu}$ is the counting
function for the number of LR tableaux.

\smallskip

\subsection{Polynomial time problems}\label{ss:basic-CC}
Until now we avoided using the words \defna{polynomial time}, since it
makes definitions quantitative and unnecessarily complicates the matter.
But we need it from this point.

The first truly important class for us is $\poly$.  It is a class of languages
where the decision problem can be solved in polynomial time.  There is a wide
variety of problems in this class, for example testing whether a graph is
connected or bipartite.  More involved graph theoretic problems in $\poly$
include planarity and having a perfect matching.\footnote{These follow
from the \emph{Kuratowski theorem} and the \emph{blossom algorithm}, resp.,
see e.g.~\cite[$\S$3.1,$\.\S$24.2]{Schr03}. }

Historically, there was a variety of ways to formalize the definition
of~$\poly$, all of which turn out to be equivalent.
We will use a \defng{Turing machine} (TM) mostly out of habit and because it is
best known (compared to \defng{RAM} and other equivalent models of computation).
From our point of view, using the colloquial \defn{polynomial time algorithm}
is absolutely fine.

We distinguish class $\poly$ from the class $\FP$ of nonnegative functions which can be
computed in polynomial time.\footnote{Some experts use a different definition
of $\FP$.  The one we use is more common in Counting Complexity. }
To remember the difference, note that the former
outputs $0$ or~$1$, while the latter can output larger numbers.  Simple examples
of functions in $\FP$ include the number of connected components of a graph,
the number of (proper) $2$-colorings, and the number of spanning
trees.\footnote{The latter follows from the \emph{matrix-tree theorem}.}

\smallskip

\subsection{Polynomial time verifier}\label{ss:basic-ver}
Now, given a language $L$, the \defn{polynomial time verifier} is a Turing machine~$M$
such that for some fixed polynomials $p,q$ we have:

\smallskip

$\circ$ \ts for all \ts $x,w\in \{0,1\}^\ast$, we have \ts $M(x,w)\in \{0,1\}$,

$\circ$ \ts for all \ts $x\in \{0,1\}^n$, machine $M$ runs in time $\le p(n)$,

$\circ$ \ts for all \ts $x\in L\cap \{0,1\}^n$, there exists $w$, s.t.\ \ts $|w|\le q(n)$ \ts and \ts $M(x,w)=1$,

$\circ$ \ts for all \ts $x\in \{0,1\}^n \sm L$ \ts and all $w$ s.t.\ \ts $|w|\le q(n)$, we have \ts $M(x,w)=0$.

\smallskip

\nin
In particular, the verifier \emph{accepts}, i.e.\ outputs~$1$, only if $w$ is a witness
for $x\in L$.
Note that the witness~$w$ have to have polynomial size to avoid the type of witnesses
we had seen in the Non-Hamiltonicity problem.  This constraint is also necessary for $M$
to work polynomial time, since otherwise it would take exponential time just to read
the exhaustive list of subgraphs in this case.

Continuing with our favorite example, in the \textsc{HamiltonianCycle} search problem,
the verifier checks if a collection of edges (this is~$w$) is a Hamiltonian cycle
in graph~$G$ (this is $x$ in the notation above).  Clearly, this can be done in
polynomial time.  Similarly, in the \textsc{LR~Tableau} search problem,
a Young tableau can be verified to be a LR~tableau corresponding to a triple
$(\la,\mu,\nu)$ in polynomial time.  This is done by checking all equalities
for the shape and the content, and all inequalities involved in the definition:
non-increase in rows, strict increase in columns, and balance
conditions for the right-to-left reading word.

\smallskip

\subsection{Complexity classes}\label{ss:basic-classes}
Complexity class $\NP$ is the class of decision problems $x \in^? L$
for which there exists a polynomial time verifier.
Similarly, class $\coNP$ is the class of decision problems
$x \in^? L$, such that there exists a polynomial time verifier
for the complementary problem $x \in^? \ov{L}$.

For example, \textsc{Hamiltonicity} $\in \NP$
and \textsc{Non-Hamiltonicity} $\in\coNP$.  Clearly, $\poly \subseteq \NP \cap \coNP$.
There are several hard decision problems known to be in $\NP \cap \coNP$,
so it is conjectured that $\poly \ne \NP \cap \coNP$.  It is also conjectured
that $\NP \ne \coNP$.  It is known that $\poly\ne \NP$ would not imply either
of these two conjectures (see e.g.~\cite[$\S$2.2.3]{Aa}).

Next, complexity class $\SP$ is the class of counting functions
for which there exists a polynomial time verifier.  Formally,
a function $f:\{0,1\}^\ast\to \nn$ is in $\SP$ if there exists a
polynomial time verifier~$M$ and  polynomial $q:\nn\to \nn$, such that for all $n\in \nn$  we have:
$$f(x) \, = \, \big|\big\{w\in \{0,1\}^{q(n)}:~M(x,w)=1\big\}\big| \quad \text{for all} \quad x\in \{0,1\}^n.
$$

Observe that \ts $\FP\subseteq \SP$.
Indeed, for $f\in \FP$ the witness for $x$
is any integer in \ts $a\in \big\{0,\ldots, f(x)-1\big\}$, and the verifier first
computes $f(x)$ and then checks if \ts $a< f(x)$.
It is widely assumed that $\FP\ne\SP$. In fact, it is hard to overstate how
strong is this assumption.

For example, let $f(G)$ be the number of
Hamiltonian cycles in~$G$.  Then $f\in \SP$ since it is
counting combinatorial objects which can be verified
in polynomial time. Similarly, LR~coefficients are given as a
function \ts $(\la,\mu,\nu) \to c^\la_{\mu\nu}$ \ts is also in~$\SP$,
since it is counting LR-tableaux which can be verified in
polynomial time by the argument above.  We give many more examples
in the next section.

Note that we do not discuss problems that are $\NP$-complete or
$\SP$-complete.  That's largely because these notions are largely
tangential to this survey.  Like with other complexity
classes and standard computational complexity notation, we will mention
them at will when we need them and hope the reader catches us.
Here is a partial list as a mental check for the reader:
$$
\poly \, \subseteq \,\UP\, \subseteq \, \NP \, \subseteq \,  \Sigma_2^{\textsc{p}} \, \subseteq \, \PH\, \subseteq \, \PSPACE\ts.
$$

We do want to emphasize the distinction of
$\NP$-complete and $\NP$-hard classes -- the former
is contained in~$\NP$, while the latter does not.  The same
with $\SP$-complete and $\SP$-hard classes.

\bigskip

\section{Combinatorial interpretation, first steps}\label{s:Main}

\subsection{Main definition} \label{ss:Main-Def}

We will be brief.  Let $f:\{0,1\}^\ast \to \nn$
be a function.  We say that $f$ has a
\defn{combinatorial interpretation} \. if
{\large
$${f \. \in \. \SP} \ts.
$$}
%
%\noindent

% The reader who is uncertain about the definition of $\SP$
% might want to reread it.
Note that until now, we used the term
``combinatorial interpretation'' in both its technical and colloquial
meaning, which usually coincide but can also differ in several special
case.  For the rest of the paper, we will use it only in the
technical sense, and use quotation marks for the colloquial meaning.

% In this and next several sections we illustrate this
% definition with many examples and non-examples.  Then, in Section~?,
% we discuss why this definition makes sense from a combinatorial point
% of view.

\smallskip

\subsection{Basic examples and non-examples} \label{ss:Main-examples}
We begin with some motivating examples, mostly following~\cite{IP22}.

\medskip

\nin
{\small $(1)$} \, Let \ts $e: P\to \nn$ \ts be the number of linear extensions of~$P$,
where \ts $P=(X,\prec)$ \ts is a poset with $n$ elements.  Clearly, $e(P)\ge 1$,
so we can define a nonnegative function \ts $e'(P): = e(P)-1$.
Now observe that \ts $e'\in \SP$ \ts simply because finding the lex-smallest
linear extension $L$ can be done in polynomial time by a greedy algorithm
(see e.g.~\cite{CW95}),  so \ts $e'(P)$ \ts counts linear extensions of~$P$
that are different from~$L$.

\medskip

\nin
{\small $(2)$} \, Let \ts $G=(V,E)$ \ts be a simple graph with $n=|V|$ vertices
and $m=|E|\ge 1$ edges.
Let \ts $c(G)$ \ts be the number of proper $3$-colorings of~$G$.
Clearly, $c\in \SP$.  Note that $3^n-c(G)$ is also in~$\SP$, since
verifying that a $3$-coloring is \emph{not} proper is in~$\poly$.

Now, taking into account permutations of colors, observe that \ts $f(G):=c(G)/6$
is an integer valued function.  To see that $f(G) \in \SP$, note that
of the six possible $3$-colorings corresponding to a given $3$-coloring one
can easily choose the lex-smallest.  In other words the combinatorial interpretation
for $f(G)$ is the set of lex-smallest $3$-colorings of~$G$.

The key point here is that starting with a $3$-coloring $\chi:V\to \{1,2,3\}$,
we can compute in polynomial time the lex-smallest $3$-coloring $\chi'$ from
the set of 6 recolorings of~$\chi$.  If $\chi=\chi'$, we verify that $\chi$
is a combinatorial interpretation, and discard~$\chi$ if otherwise.

\medskip

\nin
{\small $(3)$} \, Let $G=(V,E)$ be a simple graph with
$|V|=n$  vertices and $|E|=m$  edges.  Consider the
following elegant inequality by Grimmett~\cite{Gri76}:
$$(\ast) \qquad
\tau(G) \, \le \, \frac{1}{n} \left(\frac{2\ts m}{n-1}\right)^{n-1}
$$
for the number $\tau(G)$ of \defng{spanning trees} in~$G$.\footnote{The
original proof is a nice two line argument using the AM-GM inequality
for the product of eigenvalues of the \defng{Laplacian matrix} of~$G$.
One could argue whether this proof ``combinatorial'', but it definitely
does not extend to an explicit injection. }
We can turn this inequality into a nonnegative integer function as follows.
$$f(G) \, := \, (2\ts m)^{n-1} \. - \. n \ts (n-1)^{n-1} \ts \tau(G)\ts.
$$
Recall \ts $\tau(G)\in \FP$ \ts by the \defn{matrix-tree theorem}.
Then \ts $f\in \FP$, and so \ts $f(G)$ \ts  has a combinatorial
interpretation according to our definition.

One could argue that a ``combinatorial interpretation'' should explain \defna{why} \ts
the inequality~$(\ast)$ holds in the first place.  In fact, there are several schools
of thought on this issue (see a discussion in~\cite[$\S$4]{Pak18}).  We believe
that the computational complexity approach is both the least restrictive and
the most formal way to address this.  Indeed, the combinatorial interpretations
we study are depend solely of the functions themselves and not of the difficulty
of the proof of the functions being integer or nonnegative.

\medskip

\nin
{\small $(4)$} \, Let $h(G)$ be the number of Hamiltonian cycles in~$G$,
and let \ts $f(G) := \big(h(G)-1\big)^2$.  This is our most basic non-example.
While we cannot prove unconditionally that $f\notin \SP$, we can prove it modulo
standard complexity assumptions.  Intuitively this is relatively straightforward.
Clearly, a poly-time verifier that \ts $f(H)\ne 0$ is also a poly-time verifier
that \ts $h(G)\ne 1$.  A poly-time verifier for \ts $h(G)\ge 2$ \ts is easy: present
two distinct Hamiltonian cycles.  On the other hand, a poly-time verifier for
\ts $h(G)= 0$ \ts is unlikely since that would imply that $\NP=\coNP$.\footnote{This is
because \textsc{NonHamiltonicity} is $\coNP$-complete and \cite[Prop.~10.2]{Pap}.}

\medskip

\nin
{\small $(5)$} \, As above, let $h(G)$ be the number of Hamiltonian cycles in~$G$.
Recall \defn{Fermat's little theorem} \ts
states that \. $p\.| \. a^{p}-a$ \.  for all integers~$a$,
and prime~$p$.\footnote{Fermat stated this result in~1640 without proof, and the
first published proof was given by Euler in~1736. According to Dickson,
``this is one of the fundamental theorems of the theory of numbers''
\cite[p.~V]{Dic52}.}  Let
$$f(G) \, := \, \tfrac{1}{p} \bigl(h(G)^p \ts - \ts h(G)\bigr).
$$
It was shown in \cite[Prop.~7.3.1]{IP22}, that \ts $f(G)\in \SP$. The proof
is very short, and a variation on the original proof in~\cite{Pet72} (see also \cite{Ges84,Gol56}).
We reproduce it here in full.

\smallskip

\nin
{\em Proof.} Consider sequences \ts $(a_1,\ldots,a_p)$ \ts of integers \ts $1\le a_i \le h(G)$ \ts
and partition them into orbits under the natural cyclic action of \ts $\zz/p\zz$.
Since $p$ is prime, these orbits have either $1$ or~$p$ elements.  There
are exactly $p$ orbits with one elements, where $a_1=\ldots=a_p$.  The remaining
orbits of size~$p$ have a total of \ts $h(G)^p-h(G)$ \ts elements.  Since $p$ is fixed,
the lex-smallest orbit representative can be found in poly-time. %\hfill
\ $\sq$

\medskip

\nin
{\small $(6)$} \,
Recall the following \defn{Smith's theorem}~\cite{Tut46}.
Let $e=(v,w)$ be an edge in a cubic graph~$G$. Then the number $N_e(G)$ of
Hamiltonian cycles in $G$ containing~$e$ is always even.  Denote
\ts $f(G,e):=N_e(G)/2$.  Is \ts $f\in \SP$?  We don't know.
This seems unlikely and remains out of reach with existing technology.
But let us discuss the context behind this problem.

There are two main proofs of Smith's theorem.  Tutte's original proof
in~\cite{Tut46} uses a double counting argument.  An algorithmic version
of this proof is given by Jensen~\cite{Jen12}.
The algorithm starts with one Hamiltonian cycle in~$G$ containing~$e$,
and finds another such cycle. Jensen also shows that this algorithm
requires an exponential number of steps in the worst case.

The Price--Thomason \defng{lollipop algorithm}~\cite{Pri77,Tho78} gives
a more direct combinatorial proof of Smith's theorem.  This algorithm
also partitions the set of all Hamiltonian cycles in~$G$ containing~$e$ into
pairs,\footnote{Finding another Hamiltonian cycle was first raised in~\cite{CP88}
in the context of Smith's theorem.  This was a motivational problem for the
complexity class \ts $\PPA$, see~\cite{Pap94a}, as well as large part of our
work in~\cite{IP22}.  Whether it is $\PPA$-complete remains open.}
and is also exponential, see \cite{Cam01,Kra99}.\footnote{There
are other results similar to Smith's theorem which can be proved by
a parity argument by a variation of the lollipop algorithm,
see e.g.~\cite{CE99} and references therein.}

\begin{conj}\label{conj:Main-Smith}
The function \ts $f(G,e)$ \ts is not in $\SPr$.
\end{conj}

Note that if either Jessen's algorithm or the lollipop algorithm were
poly-time, this would imply that $f\in \SP$.  Indeed, by analogy
with~{\small $(2)$}, a poly-time algorithm would allow us to search
for Hamiltonian cycles and only count the ones that are lex-smaller
than their pairing partner.

\smallskip

\subsection{First observations}  \label{ss:Main-obs}
From the limited number of examples above, here are a few observation.
We will develop them further later on.

\medskip

\nin
{\small $(i)$} \.  In combinatorics, nonnegative integer functions
don't come from nowhere.  They are either already counting something,
e.g.\ orbits under the action of some group as in {\small $(2)$}
and~{\small $(5)$}, or are byproducts of inequalities as in~{\small $(1)$},
{\small $(3)$} and {\small $(4)$}.

\medskip

\nin
{\small $(ii)$} \.
The inequality in {\small $(i)$} could be rather trivial.  For example,
we have the trivial inequality \ts $e(P)\ge 1$ in~{\small $(1)$}, the
\defng{AM-GM inequality} in~{\small $(3)$}, and \ts $x^2\ge 0$ in~{\small $(4)$}.
It is the nature of the inequality that determines whether the function is in~$\SP$.

\medskip

\nin
{\small $(iii)$}  \.  The computational hardness of the functions works
only in one direction:  if \ts $f\in \FP$ \ts then \ts $f\in \SP$, see~{\small $(3)$},
but if \ts $f\in \SP$-hard then it can go both ways.

\medskip

\nin
{\small $(iv)$}  \. Even for some classical problems like~{\small $(6)$},
membership in $\SP$ can be open.

\bigskip

\section{Sequences} \label{s:seq}

The problems in this section come from Enumerative Combinatorics.
Although they are not the most interesting questions from the complexity
point of view, the problems of finding combinatorial interpretations
of integer sequences are much too famous not be addressed.  In our
notation and problem selection we largely follow~\cite{Pak18}.

Throughout this section we assume that the input $n$ is in unary.
We say that an integer sequence $\{a_n\}$ has a combinatorial interpretation
if a function $f: n\to a_n$ is in $\SP$.  Similarly,
we say that $\{a_n\}$ can be computed in poly-time if $f\in \FP$.  By abuse of
notation, we also say that $\{a_n\}$ is in $\SP$ and $\FP$, respectively.

\smallskip

\subsection{Catalan numbers}\label{ss:seq-Cat}
Recall the \defn{Catalan numbers}
$$\text{Cat}(n) \ = \ \frac{1}{n+1}\binom{2n}{n} \ = \ \binom{2n}{n} \. -\.\binom{2n}{n-1}\.,
$$ \ts
see e.g.\ \cite[\href{https://oeis.org/A000108}{A000108}]{OEIS}.
The fractional formula implies that \ts Cat$(n)>0$,
the subtraction formula implies that \ts Cat$(n)\in \zz$, but a priori
it is not immediately obvious that Catalan numbers have \emph{any} \ts combinatorial
interpretations. Of course, there are over 200 ``combinatorial
interpretations'' of various types given in \cite{Sta15}.

Let us show that \ts $\big\{\text{Cat}(n)\big\}\in \SP$.
Recall that \ts $\text{Cat}(n)$ \ts is equal to the number of \defn{ballot sequences},
defined as  \ts $0-1$ \ts sequences with $n$ zeros and $n$~ones,
s.t.\ every prefix has at least as many zeroes as ones.
% $\#$1's in the first $k$ steps is $\le \#$0's in the first $k$ steps,
%for all $1\le k \le 2n-1$.
This can be checked in time poly$(n)$, which proves
that Catalan numbers are in~$\SP$.\footnote{This is because $n$ is in unary.
Note that if $n$ is binary, the ballot sequences have exponential length.}

In fact, just about all ``combinatorial interpretations'' in \cite{Sta15} can
also be used to show that Catalan numbers are in~$\SP$, but some are trickier than others.
For example, $\text{Cat}(n)$ \ts is the number of $123$-avoiding permutations
in~$S_n$, and one would need to observe that there are \ts $\binom{n}{3}=O(n^3)$
possible $3$-subsequences. Thus, verifier checking the $123$-avoidance is in~$\poly$,
as desired.

On the other hand, \emph{some} ``combinatorial interpretations'' in \cite{Sta15}
are not even counting combinatorial objects, and it can take some effort
to give them an equivalent presentation which is in~$\SP$.
Notably, Exc~195 counts certain regions in $\rr^n$ in the complement
to the \defng{Catalan hyperplane arrangement} \ts $\cC_n$\ts.  This
setting raises some interesting computational questions.

To present the regions as a combinatorial objects one can
use a collection of signs, one for each hyperplane. Since the
arrangement \ts $\cC_n$ \ts has $\Theta(n^2)$ hyperplanes, a trivial consistency check
of all $(n+1)$-tuples of hyperplanes would give an exponential
time algorithm for testing whether the resulting region is nonempty.
This is not good enough for being in~$\SP$.

Now, in this specific case of the Catalan arrangement,
there is an easy poly-time testing algorithm which uses the simple
structures of hyperplanes in \ts $\cC_n$ \ts and avoids the redundancy
in the exhaustive testing above.\footnote{In fact, the problem of counting the
number of regions in the complement of general rational hyperplane
arrangements is in~$\SP$.  Indeed, one can use standard results
in Linear Programming to give a poly-time verifier for all regions
encoded by \emph{subsets} of the set of halfspaces defined by the
hyperplanes.  This implies that counting regions problem is always in~$\SP$.
We thank Tim Chow for this observation, see \ts
\href{https://mathoverflow.net/a/428272/4040}{mathoverflow.net/a/428272}}
This algorithm is the verifier giving the desired combinatorial
interpretation.\footnote{The bijection
in the solution of Exc~195 in~\cite{Sta15}
(which requires a proof!) is another approach to have these regions
are in bijection with combinatorial objects. }

\smallskip

\subsection{Polynomial time computable combinatorial sequences}\label{ss:seq-poly}
Note that since~$n$ is in unary, the sequence \ts $\big\{\text{Cat}(n)\big\}$ is in $\FP$
since it can be computed in polynomial time.  The same holds for \defng{Fibonacci
numbers} \cite[\href{https://oeis.org/A000045}{A000045}]{OEIS},
\defng{numbers of involutions} \cite[\href{https://oeis.org/A000085}{A000085}]{OEIS},
\defng{partition numbers} \cite[\href{https://oeis.org/A000110}{A000110}]{OEIS},
and myriad other sequences which can be computed via recurrence relation.

Formally, we observe in~\cite[Prop.~2.2]{Pak18} that every \defng{D-algebraic
sequence} \ts is in~$\FP$ when the input $n$ is in unary.  Thus, in particular,
this holds for all \defng{algebraic}  \ts and \defng{P-recursive sequences},
see e.g.\ \cite[Ch.~6]{Sta99}.

On the other hand, there are sequences which likely cannot be computed
in poly$(n)$ time.  For example, the number \ts SAW$(n)$ \ts of
\defng{self-avoiding walks} \ts of length~$n$ in $\zz^2$ starting at the origin,
is conjectured not to be in~$\FP$ \cite[Conj.~2.14]{Pak18}. Clearly, \ts
$\bigl\{\text{SAW}(n)\bigr\}\in \SP$,
so we turn our attention to sequences which are \emph{unlikely} to be in~$\SP$,
 or are in~$\SP$ for less obvious reasons.

\smallskip

\subsection{Unimodality and log-concavity}\label{ss:seq-unimod}
Both \defng{unimodality} and \defng{log-concavity} properties of
combinatorial sequences are heavily studied in the literature,
see e.g.\ \cite{Bra15} (see also \cite{Bre89,Bre94,Sta89} for
more dated surveys). Following~\cite{Pak19}, every time you
have an inequality $X\le Y$, we can convert it into a
nonnegative integer $(Y-X)$ and ask if it has a
combinatorial interpretation. For combinatorial sequences
this is especially notable, and the approach above works
well again.

To see explicit examples, recall multi-parameter combinatorial sequences such as
\defng{binomial coefficients} \ts $\binom{n}{k}$\.,
\defng{Delannoy numbers} \ts $D(i,j)$ \ts
\cite[\href{https://oeis.org/A008288}{A008288}]{OEIS},
\defng{Stirling numbers} \ts of both kinds
\cite[\href{https://oeis.org/A008275}{A008275}]{OEIS} and
\cite[\href{https://oeis.org/A008277}{A008277}]{OEIS},
\defng{$q$-binomial coefficients} \ts $\binom{n}{k}_q$
(see e.g.\ \cite[$\S$1.7]{Sta99}), etc.  All of these
satisfy various unimodality and log-concavity properties, e.g.\
$$\aligned
& \tbinom{n}{k-1} \tbinom{n}{k+1} \, \le \, \tbinom{n}{k}^2\., \\
& D(i,j)\, \le \, D(i+1,j-1) \quad \text{for all} \quad i<j, \quad \text{and}\\
&
[q^{m-1}] \ts \tbinom{n}{k}_q \, \le \,  [q^m] \ts \tbinom{n}{k}_q
 \quad \text{for all} \quad 0< m \le \tfrac{k(n-k)}{2}\..
\endaligned
$$
We refer to \cite{Sag92,CPP21b} for the first two of these inequalities both of which
have a direct injective proof.  The last inequality is due to Sylvester~\cite{Syl},
see also \cite{PP13,Pro82,Sta89} for modern treatment.
Clearly, each of these inequalities has a combinatorial interpretation simply
because both sides are in~$\FP$.  For example, \ts
$\bigl\{\tbinom{n}{k}^2-\tbinom{n}{k-1} \tbinom{n}{k+1}\bigr\} \in \FP$, etc.

\smallskip

\subsection{Partitions}\label{ss:seq-part}
\defn{Ramanujan's congruence} \. $p(5n-1) \equiv 0  \pmod 5$ \ts has a famous
``combinatorial interpretation'' by Dyson, who conjectured (among other things)
that \ts $\tfrac15 \. p(5n-1)$ \ts is equal to the number of partitions
$\la \vdash (5n-1)$ with \defn{rank} $\la_1-\la_1' = 0 \pmod 5$, see \cite{Dys44}.
This conjecture was proved in~\cite{AS54} and later extended in a series of
remarkable papers, see \cite{AG88,GKS90,Mah05}.

Now, Ramanujan proved many more congruences such as
$p(25n-1)\equiv 0  \pmod {25}$, see e.g. \cite[$\S$6.6]{Har40},
but there seem to be no Dyson-style rank statistics in this case.
On the other hand, now that the congruence is known, it follows
that \ts $\bigl\{\tfrac1{25} \. p(25n-1)\bigr\} \in \FP$.  This
is because
$$
\sum_{n=0}^{\infty} \. p(n) \ts t^n \, = \, \prod_{i=1}^{\infty} \.
\frac{1}{1-t^i}
$$
is D-algebraic, or because $\{p(n)\}$ can be computed in poly$(n)$ time
via \defng{Euler's recurrence} (among several other ways),
see \cite[$\S$2.5]{Pak18} and references therein.  This implies
that \ts $\bigl\{\tfrac1{25} \. p(25n-1)\bigr\}$ \ts \emph{already has} \ts
a combinatorial interpretation.

Similarly, the curious inequality \.
$p_{1 \ts 4}(n) \ge p_{2 \ts 3}(n)$ \. for the numbers of partitions of~$n$
into parts $\pm 1 \pmod 5$ and $\pm 2 \pmod 5$, respectively.  Finding an explicit
injection proving the inequality was suggested by Ehrenpreis, see \cite{AB89,Kad99}.
From the computational complexity point of view, we already have \.
$\{p_{1 \ts 4}(n) - p_{2 \ts 3}(n)\} \in \FP$,
which shows that the desired injection can be computed in poly-time.\footnote{This is
similar and partially motivated by the discussion of complexity of partition
bijections viewed as algorithms, see~\cite[$\S$6.1]{KP}
and \cite[$\S$8.4.5]{Pak06}.}

Finally, the log-concavity of the partition function \cite{DP15},
implies that the sequence \ts $\big\{p(n)^2-p(n-1)\ts p(n+1)\,,\,n>25\big\}$
\ts is in $\SP$, simply because \ts $\{p(n)\}\in \FP$.

\smallskip

\subsection{Unlabeled graphs}\label{ss:seq-isom}
Let $u_n$ be the number of non-isomorphic \defn{unlabeled graphs} on~$n$
vertices, see  \cite[\href{https://oeis.org/A000088}{A000088}]{OEIS}.
Wilf conjectured in \cite{Wilf}, that $\{u_n\}$ \emph{cannot} be computed
in poly$(n)$ time, see also~\cite[Conj.~1.1]{Pak18}.  Does $\{u_n\}$ have a
combinatorial interpretation?  This is not so clear.  The difficulty is
that we are counting orbits rather than combinatorial objects and there
is no obvious way to choose orbit representatives:

\begin{op} \label{op:seq-graphs}
\ts The sequence \ts $\{u_n\}$ \ts is in \ts $\SPr$.
\end{op}

To understand the context of this problem, consider a closely related sequence.
Let $a_n$ be the number of nonisomorphic \emph{unlabeled plane triangulations}
on~$n$ vertices, see \cite[\href{https://oeis.org/A000109}{A000109}]{OEIS}.
In \cite[Conj.~1.3]{Pak18}, we conjectured that $\{a_n\}$ be computed
in poly$(n)$ time.  This would immediately imply that $\{a_n\}$  is in~$\FP$
and thus in~$\SP$.  Since the conjecture remains open, we show the latter directly:

\begin{prop}\label{p:seq-tri}
The sequence \ts $\{a_n\}$ \ts is in \ts $\SPr$.
\end{prop}

We postpone the proof until~$\S$\ref{ss:proofs-Prop-tri}. The idea that the
group of automorphisms of triangulations has polynomial size and all
automorphism can be computed explicitly via a poly-time algorithm for
the isomorphism of planar graphs.  We are able to compute the whole
orbit and then use symmetry breaking by taking lex-smallest orbit representative.

\begin{conj}[{\rm \cite[Conj.~1.3]{Pak18}}{}]\label{conj:seq-tri}
The sequence \ts $\{a_n\}$ \ts is in \ts $\FPr$.
\end{conj}

We believe that this conjecture can be derived using the
tools in~\cite{Fusy05,KS18}.
Back to the sequence $\{u_n\}$.  If \textsc{GraphIsomorphism} was known
to be in $\poly$,\footnote{Formally, we need an effective version
\textsc{GraphIsomorphism}, which produces generators for $\Aut(G)$
as a subgroup of~$S_n$.  This is known in many cases and related to
the notion of \defn{canonical labeling}, see \cite{Bab19,BL83,SW19}.}
one could try to use the
symmetry breaking approach in the proof of Proposition~\ref{p:seq-tri}.
Babai's recent \defng{quasipolynomial}
upper bound \ts $n^{O((\log n)^c)}$ \ts on graph isomorphism \cite{Bab18},
falls short of what we need towards resolving Open Problem~\ref{op:seq-graphs}.

Note that plane triangulations are dual to $3$-connected cubic graphs,
so the following problem lies in between Proposition~\ref{p:seq-tri}
and Open Problem~\ref{op:seq-graphs}.

\begin{conj}\label{conj:seq-reg}
Let \ts $R_k(n)$ \ts be the number of \ts $k$-regular unlabeled graphs
on \ts $n$ \ts vertices. Then \ts $\big\{R_k(n)\big\}$ \ts is in \ts
$\SPr$, for all \ts $k\ge 1$.
\end{conj}

We are  optimistic about this conjecture since for $k$-regular
graphs the \textsc{GraphIsomorphism} problem is in~$\poly$.  This was
proved by Luks in~\cite{Luks82}, see also~\cite{BL83,SW19}.

\smallskip

Finally, there is a curious connection to log-concavity
(see~$\S$\ref{ss:seq-unimod}).  Denote by \ts $u_n(m)$ \ts the number
of nonisomorphic graphs with \ts $n$ \ts vertices and \ts $m$ \ts edges.
It follows from~\cite{PR86} (see also~\cite{Vat18}), that \ts
$u_n(m)^2\ge u_n(m-1)\ts u_n(m+1)$.  If \. $\{u_n(m)\}\in \SP$
(see Open Problem~\ref{op:seq-graphs}), it would make sese to ask
if we also have \ts $\big\{u_n(m-1)\ts u_n(m+1)-u_n(m)^2\big\}\in \SP$.
Analogous questions can be asked about non-isomorphic planar graphs,
plane triangulations, etc.

\smallskip

\subsection{Knots}\label{ss:seq-knots}
Denote by \ts $k_n$ \ts the number of
distinct knots with bridge number at most~$n$, see e.g.\
\cite[\href{https://oeis.org/A086825}{A086825}]{OEIS}.
Here the \defn{bridge number} is a knot invariant defined
as the minimal number of bridges required to draw a knot
in the plane, see e.g.~\cite[$\S$4.3]{Mur96}.

\begin{op} \label{op:seq-knots}
The sequence \ts $\{k_n\}$ \ts is not in \ts $\SPr$.
\end{op}

To underscore combinatorial nature of the problem, note that
knot diagrams are a (subset of) planar $4$-regular graphs with
signs at the vertices, so $n$ is the bound on the number of vertices.
The difficulty starts with the word ``distinct'' which is formalized
as \defng{non-isotopic} and is also combinatorial in nature:
two knots are \defn{isotopic} if they are connected by a finite
sequence of \defng{Reidemeister moves}.  Unfortunately,
from computational point of view, the issue with identifying
distinct knots is much deeper than with nonisomorphic graphs.

First, note that it is not at all obvious that the isotopy is
decidable.  Could it be that the number of necessary
Reidemeister moves between two isotopic
knots with~$n$ crossings grows faster than the busy beaver
function?  The answer turns out to be ``No''; the sequence
$\{k_n\}$ is computable indeed.  The best known upper bound on the number
of Reidemeister moves is the tower of twos of height $2^{O(n)}$
is given by Coward and Lackenby~\cite{CL14}.\footnote{There are also
various hardness results suggesting that such sequence is hard to compute,
see e.g.~\cite{dMRST,Lac17,KT21}. }

We conclude with a simpler problem, or at least the one that has been resolved.
Denote by \ts $a_n$ \ts the number of knot diagrams on $n$ labeled crossings
which are isotopic to the \defng{unknot}.  The fact that $\{a_n\}$ is
in~$\SP$ follows from a famous result by Hass, Lagarias and Pippenger \cite{HLP99}.\footnote{It is
also an immediate corollary from~\cite{Lac15}, which shows that unknot
can be obtained by a sequence of $O(n^{11})$ Reidemeister moves.}
Similarly, denote by \ts $b_n$ \ts the number of knot diagrams on $n$
labeled crossings which are \emph{not} isotopic to the \defng{unknot}.
The sequence $\{b_n\}$ is also in $\SP$ by a recent result of Agol,
see~\cite[$\S$3.5]{Lac17}.

\bigskip

\section{Subgraphs} \label{s:sub}

Discrete Probability is a major source of combinatorial inequalities,
most of which can be converted into nonnegative functions.  Whether
these functions are in $\SP$ is then a challenging problem.
In this section we concentrate on various counting subgraphs
problems.

\smallskip

\subsection{Matchings}\label{ss:sub-match}
Let  $G=(V,E)$ be a simple graph, and let $p(G,k)$ denote the number of
\defng{$k$-matchings} in a simple graph $G=(V,E)$ defined as the number
of $k$-subsets of~$E$ of pairwise nonadjacent edges.  Clearly,
$p(G,k)\in \SP$.  Following~\cite{Pak19}, consider a function \ts
$$
f(G,k) \, := \, p(G,k)^2 \. - \.  p(G,k-1) \. p(G,k+1).
$$
Famously, Heilmann and Lieb proved that $f(G,k)\ge 0$ \cite{HL72},
see also \cite[$\S$6.3]{God93} and~\cite{MSS15} for more context
on this remarkable result.  It was observed in~\cite{Pak19} that
$f\in \SP$ follows immediately from Krattenthaler's injective
proof of the Heilmann--Lieb theorem~\cite{Kra96}.

Define \ts $q(G)$ denote the number of spanning subgraphs \ts $H=(V,E')$,
\ts $E'\subseteq E$, which contain a perfect matching.
Observe that the function \ts $q\in\SP$, since testing whether~$H$
has perfect matching is in~$\poly$, see e.g.\ \cite[$\S$9.1]{LP86}.
The following subsection shows that this is unlikely for other
graph properties.

\smallskip

\subsection{Hamiltonian subgraphs}\label{ss:sub-Ham}
Let $f(G)$ denote the number of Hamiltonian spanning subgraphs
of a simple graph $G=(V,E)$.  Whether $f\in^?\SP$ is a difficult
question and does not follow directly from the definition since
we need a poly-time algorithm to decide Hamiltonicity of~$G$.\footnote{This
is another example where a combinatorialist might disagree, since
the definition $f(G)$ already gives a \ts \emph{kind of} \ts ``combinatorial interpretation''.
% On the other hand, \textsc{Hamiltonicity} is famously $\NP$-complete,
% so we are not plowing new ground here.
}

Note that this is a close call, since there is an algorithm to
\emph{verify} that each $H$ is Hamiltonian by showing a Hamiltonian
cycle~$C$ in~$H$. Thus, one would think that pairs $(H,C)$ give a
combinatorial interpretation of $f$, but of course one would
need to pick only one such cycle $C$ per~$H$.  For example,
the lex-smallest $C$ would work, but there is no poly-time
algorithm to verify that.

\begin{op}\label{op:sub-Ham}
Function \ts $f$ \ts is not in \ts $\SPr$.\footnote{Here and all other open problems
and conjectures in this paper we implicitly allow the use of any of the
standard complexity assumptions.  Otherwise, these open problems are
both deeper and less approachable.
}
\end{op}

Even more difficult is the \ts $\ov{f}(G):=2^m-f(G)$ \ts function
which counts \emph{non-Hamiltonian} spanning subgraphs of~$G$, since there
is no efficient verifier in this case.  That makes the following
problem a little more approachable, perhaps:

\begin{conj}\label{conj:sub-Ham-non}
Function \ts $\ov{f}$ \ts is not in \ts $\SPr$.
\end{conj}

\smallskip

\subsection{Spanning forests}\label{ss:sub-forests}
Let $G=(V,E)$ be a simple connected graph with $n=|V|$ vertices, and let
$F(G,k)$ denote the number of spanning forests in~$G$ with~$k$ edges.
A special case of the celebrated result by Adiprasito, Huh and Katz \cite{AHK},
proves log-concavity of \ts $\{F(G,k)\}$:
$$(\star) \qquad
F(G,k)^2 \, \ge \, F(G,k-1)\cdot F(G,k+1) \quad \text{for all} \quad 1\le k \le n-2.
$$
Following~\cite{Pak19}, define \ts $f(G,k):= F(G,k)^2 - F(G,k-1)\ts F(G,k+1)$.

\begin{conj}\label{conj:sub-forest}
Function \ts $f$ \ts is not in \ts $\SPr$.
\end{conj}

We have relatively little evidence in favor of this conjecture other
than we tried very hard and failed to show that $f\in \SP$.
The original proof was significantly strengthened and
simplified in~\cite{ALOV,BH20,CP21}
(see also \cite{CP22} for a friendly exposition).\footnote{In~\cite[p.~314]{Sta00},
Stanley writes about $(\star)$: ``\emph{Our own feeling is that these questions
have negative answers, but that the counterexamples will be huge and
difficult to construct.}''  We think of this quote as a suggestion that
there is no direct combinatorial proof of~$(\star)$, pointing
in favor of Conjecture~\ref{conj:sub-forest}.}

\smallskip

\subsection{Perfect matchings}\label{ss:sub-PM}
Let \ts $G=(V,E)$ \ts be a $k$-regular bipartite multigraph on $2n$ vertices,
and let $\PM(G)$ be the number of perfect matchings in~$G$.
The celebrated \defn{van der Waerden Conjecture}, now proved
(see e.g.\ \cite{vL82} and \cite[$\S$8.1]{LP86}), is equivalent to
$$
\PM(G) \. \ge \. \frac{k^n \ts n!}{n^n}\..
$$
Let \. $f(G):= n^n \ts \PM(G) - k^n  n!$ \. The following result
is a variation on \cite[Thm.~7.1.5]{IP22}, and shows that it
is unlikely that $f$ has a combinatorial interpretation.

\smallskip

\begin{prop}\label{prop:sub-PM}
Assume that edge multiplicities in graph~$G$ are given by \ts $\SPr$ \ts functions. \\
If \ts $f\in \SPr$, then \ts $\PHr=\Sigma_2^{\textsc{p}}$.
\end{prop}

\begin{proof}
Let \ts $n=2$, \ts $V=\{a_1,a_2,b_1,b_2\}$, and let $E$ consists
of edges $(a_1,b_1)$ and $(a_2,b_2)$ with multiplicity~$x_1$,
edges $(a_1,b_2)$ and $(a_2,b_1)$ with multiplicity~$x_2$.
Then $G=(V,E)$ is bipartite and $k$-regular, where \ts $k=x_1+x_2$.
We have \ts $\PM(G)=x_1^2+x_2^2$ \ts and
$$
f(G) \. = \. 4\ts \PM(G) \. - \. 2 \ts k^2 \. = \.  4\ts (x_1^2+x_2^2) \. - \. 2\ts (x_1+x_2)^2
\. = \.  2\ts (x_1-x_2)^2\ts.
$$
The result now follows from Corollary~2.3.2 in~\cite{IP22}.
\end{proof}

Compare this result with \defn{Schrijver's inequality}~\cite{Schr98}
$$
\PM(G) \. \ge \. \left(\frac{(k+1)^{k-1}}{k^{k-2}}\right)^n\.,
$$
for all \. $n\ge k \ge 3$.  An elementary proof of the $k=3$ case
is given in~\cite{Voo}. We challenge the reader to give a direct
combinatorial proof of this inequality for any fixed~$k>3$.

\smallskip

Finally, let us mention \defn{Bregman's inequality} \cite{Bre73}
formerly known as \defng{Minc's conjecture} (see also \cite[$\S$6.2]{Minc78}).
In the special case of $k$-regular bipartite
simple graphs, the setting of the former \defng{Ryser's conjecture}, it gives \.
$\PM(G)^k \le (k!)^{n}$.  Since the proof in this case is relatively
short, it would be interesting to see if this inequality is in \ts $\SP$.

\smallskip

\subsection{Bunkbed conjecture}\label{ss:sub-bunkbed}
Let $G=(V,E)$ be a multigraph.  Denote by $G\times K_2$ the \defn{bunkbed graph}
obtained as a Cartesian product. Formally, two copies of $G$ are connected by parallel edges
as follows:  each vertex $v\in V$ corresponds to vertices $v_0=(v,0)$ and $v_1=(v,1)$
which form an edge $(v_0,v_1)$.

For vertices $v,w\in V$, denote by \ts $B_0(v,w)$ \ts and \ts $B_1(v,w)$ \ts the number of spanning
subgraphs~$H$ of $G\times K_2$, such that \. $v_0\lra_H w_0$ \. and  \.
$v_0\lra_H w_1\ts$, respectively.
In other words, we are counting subgraphs where $w_0$ or~$w_1$ lie in the same
connected component as~$v_0$.

\begin{conj}[{\rm \defn{Bunkbed conjecture}}{}]\label{conj:sub-bunkbed}
For all \. $G=(V,E)$ \. and all \. $v,w\in V$, we have \. $B_0(v,w) \ge B_1(v,w)$.
\end{conj}

This conjecture was formulated by Kasteleyn (c.~1985), see \cite[Rem.~5]{BK01}, in
the context of \defng{percolation}, and has become popular in the past two
decades, see e.g.\ \cite{Hag03,Lin11} and most recently
\cite{Gri22,HNK21}.\footnote{The conjecture is usually formulated more generally,
as an inequality for $p$-percolation. Replacing edges with series-parallel
graphs simulates $p$-percolation on $G\times K_2$ for all rational~$p$, and shows
that two formulations are equivalent. }
The fact that it is notoriously difficult to establish,
combinatorially or otherwise, suggests the following:

\begin{conj}\label{conj:sub-Bunkbed}
Function \ts $B_0-B_1$ \ts is not in \ts $\SPr$.
\end{conj}

At first glance this might seem contradictory to the bunkbed conjecture,
but notice that it only says that if Conjecture~\ref{conj:sub-bunkbed} holds
then it holds for ``non-combinatorial reasons'', like the van der Waerden
Conjecture.  More precisely, Conjecture~\ref{conj:sub-Ham-non}
rules out a simple direct injection establishing \ts
$B_1 \le B_0$.\footnote{Formally, denote by $\cB_0(v,w)$
and $\cB_1(v,w)$ the sets of subgraphs counted by $B_0$ and $B_1$, respectively.
Suppose there exists an injection \ts $\vp: \cB_1(v,w) \to \cB_0(v,w)$,
s.t.\  both \ts $\vp$ \ts and \ts $\vp^{-1}$ (where defined) are computable in
polynomial time.  Then \ts $B_0(v,w)-B_1(v,w)$ \ts has a combinatorial
interpretation as the number of elements in \ts $\cB_0(v,w)\sm \vp\big(\cB_1(v,w)\big)$.
}
On the other hand, if
Conjecture~\ref{conj:sub-bunkbed} is false, then Conjecture~\ref{conj:sub-Ham-non}
is trivially true.  In other words,  Conjecture~\ref{conj:sub-Ham-non} is
\emph{complementary} to the bunkbed conjecture and could be easier to resolve.

Let us note that the bunkbed conjecture is known in a few special cases,
such as complete graphs \cite{vHL19} and complete bipartite
graphs~\cite{Ric22}.  It would be interesting to see if the proofs
imply that \ts $B_0-B_1\in \SP$ \ts in all these cases.

\smallskip

\subsection{Kleitman's inequality}\label{ss:sub-Kle}
Let \ts $\cA$ \ts be a collection of labeled graphs on $[n]=\{1,\ldots,n\}$.
We say that $\cA$ is \defn{hereditary}, if for every $G\in \cA$ and
every spanning subgraph $H$ of~$G$, we have $H \in \cA$.  Examples of hereditary
properties include \emph{planarity}, \emph{$3$-colorability}, \emph{triangle-free},
\emph{non-connectivity}, \emph{non-Hamiltonicity}, and not containing a perfect matching.

\begin{thm}[{\rm Kleitman~\cite{Kle66}}{}]\label{t:sub-Kle}
Let \ts $\cA$ \ts and \ts $\cB$ \ts be hereditary collections of labeled graphs on~$[n]$.
Then:
$$
|\cA| \cdot |\cB| \, \le \, 2^{\binom{n}{2}} \cdot |\cA\cap\cB|\ts.
$$
\end{thm}

This \defn{Kleitman's inequality} is easier
to understand in probabilistic terms, as having a
positive correlation between uniform random graph events:
$$
\pp[\ts G\in \cA\ts ] \. \le \. \pp[\ts G\in \cA\,| \, G\in \cB\ts ]\ts.
$$
It is then natural to ask if Kleitman's inequality is in~$\SP$.

\begin{prop}\label{p:sub-Kle}
Let \ts $\cA$ \ts and \ts $\cB$ \ts be hereditary collections of labeled graphs on~$[n]$,
such that the membership problems \ts $G\in^?\cA$ \ts and  \ts $G\in^?\cB$ \ts are in~$\Pr$.
Then:
$$
2^{\binom{n}{2}} \cdot |\cA\cap\cB| \. - \.
|\cA| \cdot |\cB| \. \in \. \SPr \ts.
$$
\end{prop}

The result follows from Kleitman's original proof.
In this context, let us mention the \defng{Ahlswede--Daykin {\em (AD)} inequality},
which is an advanced generalization of Kleitman's inequality, see $\S$\ref{ss:no-semi-alg}.
Other classical inequalities such as the \defng{FKG inequality} and the \defng{XYZ inequality}
(see e.g.\ \cite[Ch.~6]{AS16}) are direct consequences of the AD~inequality.

In \cite[Prop.~2.5.1]{IP22}, we prove that assuming the
(univariate) \defng{Binomial Basis Conjecture} (BBC), then
AD~inequality is not in~$\SP$.  In fact, our proof shows that already
\defng{Harris inequality} \cite{Har60} is not in~$\SP$ under~BBC.
This is in sharp contrast to Proposition~\ref{t:sub-Kle}.

\smallskip

\subsection{Ising model}\label{ss:sub-Ising}
In this section, we consider a counting version of the \defng{Ising model},
see e.g.\ \cite[$\S$1.7]{Bax82} for the introduction.

Let $G=(V,E)$ be a multigraph with $n=|V|$ vertices and $m=|E|$ edges.
For a subset $S\subseteq V$, let \.
$$\aligned
\al(S) \, & := \, \{(v,w)\in E\,:\,v,w\in S\} \. \cup \. \{(v,w)\in E\,:\,v,w\notin S\}.
% \ \  \text{and}
 %\\ \be(S)\, & := \,\{(v,w)\in E\,:\,v\in S, w\notin S\}\ts.
\endaligned
$$
%Note that \. $|\al(S)|+|\be(S)|=m$.
Define the \defn{correlation function}\footnote{Compared to the original version
in \cite{Gri67,KS68}, we modify the definition by fixing the same weight \ts
$(\log 2)$ \ts on all edges in~$E$, so the correlation functions have integral values.
Since our graphs can have multiple edges, both the Griffiths and the GKS inequalities
remain equivalent to the original.}
$$
\Cor(v,w) \, := \, \sum_{S\subseteq V\,: \, v,w\ts \in \ts \al(S)} \. 4^{|\al(S)|} \ - \.
\sum_{S\subseteq V\,:\, v,w\ts \notin \ts \al(S)} \. 4^{|\al(S)|}\,.
$$
Note that the statistical sum here is over \defng{induced subgraphs} rather
than the spanning subgraphs in the previous two problems.

Griffiths~\cite{Gri67} showed
that \ts $\Cor(v,w)\ge 0$ \ts by an inductive combinatorial argument.  When
untangled, it can be used to prove the following:
\smallskip

\begin{prop}\label{p:Ising-Cor-functions}
The correlation function \. $\Cor: (G,v,w)\to \nn$ \ts is in \ts $\SPr$.
\end{prop}

%\smallskip

The \defn{Griffiths--Kelly--Sherman} \defn{\em (GKS)}
\defn{inequality} \cite{Gri67,KS68} is a  triangle-type inequality for the
correlation functions:
$$
2^m \. \Cor(v,w) \, \sum_{S\subseteq V} \. 4^{|\al(S)|} \, \ge
\, \Cor(v,u) \. \Cor(u,w) \quad \text{for all} \ \ u,\ts v,\ts w  \in V,
$$
cf.\ \cite[Thm.~3.14]{GP20} for the planar graphs case.

%\smallskip

\begin{conj}\label{op:Ising-Cor-GKS}
The GKS inequality is not in \ts $\SPr$.
\end{conj}

%\smallskip

We refer to~\cite{GHS70} for an even more curious
\defng{Griffiths--Hurst--Sherman} \defng{\em (GHS)}
\defng{inequality}, and to~\cite{Ell85} for
Statistical Physics context, unified proofs and further references.

\bigskip

\section{Linear extensions} \label{s:LE}

As structures go, linear extension of finite posets occupy the middle
ground between \emph{easy} combinatorial objects such as standard Young tableaux
or spanning trees, and \emph{hard} objects such as $3$-colorings of graphs or
Hamiltonian cycles.\footnote{See $\S$? for a complexity theoretic explanation.
We refer to survey articles \cite{BW00,Tro95} for the notation, standard
background on posets, and further references.}

\smallskip

\subsection{Bj\"orner--Wachs inequality}\label{ss:LE-BW}
Let \ts $P=(X,\prec)$ \ts be a finite poset.
A \defn{linear extension} of $P$ is a bijection \ts $\rho: X \to [n]$,
such that \ts $\rho(x) < \rho(y)$ \ts for all \ts $x \prec y$.
Denote by \ts $\Ec(P)$ \ts the set of linear extensions of~$P$,
and write \. $e(P):=|\Ec(P)|$.

For each element \ts $x \in X$, let \.
$B(x) := \big\{y \in X \. : \. y \succcurlyeq x  \big\}$ \.
be the \defn{upper order ideal} generated by~$x$, and let \. $b(x):=|B(x)|$.
The \defn{Bj\"orner--Wachs inequality} \cite[Thm 6.3]{BW89} states that
$$
e(P) \,\. \prod_{x \in X} b(x) \ \ge  \, n!
$$

\begin{prop}[{\rm \cite[Thm~1.13]{CPP22b}}{}] \label{p:LE-BW}
The Bj\"orner--Wachs inequality is in \ts $\SPr$.
\end{prop}

The proof in \cite[$\S$3]{CPP22b} is essentially the same as the original
combinatorial proof by Bj\"orner and Wachs.  This is in contrast with a probabilistic
proof in \cite[$\S$4]{CPP22b} via (another) Shepp's inequality~\cite{She80}
which in turn uses the FKG inequality. Similarly, this is in contrast with
the Hammett--Pittel analytic proof~\cite{HP}, and Reiner's \defng{$q$-analogue} \ts
based proof given in \cite[$\S$5]{CPP22b}.  Neither of these three proofs seem
to imply the proposition.

\smallskip

\subsection{Sidorenko's inequality}\label{ss:LE-Sid}
A \emph{chain} in a poset \ts $P=(X,\prec)$ \ts is a subset \ts
$\{x_1,\ldots,x_\ell\} \subseteq X$, such
that \ts $x_1 \prec x_2 \prec \ldots \prec x_\ell\ts.$
Denote by \ts $\cC(P)$ \ts the set of chains in~$P$.

Suppose \ts $P=(X,\prec)$ \ts and \ts $Q=(X,\prec')$ \ts be two posets
on the same set with \. $|X|=n$ \. elements, such that \.
$\bigl|C \cap C'\bigr| \le 1$ \. for all \ts $C\in \cC(P)$ \ts and \ts $C' \in \cC(Q)$.
\defn{Sidorenko's inequality} states that \, $e(P) \ts e(Q) \ge \. n!$ \. \cite{Sid}.

\begin{prop}[{\rm \cite[Thm~1.15]{CPP22b} and \cite[$\S$3]{GG20b}}{}]\label{p:LE-Sid}
Sidorenko's inequality is in \ts $\SPr$.
\end{prop}

Natural examples of posets \ts $(P,Q)$ \ts as above,
are the \defn{permutation posets} \ts $\bigl(P_\si, P_{\ov \si}\bigr)$,
where \ts $P_\si=([n],\prec)$ \ts is defined as \.
$i\prec j$ \. if and only if \.  $i<j$ \ts and \ts $\si(i)<\si(j)$ \ts for all \ts $i,j\in [n]$,
and \.  $\ov{\si}:=\bigl(\si(n),\ldots,\si(1)\bigr)$.
In this case \ts $P_{\si}$ \ts is a \defng{$2$-dimensional poset}, and
\ts $P_{\ov \si}$ \ts is its \defng{plane dual}.

The original proof by Sidorenko was based on combinatorial optimization.
Other proofs include and earlier direct surjection by Gaetz and Gao \cite{GG20a},
and the geometric proof by Bollob\'as, Brightwell and Sidorenko \cite{BBS99},
via Stanley's theorem on \defng{poset polytopes\ts}~\cite{Sta86} and Saint-Raymond's
proof of \defng{Mahler's conjecture} for convex corners \cite{StR81}.
Neither of these three proofs imply the proposition, at least not directly.

\smallskip

\subsection{Stanley's inequality}\label{ss:LE-Stanley}
Let \ts $P=(X,\prec)$ \ts be a finite poset.
For an element \ts $x\in X$ \ts and integer \ts $k \in [n]$, denote by
\ts $\aN(k)=\aN(P,x,k)$ \ts the number of linear extensions \ts $\rho\in \Ec(P)$,
such that \ts $\rho(x)=k$.  \defn{Stanley's inequality} \cite[Thm~3.1]{Sta81} states that:
%\begin{*equation}\label{eq:Sta-ineq}
$$\aN(k)^2 \,\. \ge \,\. \aN(k-1) \,\. \aN(k+1) \quad \text{for all} \ \ 1< k<n\ts.
$$
%\end{equation}

\begin{conj}\label{conj:LE-Stanley}
Stanley's inequality is not in \ts $\SPr$.
\end{conj}

In the past few years, we made a considerable effort trying to resolve this problem.
The original proof in~\cite{Sta81} is an easy but ingenuous reduction from
the \defng{Alexandrov--Fenchel inequality}, that was used previously to prove
the van def Waerden conjecture (see~$\S$\ref{ss:sub-PM}).
Since poset polytopes are rather specialized,
initially we hoped that Stanley's inequality is in~$\SP$.  Another
positive evidence is the recent \defng{effective} characterization
of the equality conditions in Stanley's inequality, by Shenfeld and van~Handel \cite{SvH20}.
In our terminology, they showed that the vanishing problem \ts
$\big\{\aN(k)^2 =^?\aN(k-1) \. \aN(k+1)\big\}$ \ts can be decided
in poly-time.

With Chan~\cite{CP21}, we develop a new \defng{combinatorial atlas} technology
(see also~\cite{CP22}), which gave a purely linear algebraic proof of Stanley's
inequality and its generalization to weighted linear extensions (see also \cite{CP22+}).
This also allowed us to give a new proof of the equality conditions.  Unfortunately,
the limit argument in our proof does not allow to give a $\SP$ description
(see \cite[$\S$17.17]{CP21}).
Separately, with Chan and Panova, we employed several combinatorial approaches
in \cite{CPP21a,CPP22a} to posets of width two, and an algebraic approach in
\cite[$\S$7 and $\S$9.11]{CPP22b}.  Unfortunately, Conjecture~\ref{conj:LE-Stanley}
remains elusive.

\smallskip

\subsection{XYZ inequality}\label{ss:LE-XYZ}
Let \ts $P=(X,\prec)$ \ts be a finite poset, and let \ts $x,y,z\in X$ \ts
be incomparable elements.  Let \.
$P_{xy}:= P\cup \{x\prec y\}$, \. $P_{xz}:= P\cup \{x\prec z\}$ \. and \. $P_{xyz}:= P\cup \{x\prec y, x\prec z\}$.
Shepp's classic \defn{XYZ inequality}~\cite{She82},
states that
$$%(\boxplus) \qquad
e(P\bigr) \, e(P_{xyz}) \ \ge \
e(P_{xy}\bigr) \, e(P_{xz}).
$$
As with other correlation inequalities, the XYZ inequality is easier to understand in terms
of uniform random linear extensions \ts $\rho \in \Ec(P)$:
$$
\pp[\ts \rho(x)< \rho(y) \, | \, \rho(x)< \rho(z) \ts ] \, \ge \,  \pp[\ts \rho(x)< \rho(y) \ts ]\ts.
$$

% \smallskip

\begin{conj}\label{conj:LE-XYZ}
The XYZ inequality is not in \ts $\SPr$.
\end{conj}

\smallskip

The original proof (see also \cite[$\S$6.4]{AS16}), uses the FKG inequality,
which, like the AD inequality, is not in $\SP$ in full generality.
A double counting argument
proving the XYZ inequality is given in~\cite{BT02}.  Unfortunately,
it does not prove that the inequality is in~$\SP$, both because of the
double counting argument and the use of \textsc{BipartiteUnbalance} problem,
see \cite[$\S$9.2]{IP22}.  This makes the conjecture especially interesting.

\bigskip

\section{Young tableaux}\label{s:tab}
We adopt the standard and largely self-explanatory notation from
Algebraic Combinatorics.  We refer to \cite{Sag92,Sta99} for both notation
and the background.  Unless stated otherwise, we use unary encoding for
all partition and integer parameters throughout this section.  For convenience,
we are using notation \ts $f\le_\# g$ \ts to mean that \ts $(g-f)\in \SP$.

\smallskip

\subsection{Standard Young tableaux}\label{ss:tab-SYT}
Let \ts $f^\la=|\SYT(\la)|$ \ts be the number of \defng{standard Young tableaux}
of shape $\la\vdash n$.  Recall the \defn{hook-length formula}:
\begin{equation}\label{eq:HLF} \tag{HLF}
f^{\lambda} \, = \, n!\, \prod_{(i,j)\in \lambda}\, \frac{1}{h_\la(i,j)} \,,
\end{equation}
where \ts $h_\la(i,j)=\la_i+\la'_j-i-j+1$ \ts is the \defn{hook-length} in~$\la$.
This implies that \ts $f^\la\in \FP$.  Consequently, we have that \ts
$n!/f^\la\in \FP$ \ts and \ts $(f^\la)^2 \le_\# n!$

Similarly, let \ts $f^{\la/\mu}=|\SYT(\la/\mu)|$ \ts
the number of standard Young tableaux of skew shape $\la/\mu$,
where $|\la/\mu|=n$.  Recall the \defn{Aitken--Feit
determinant formula}:
\begin{equation}\label{eq:AFDF} \tag{AFDF}
f^{\la/\mu} \, =  \, n!\.
\det \left( \frac{1}{(\la_i-\mu_j -i+j)!}\right)_{i,j=1}^{\ell(\la)}\,.
\end{equation}
which proves that \ts $f^\la\in \FP$.  Consequently, we have that
\ts $f^{\la/\mu}\le_\# n!$

Now, it follows from the \defn{Naruse hook-length formula} (NHLF) that
\begin{equation}\label{eq:NHLF} \tag{$\otimes$}
f^{\lambda/\mu} \, \prod_{(i,j)\in \lambda/\mu} h_\la(i,j) \
\ge \ n! \,,
\end{equation}
see \cite{MPP18b}.  Note that this is a much sharper bound than the
one given by the Bj\"orner--Wachs inequality (see~$\S$\ref{ss:LE-BW}).

With Morales and Panova, we gave several proofs of the NHLF and its
generalizations, both algebraic and combinatorial \cite{MPP17,MPP18a}.
A recursive proof is given by Konvalinka \cite{Kon20}.
Finding a direct combinatorial proof which allows efficient sampling from
$\SYT(\la/\mu)$, perhaps generalizing the
\defng{NPS bijection} or the \defng{GNW hook walk},
remains an important open problem, see e.g.\ \cite[$\S$5.6]{HMPSW}.

\begin{conj}\label{conj:NHLF-lower-SP}
Inequality~\eqref{eq:NHLF} is in~$\SPr$.
\end{conj}

{\small
\begin{rem}
Even the simplest special cases of $(\otimes)$ are hard to establish
directly.  For example, rotate diagram~$\la$ by 180$^\circ$ and denote by
$\la^\ast$ be the resulting skew shape.  Let \ts $h_\la^\ast(i,j) := i+j-1$
\ts be the hook-lengths in~$\la^\ast$.  Inequality~$(\otimes)$ in this case
follows from
$$(\otimes\otimes) \qquad
\prod_{(i,j)\in \lambda^\ast} h_{(a^b)}(i,j) \ = \
\prod_{(i,j)\in \lambda} h^\ast_\la(i,j) \ \ge \
\prod_{(i,j)\in \lambda} h_\la(i,j)\,,
$$
where \ts $a=\la_1$ \ts and \ts $b=\ell(\la)$, see \cite[Prop.~12.1]{MPP18b}.
A direct combinatorial proof of $(\otimes\otimes)$ is given in~\cite{PPS20}.
This proof uses \defng{Karamata's inequality} \ts that is not
in $\SP$ in full generality, see~\cite[$\S$7.5]{IP22}.
\end{rem}
}

\smallskip

\subsection{Semistandard Young tableaux}\label{ss:tab-SSYT}
Let \ts $K_{\la\ts \mu} = |\SSYT(\la,\mu)|$ \ts be the \defn{Kostka numbers},
where \ts $\la,\mu\vdash n$.  The inequality \ts
$K_{\la\mu}\le_\# K_{\la \ts 1^n} = f^\la$  \ts is easy to show directly.
Much more interesting is the following generalization.

Let \ts $\ba =(a_1,\ldots,a_\ell)$ \ts and \ts $\bb=(b_1,\ldots,b_\ell)$ \ts
be two weakly decreasing vectors.  We say that
$\ba$ \defn{majorizes} $\bb$, write \ts $\ba \trianglerighteq \bb$,
if \. $a_1+\ldots +a_i \ge b_1+\ldots +b_i$ \. for all $1\le i < \ell$, and  \.
$a_1+\ldots +a_\ell = b_1+\ldots +b_\ell$\ts.  Recall that \ts
$K_{\la\mu} \le K_{\la\nu}$ \ts for all \ts $\mu \trianglerighteq \nu$.

\begin{prop} \label{prop:Kostka-White}
$K_{\la\mu} \le_\# K_{\la\nu}$ \. for all \ts $\mu \trianglerighteq \nu$.
\end{prop}

Although not stated in this language, the proof follows easily from the
combinatorial proof in~\cite{Whi80}.  Here we are using the following
trivial observation:  if \ts $f_1,\ldots,f_k\in \SP$ \ts for $k=$ poly$(n)$,
then \ts $f_1+\ldots + f_k \in \SP$.

Finally, let \ts $\la,\al,\be,\ga\vdash n$ \ts be such that \ts
$\al \trianglerighteq \be \trianglerighteq \ga$ \ts and \ts $\al+\ga=2\ts \be$.
Kostka numbers then satisfy log-concavity property given by the
\defn{HMMS inequality}:
\ts $K_{\la\be}^2 \ge K_{\la\al} K_{\la\ga}$ \ts \cite[Thm~2]{HMMS22}.

\begin{op}\label{op:Kostka-log-concavity}
The HMMS inequality is not in \ts $\SPr$.
\end{op}

The original proof of the HMMS inequality is based
on the \defng{Lorentzian property} of \defng{Schur polynomials}.
Solving the open problem would give an early indication in favor
of Conjecture~\ref{conj:sub-forest}, since the log-concavity
of forest numbers \ts $\{F(G,k)\}$ \ts is proved in \cite{BH20}
based on the same general approach.

{\small

\begin{rem}\label{r:tab-dominance}
The \defn{Dominance order} ``$\triangleright$'' is motivated by a technical part
in the proof of the \defng{Young symmetrizer} \ts construction,
see e.g.\ \cite[$\S$IV.2]{Weyl39} and \cite[$\S$2.4]{Sag01}, and reflects the inherent
planarity of Young diagrams.  The algebraic proof of
Proposition~\ref{prop:Kostka-White} given in~\cite{LV73} is based on iterative
calculation.\footnote{Dennis White kindly informed us that he meant \cite{LV73} as a missing
reference [4] in \cite{Whi80}.}
We note that \cite{Ver06} emphasizes the importance of identity \ts
$$(\vartriangle) \qquad
1 \uparrow^{S_n}_{S_{\mu_1}\times S_{\mu_2}\times \. \cdots} \, \bigcap \
\sign \uparrow^{S_n}_{S_{\mu_1}\times S_{\mu_2}\times \.\cdots} \. = \. \SS^\mu
$$
(see also \cite{May75}).  This is a curious byproduct
of the dominance order and its reverse.
\end{rem}
}
\smallskip

\subsection{Contingency tables}\label{ss:tab-CT}
Let \ts $\cT(\la,\mu)$ \ts be the set of \defn{contingency tables}, defined
as nonnegative integer matrices with rows sums~$\la$ and column sums~$\mu$.
Denote by \ts $\rT(\la,\mu):=|\cT(\la,\mu)|$ \ts the number of such tables.
Note that \ts $K_{\la,\mu}\le_\# \rT(\la,\mu)$, where the injection sends \ts
$A\in \SSYT(\la,\mu)$ \ts into a matrix \ts $(m_{ij})\in \cT(\la,\mu)$, where
\ts $m_{ij}$ \ts is the number of letters $j$ in $i$-th row of~$A$.

\smallskip

\defn{Barvinok's inequality} \cite{Bar07}, states that \. $\rT(\la,\mu) \le \rT(\nu,\tau)$ \. for all
\ts $\la \trianglerighteq \nu$ \ts and \ts $\mu \trianglerighteq \tau$.

\begin{prop} \label{prop:CT-Barv}
Barvinok's inequality is in \ts $\SPr$.
\end{prop}

There are two natural proofs of Proposition~\ref{prop:CT-Barv}.  First,
following the original proof in \cite[p.~111]{Bar07}, one can
use the \ts \defng{RSK correspondence}, which is poly-time by definition.
This gives \ts $\rT(\la,\mu) = \sum_{\pi} K_{\pi\la} K_{\pi\mu}$ \ts
and then use Proposition~\ref{prop:Kostka-White}.  We then need to
use the \defng{inverse RSK correspondence}, which is also poly-time.
A more direct (still rather involved) approach is outlined
in~\cite{Pak19}.\footnote{A crucial
part of the injective proof of both White's and Barvinok's inequality
is the \defng{parenthesization construction} by Greene and Kleitman,
see \cite{GK76} and \cite[$\S$3]{GK78} (see also~\cite{dB+51}).}

{\small
\begin{rem} \label{rem:CT-Barv}
We note in passing that the log-concavity property for contingency
tables remains open.  Formally, let \ts $\la,\al,\be,\ga\vdash n$ \ts be
as in the HMMS inequality (see $\S$\ref{ss:tab-SSYT}).  Barvinok asks
if \. $\rT(\la,\be)^2 \ge^? \rT(\la,\al) \. \rT(\la,\ga)$, see \cite[p.~110]{Bar07}.
\end{rem}
}
\smallskip

\subsection{Littlewood--Richardson coefficients}\label{ss:tab-LR}
Let \ts $c^\la_{\mu\nu}$ \ts be the \defn{Littlewood--Richardson {\em (LR)} coefficients},
where \ts $\la\vdash n$, \ts $\mu\vdash k$ \ts and \ts $\nu\vdash n-k$.
Standard combinatorial interpretations for LR coefficients imply:
$$
c^\la_{\mu\ts\nu} \ts f^\mu \ts f^\nu \, \le_\# \, f^\la \quad \text{and} \quad
c^\la_{\mu\ts\nu} \ts f^\la \, \le_\# \, \tbinom{n}{k} \ts f^\mu \ts f^\nu\..
$$
We refer to \cite{Ker84,Whi81,Zel81} for the motivational explanation on how
to derive these inequalities from the RSK correspondence or via the \defng{jeu-de-taquin
correspondence}.  Taking their product gives \. $(c^\la_{\mu\ts\nu})^2 \le \binom{n}{k}$,
as was recently observed in~\cite[$\S$4.1]{PPY19}. We call this the \defn{PPY inequality}.

\begin{op} \label{op:tab-PPY}
The PPY inequality is not in~$\ts\SPr$.
\end{op}

As we discussed earlier, the problem could be resolved either by a
direct injection proving the PPY inequality, or by giving an explicit
combinatorial interpretation for \. $\binom{n}{k} -(c^\la_{\mu\ts\nu})^2$.

\smallskip

The remarkable \defn{Lam--Postnikov--Pylyavskyy inequality}~\cite{LPP07} states
that:\footnote{The actual LPP inequality in~\cite{LPP07}
is more general and written in terms of skew Schur functions. }
\begin{equation}\label{eq:LPP}\tag{LPP}
c^\la_{\mu\ts \nu} \. \le \. c^\la_{\mu\vee \nu \. \mu \wedge\nu}\,,
\end{equation}
where \ts $\mu\vee \nu$ \ts and \ts $\mu\wedge \nu$ \ts denote
the union and the intersection of Young diagrams.
The original proof is heavily algebraic and does not seem to give a
clue on how this can be proved injectively, see \cite{BBR06} for
some special cases.

\begin{op} \label{op:tab-LPP}
The LPP inequality is not in~$\ts\SPr$.\footnote{We state both open
problems in the negative largely because we would
much rather see negative solutions than positive.  Unfortunately,
at the moment there is very little evidence in favor of either direction.
}
\end{op}

Note a closely related \defn{Bj\"orner's inequality}:
$$\tbinom{n}{m} \ts f^{\mu} \ts f^{\nu} \, \le \, \tbinom{n}{r} \ts f^{\mu\vee \nu} \ts f^{\mu\wedge \nu},
$$
where \ts $|\mu|=m$, \ts $|\mu|+|\nu|=n$ \ts and \ts $|\mu\wedge\nu|=r$ \ts \cite[$\S$6]{Bjo11}.
It follows by summing over all $\la\vdash n$, of \eqref{eq:LPP} multiplied by~$f^\la$.
Of course, Bj\"orner's inequality is in~$\FP$ and follows from the~HLF.

\smallskip

\subsection{Inverse Kostka numbers} \label{ss:tab-Inverse-Kostka}
Denote by \ts $\bK=\big(K_{\la\mu}\big)$ \ts the \defn{Kostka matrix}.
Ordering all partitions w.r.t.\ the size $n=|\la|=|\mu|$ and the
majorization order \ts ``$\triangleright$'' \ts and
using \ts $K_{\la\la}=1$, we conclude that the matrix is upper triangular
and thus has an integer inverse.  Denote by \ts $\bK^{-1}=\big(K^{-1}_{\la\mu}\big)$ \ts
the \defn{inverse Kostka matrix}, and by \ts $K^{-1}_{\la\mu}$ \ts the
\defn{inverse Kostka numbers}.

E\u{g}ecio\u{g}lu and Remmel~\cite{ER90} showed that \ts $K^{-1}_{\la\mu}$ \ts
has a signed combinatorial interpretation as a sum over certain
\defng{rim-hook tableaux} (RHT) of shape~$\la$ and weight~$\mu$. It follows from
the construction that \ts $K^{-1}$ \ts is in \ts $\GapP=\SP-\SP$.
Direct involutions proving validity \ts $\bK\cdot \bK^{-1} = \bK^{-1}\cdot \bK = I$ \ts
were given in~\cite{ER90,LM06}.  Other signed combinatorial interpretations
are given in~\cite{Duan03,PR17}.

\begin{conj}\label{conj:tab-inverse-Kostka}
The function \ts $|K^{-1}_{\la\mu}|$ \ts
is not in $\ts\SPr$.
\end{conj}

In other words, the conjecture claims that the absolute value of the
inverse Kostka numbers does not have a combinatorial interpretation.
Thus, a signed sum over combinatorial objects is the best one can have.
Further motivation behind the conjecture will become clear in the next
section.

\bigskip

\section{Characters}\label{s:char}
In this section we discuss complexity problems related to $S_n$ characters.
As before, $n$ and all partitions are given in unary.

\smallskip

\subsection{The values}\label{ss:char-values}
Let \ts $\chi^\la(\mu)$ \ts denote the character value of the irreducible
$S_n$ module \ts $\SS^\la$ \ts on the conjugacy class \ts $[\mu]$,
where \ts $\la,\mu\vdash n$. The \defng{Murnaghan--Nakayama {\em (MN)} rule},
see e.g.\ \cite[$\S$4.10]{Sag92} and \cite[$\S$7.17]{Sta99}, gives a signed
combinatorial interpretation for \ts $\chi^\la(\mu)$ \ts as a sum of signs
over \defng{rim-hook tableaux} \ts $\RHT(\la,\mu)$ \ts of shape~$\la$ and weight~$\mu$.  Similar to
the inverse Kostka numbers, it follows from
the construction that function \ts $\chi: (\la,\mu) \to \zz$ \ts is in \ts $\GapP=\SP-\SP$.

Does there exist a combinatorial interpretation of the character square
\ts $\bigl(\chi^\la(\mu)\bigr)^2$?  This is an interesting question, and the answer is
even more interesting.  On the one hand, the answer is yes when \ts
$\mu=(k^{n/k})$ \ts is a rectangle.   In this case, all rim-hook tableaux in \ts
$\RHT(\la,\mu)$ \ts have the same sign, see e.g.\ \cite[$\S$2.7]{JK81} and \cite{SW85},
so \ts $\big(\chi^\la(\mu)\big)^2=|\RHT(\la,\mu)|^2$ \ts has a combinatorial interpretation
as the number of ordered pairs of rim-hook tableaux.\footnote{In fact, it follows from
\cite{FS98,SW85} that \ts $\big\{|\RHT(\la,\mu)|\big\}\in \FP$ \ts in this case.  }

On the other hand, in full generality we have:

\smallskip

\begin{thm}[{\cite{IPP22}}{}] \label{t:char-abs}
If \. $(\chi)^2 \in \SPr$, then \ts $\coNPr = \CeqP$ \ts and \ts $\PHr=\Sigma_2^{\textsc{p}}$\ts.
\end{thm}

\smallskip

In other words, it is \emph{very unlikely} that character square has a
combinatorial interpretation, assuming th polynomial hierarchy does
not collapse to the second level.

{\small
\begin{rem} \label{rem:char-identities}
Following \cite{IPP22}, one way to understand the implications of the theorem
is to compare two identities:
$$
n! \, = \, \sum_{\la\vdash n} \, \big(\chi^\la(1)\big)^2  \quad \text{and} \quad
n! \, = \, \sum_{\pi\in S_n} \big(\chi^\la(\pi)\big)^2\,, \ \ \. \text{for all \ $\la\vdash n$.}
$$
The former is the \defng{Burnside identity} \ts since \ts $\chi^\la(1)=f^\la$, and follows
from the RSK correspondence.  The latter  is the
character orthogonality formula, and the theorem explains why there is no natural
analogue of the RSK correspondence in this case.  Simply put, the terms on the right
are not actually counting any combinatorial objects.\footnote{There is, however,
an involutive proof of both character orthogonality relations based on the MN rule
\cite{Whi83,Whi85}.
}
\end{rem}
}

\smallskip

\subsection{Row and column sums}\label{ss:char-sums}
In~\cite[$\S$3]{Sta00}, Stanley defines
$$
a_\la \. := \. \sum_{\mu\vdash n} \. \chi^\la(\mu) \quad \text{and}
\quad b_\la \. := \. \sum_{\mu\vdash n} \. \chi^\mu(\la)\., \quad \text{where \ $\la\vdash n$.}
$$
which he calls \defn{row sums} and \defn{column sums} in the character table, respectively.
It is known that \ts $b_\la = \bigl|\{\om\in S_n\,:\,\om^2=\si\}\bigr|$, where \ts
$\si\in [\la]$ \ts is a fixed permutation of type~$\la$,
see e.g.\ \cite[Exc~7.69]{Sta99} and~\cite[Ex.~11, p.~120]{Mac95}.\footnote{It follows
form here that \ts $b_\la >0$ \ts if and only if every even part of~$\la$ has even
multiplicity, see e.g.\ \cite{BO}.}
Thus, the column sums are
in~$\ts\SP$.  One can ask a similar question about the row sums.

First, we note that \ts $a_\la \ge 0$ \ts for all~$\la$.\footnote{It
is known that \ts $a_\la \ge 1$ \ts for all $\la\vdash n\ge 2$ \cite{Fru86}. For
the alternating group~$A_n$,
this \defng{strict positivity} is proved in~\cite{HZ06} by a non-combinatorial
(and much shorter) argument. For other finite simple groups, strict positivity
is proved in~\cite{HSTZ13}.}  This follows from the fact that \.
$a_\la = \<\rho_n,\chi^\la\>$,
where \ts $\rho_n$ \ts is the character of the
\defng{conjugation representation}, see e.g.\ \cite[Exc~7.71]{Sta99}.\footnote{In fact,
this approach can be used to show that
row sums of characters are nonnegative integers for all finite
groups \cite{Sol61}.}

\begin{conj}[{\rm cf.\ Problem~12 in~\cite{Sta00}}{}] \label{conj:char-row-sums}
Row sums \ts $\{a_\la\}$ \ts are not in~$\ts\SPr$.
\end{conj}

In fact, there are very few cases when a combinatorial interpretation of
\ts $a_\la$ \ts is known.  For example, \ts $a_{(n)}= p(n)$ \ts and \ts
$a_{(1^n)}=|\{\la\vdash n\.:\. \la=\la'\}|$, see \cite[Prop.~1]{BE16}.
Proving the conjecture would represent a major advance in the area,
as we explain below.

{\small
\begin{rem}  \label{rem:char-total}
The total sum of the entries in the character tables of $S_n$ is a sequence
of interest in its own right, see \cite[\href{https://oeis.org/A082733}{A082733}]{OEIS}.
So is the determinant of the character table, see \cite[Cor.~6.5]{Jam78} and
\cite[\href{https://oeis.org/A007870}{A007870}]{OEIS}, and even the permanent
\cite{SS84}.
\end{rem}
}

\smallskip

\subsection{Refinements}\label{ss:char-refine}
Denote by \ts $\rho^\mu[\si] := |\{\om \in [\mu] \. : \. \om \si = \si\om\}|$,
where \ts $[\mu]\ssu S_n$ is the conjugacy class of permutations of type
$\mu\vdash n$.  In other words, \ts $\rho^\mu$ \ts is the character of the
conjugation action on the conjugacy class $[\mu]$.
From above, we have \. $\rho_n = \sum_{\mu\vdash n} \rho^\mu$.
Define \defn{refined row sums} \. $a_{\la\mu} := \<\rho^\mu,\chi^\la\> \in \nn$,
and note that \. $a_\la = \sum_{\mu\vdash n} a_{\la\mu}$\ts.

\begin{conj}\label{conj:char-row-sums-refined}
Refined row sums \ts $\{a_{\la\mu}\}$ \ts are not in~$\ts\SPr$.
\end{conj}

Clearly, the proof of Conjecture~\ref{conj:char-row-sums}
implies the same for Conjecture~\ref{conj:char-row-sums-refined}.
We warn the reader that it does not follow from definition
that \ts $\{a_{\la\mu}\}\in \GapP$.  Indeed, the definition states
$$
a_{\la\mu} \, := \, \<\rho^\mu,\chi^\la\> \, = \, \frac{1}{n!} \. \sum_{\om \in S_n} \. \rho^\mu[\si] \. \chi^\la[\om]\..
$$
Even though the terms on the RHS are in~$\SP$, there is no obvious
way to divide the sum by~$n!$  The claim is true nonetheless, as
we explain later in this section.

\smallskip

Towards the open problem, it was proved by Kra\'{s}kiewicz
and Weyman~\cite{KW01}
(see also \cite[Exc~7.88b]{Sta99} and \cite[$\S$3]{RW20}), that
\begin{equation}\label{eq:KW}\tag{KW}
a_{\la\ts (n)} \, = \, \big|\bigl\{A\in \SYT(\la) \ :\ \maj(A) = 0 \mod n\bigr\}\big|\.,
\end{equation}
where \ts $\maj(A)$ \ts denotes \defng{major index} of~$A$.
This implies that \ts $\{a_{\la\ts (n)}\}\in\SPr$.
The following result show how close \eqref{eq:KW} gets us
to resolving the open problem.

\begin{prop}[{\rm folklore}{}]  \label{prop:char-dist}
Denote by $\cd$ the set of partitions into distinct parts.
Then \\ $\{a_{\la\mu}\.:\. \mu \in \cd\} \in\SPr$.  Furthermore,
if \ts $\{a_{\la\mu}\.:\.\mu =(r^{n/r})\}\in\SPr$,
then \ts $\{a_{\la\mu}\}\in\SPr$.
\end{prop}

Versions of this result are well-known.  For completeness, we include
a short proof in~$\S$\ref{ss:proofs-char-dist}.  The proposition implies
that to disprove Conjecture~\ref{conj:char-row-sums} it suffices
to give a combinatorial interpretation for \ts
$\big\{a_{\la\. r^{n/r}}\big\}$.

{\small
\begin{rem}\label{rem:char-White-Polya}
Curiously, refined row sums can be defined and generalized using
P\'olya's theory for general permutation groups, see \cite{Whi19,RW20}.
This approach leads to a plethora of numbers in search of combinatorial
interpretations, including some of those in~$\S$\ref{ss:seq-isom}.
From our point of view, this is a good starting
place to prove for non-existence of such combinatorial
interpretations in full generality.
\end{rem}
}

\smallskip

\subsection{Plethysm coefficients}\label{ss:char-plethysm}
Denote by \ts $p_\la(\mu,\nu)$ \ts the \defn{plethysm coefficient},
which can defined in terms of Schur functions as \ts
$p_\la(\mu,\nu)=\<s_\mu[s_\nu],s_\la\>$, see e.g.\ \cite[\S$7$.A2]{Sta99}
and \cite[$\S$1.8]{Mac95}.  Note that the bracket product \ts $s_\mu[s_\nu]$ \ts is
noncommutative and equal to the trace of the composition of irreducible
$\GL$-modules corresponding to $\mu$ and $\nu$: \ts $S^\mu(S^\nu V)$.

\begin{conj}[{\rm cf.\ Problem~9 in~\cite{Sta00}}{}]\label{conj:char-plethysm}
Plethysm coefficients \ts $\big\{p_\la(\mu,\nu)\big\}$ \ts are not in~$\ts\SPr$.
\end{conj}

Computing plethysm coefficients is so exceedingly difficult, there are very
few special cases when they are known to have a combinatorial interpretation.
The formalism in~\cite[$\S$4]{LR11} implies that \ts
$\big\{p_\la(\mu,\nu)\big\}\in \GapP$.\footnote{See also \cite{FI20} for the
binary case.}

The refined character sums can be expressed in terms of plethysm coefficients:
$$
a_{\la \ts (m^k)}\, = \, p_\la\big((k),\rho^{(m)}\big)
\, = \, \sum_{\nu\vdash m} \. a_{\nu \ts (m)} \, p_\la\big((k),\nu\big),
$$
where $\la\vdash n=m\ts k$, see e.g.\ \cite{AS18,Sun18}.  By~\eqref{eq:KW},
a combinatorial interpretation for plethysm coefficients \ts $p_\la(\mu,\nu)$ \ts
when $\mu$ is a row shape, suffices to disprove Conjectures~\ref{conj:char-row-sums}
and~\ref{conj:char-row-sums-refined}.

{\small
\begin{rem}\label{rem:char-Lie}
There is a closely related study of multiplicities \ts $\wt a_{\la\mu}$ \ts
in the \defn{higher Lie modules}, see \cite{AS18,Kly74}.  The special case \ts $\mu=(n)$ \ts is given
by
$$
\wt a_{\la\ts (n)} \, = \, \big|\bigl\{A\in \SYT(\la) \ :\ \maj(A) = 1 \mod n\bigr\}\big|\.,
$$
see \cite{KW01}.  The results are completely parallel here:
it is not known whether  \ts $\big\{\wt a_{\la\mu}\big\}$ \ts are in $\SP$, and
it suffices to resolve this problem in the rectangular case
\ts $\wt a_{\la \ts (m^k)}$, which in turn reduces to plethysm coefficients \ts
$p_\la\big((k),\nu\big)$.  We refer to \cite{AS18,Sun18} for details and
further references.
\end{rem}
}

\smallskip

\subsection{Hurwitz numbers} \label{ss:char-Hurwitz}
Let \ts $\mu=(\mu_1\ldots,\mu_\ell)\vdash n$, and let \ts $g\le 0$ \ts be a fixed constant.
Denote by \ts $H_{g\mu}$ \ts is the number of products of transpositions \ts
$(i_1,j_1)\cdots (i_m,j_m)=w$ \ts in $S_n$, such that

\smallskip

$\circ$ \ $m\ts =\ts 2g-2+n+\ell(\mu)$,

$\circ$ \ $w\in S_n$ \ts has cycle type~$\mu$, and

$\circ$ \ $\<(i_1,j_1),\ldots,(i_m,j_m)\>\subseteq S_n$ \ts is a transitive subgroup.

\smallskip
\nin
The (single) \defn{Hurwitz numbers} \ts are defined as \ts $h_{g\mu}:=H_{g\mu}/n!$,
see e.g.~\cite{GJ97}.  Although one can use the Frobenius character formulas to prove
that they are integers, the following result comes as a surprise.

\smallskip

\begin{prop}[{\rm see \cite{GJV05}}{}]\label{prop:char-Hurw}
Hurwitz numbers \ts $\{h_{g\mu}\}$ \ts are in~$\ts\SPr$.
\end{prop}

\smallskip

The result extends to \defng{double Hurwitz numbers} which we leave undefined.
See also \cite{DPS14} for an alternative combinatorial interpretation
via bijection with certain \defng{Hurwitz galaxies}.

In the minimal case $g=0$, Hurwitz numbers have a product formula, and thus in~$\FP$.
The proposition is remarkable since most proofs of this formula use a
\defng{double counting argument}, see~\cite{BS}.  Notably, we recall D\'enes's
beautiful proof of \ts $h_{0\ts (n)}=n^{n-2}$ \ts formula
for the number of minimal factorizations of a long cycle.
See \cite{GY02} for a bijection in this case.

{\small

\begin{rem}\label{r:char-Hurw}
The subject of Hurwitz numbers is quite extensive and ever growing,
so giving comprehensive references is a challenge.
Let us mention that already in his original paper~\cite{Hur91},
Hurwitz gave a topological interpretation of \ts $\{h_{g\mu}\}$, see e.g.~\cite{CM16}.
Hurwitz numbers also arise in the study of branched covers of \ts $\cc\pp^1$ \ts
in Enumerative Algebraic Geometry, see~\cite{ELSV,Oko,OP}.
We refer to~\cite{BS,DPS14,PS02} for a more combinatorial treatment.
\end{rem}
}
\bigskip

\section{Kronecker coefficients}\label{s:Kron}

\subsection{Reaching for the stars} \label{ss:Kron-def}
Let \ts $g(\la,\mu,\nu)$, where $\la,\mu,\nu\vdash n$,
denote the \defn{Kronecker coefficients}:
$$
g(\la,\mu,\nu) \, := \, \<\chi^\la\chi^\mu,\chi^\nu\> \, = \, \frac{1}{n!} \.
\sum_{\si \in S_n} \. \chi^\la(\si) \. \chi^\mu(\si)\. \chi^\nu(\si)\..
$$
By definition, \ts $g(\la,\mu,\nu)\in \nn$.  Whether it has a combinatorial
interpretation remains a major open problem first posed by
Murnaghan \cite{Mur38,Mur56}.

\begin{conj}[{\rm cf.\ Problem~10 in~\cite{Sta00}}{}] \label{conj:Kron-main}
Kronecker coefficients \ts $\{g(\la,\mu,\nu)\}$ \ts are not in~$\ts\SPr$.
\end{conj}

It is known that  \ts $\{g(\la,\mu,\nu)\}\in \GapP$, see \cite{BI08}.
This follows, for example, from
$$
g(\la,\mu,\nu) \ = \ \sum_{\om\in S_\ell} \, \sum_{\pi\in S_m} \,
\sum_{\tau\in S_r} \ \, \text{sign}(\om\ts \pi \ts \tau) \.\cdot \.
\rT\bigl(\la+(1^\ell)-\om, \. \mu+(1^m)-\pi, \. \la+(1^r)-\tau \bigr),
$$
where \ts $\ell(\la)=\ell$, $\ell(\mu)=m$, $\ell(\nu)=r$, and \ts
$\rT(\al,\be,\ga)$ \ts is the number of \defng{3-dim contingency arrays}
with $2$-dim marginals given by $(\al,\be,\ga)$, see \cite[Eq.~$(8)$]{PP17}.
The same equation implies that
\ts $\{g(\la,\mu,\nu)\}\in \FP$ \ts for partitions with bounded number of rows:
\ts $\ell,m,r = O(1)$, see
\cite{CDW12,PP17}.\footnote{When the encoding is in binary, both $\GapP$
and~$\FP$ claims remain true, but the argument now requires
\defng{Barvinok's algorithm} for counting integer points in polytopes of
bounded dimension in poly-time \cite{Bar93}. }

\smallskip

\subsection{Where to look}\label{ss:Kron-where-to-look}
There are several families of examples when Kronecker coefficients are
known.\footnote{For a quick guide to the literature, see e.g.\  the
\href{https://mathscinet.ams.org/mathscinet/search/publdoc.html?r=1&pg1=MR&s1=3650220}{MathSciNet review}
of~\cite{Bla17} by Christopher~Bowman.}
These include Blasiak's remarkable combinatorial interpretation of \ts
$\{g(\la,\mu,\nu)\}$, where \ts $\nu$ \ts is a \defng{hook} \cite{Bla17}I
(see also~\cite{Liu17}), and an $\NP$-complete combinatorial
interpretation for \defng{simplex-like triples} \.
$(\la,\mu,\nu)$ \. in~\cite[$\S$3]{IMW17}.
That contrasts with the following:

\begin{conj}\label{conj:Kron-triple}
Kronecker coefficients \ts $\{g(\la,\la,\la)\,:\,\la=\la'\}$ \ts are not in~$\ts\SPr$.
\end{conj}

This conjecture is  motivated by our inability to improve upon basic
bounds in this case: \. $1\le g(\la,\la,\la) \le f^\la$ \ts for all \ts $\la=\la'$. Here the lower bound
is proved in \cite{BB04}, while the upper bound follows from an observation \ts $g(\la,\mu,\nu) \le f^\la$ \ts
in \cite[$\S$3.1]{PPY19}.
Even for the \defng{staircase shape} \ts $\la= (k,k-1,\ldots,1)$ \ts these remain
the best known bounds. For the \defng{square shape}, we
recently showed a lower bound $\ts g(k^k,k^k,k^k)=e^{\Omega(\sqrt{k})}$  \cite{PP22},
which is very far from the upper bound $\ts g(k^k,k^k,k^k)=e^{O(k^2\log k)} \ts$ that is conjecturally tight.

{\small
\begin{rem}  \label{r:two-row}
When $\la=\mu$ and $\nu$ is a \defng{two-row partition}, we have the following formula for Kronecker coefficients:
$$(\ddag) \qquad
g\big(\la,\la,(n-k,k)\big) \, = \, \sum_{\al\ts \vdash \ts  k} \. \sum_{\be\ts \vdash \ts  n-k} \. \bigl(c^\la_{\al\be}\bigr)^2
\, - \. \sum_{\al\ts \vdash \ts k-1} \.  \sum_{\be\ts \vdash \ts  n-k+1} \. \bigl(c^\la_{\al\be}\bigr)^2,
$$
see \cite[Lem.~3.1]{PP14}.  Consider the inequality \. ``RHS of $(\ddag)\ge 0$''.
While different from the LPP and PPY inequalities in~$\S$\ref{ss:tab-LR},
it has the same flavor and is sufficiently similar to be out of reach
by direct combinatorial tools in full generality.\footnote{When \ts $\la_1\ge 2k-1$ \ts or \ts
$\ell(\la)\ge 2k-1$, a combinatorial description of \ts $g\big(\la,\mu,(n-k,k)\big)$ \ts is given in~\cite{BO05}.
}

In a special case when \ts $\la=(m^\ell)$ \ts is a rectangle, %and \ts $\nu=(m\ell-k,k)$,
the equality~$(\ddag)$  gives:
$$
%(\circledast) \quad
g\big(m^\ell,m^\ell,(n-k,k)\big) \, = \, [q^k]\tbinom{m+\ell}{\ell}_q \. - \. [q^{k-1}]\tbinom{m+\ell}{\ell}_q\ts,
$$
see \cite{PP13,PP14} (see also~\cite[Lem.~7.5]{Val14}).
In one direction, this proves unimodality of \defng{$q$-binomial coefficients}
(see~$\S$\ref{ss:seq-unimod}).  In the other direction, this highlights the
obstacle towards a natural ``combinatorial interpretation'' of Kronecker
coefficients, since proving this unimodality by an explicit injection is famously
difficult.\footnote{With Greta~Panova, we gave a cumbersome ``combinatorial interpretation''
for \ts $g\big(m^\ell,m^\ell,(n-k,k)\big)$ \ts in terms of certain trees, see \href{https://www.math.ucla.edu/~pak/hidden/papers/Panova_Porto_meeting.pdf}{these slides}, p.~9.
The proof is obtained by recursing O'Hara's $q$-binomial identity~\cite{OH90}.}
\end{rem}
}

\smallskip

\subsection{Taking a step back}\label{ss:Kron-reduced}  Let \ts $\al,\be,\ga$ \ts
be fixed integer partitions, not necessarily of the same size.  The
\defn{reduced Kronecker coefficients} \ts $\rg(\al,\be,\ga)$ \ts are defined as
\defng{stable limits} of Kronecker coefficients when a long first row is added:
$$(\circ) \qquad \ \
\rg(\al,\be,\ga) \. := \. \lim_{n\to \infty} \. g\bigl(\al[n],\be[n],\ga[n]\bigr),
$$
where \. $\al[n]:= (n-|\al|,\al_1,\al_2,\ldots)$ \. and \ts $n\ge |\al|+\al_1\ts$,
see \cite{Mur38,Mur56}.  Here the sequence \. $\big\{g\bigl(\al[n],\be[n],\ga[n]\bigr)\big\}$ \.
is weakly increasing and stabilizes already at \ts $n\ge |\al|+|\be|+|\ga|$,
see \cite{BOR11,Val99}.  In other words, the reduced Kronecker coefficients are a special
case of Kronecker coefficients for triples of shapes with a long first row.

The problem of finding a combinatorial interpretation of the reduced Kronecker coefficients
goes back to Murnaghan and Littlewood, and has been repeatedly asked over the past decades,
see e.g.\ \cite{Kir04,Man15}.  Part of the reason is that they
generalize the LR coefficients \. $\rg(\al,\be,\ga) = c^\al_{\be\ga}$ \. for \.
$|\al| =  |\be| + |\ga|\ts,$ see~\cite{Lit58}, and thus play an intermediate role.

\begin{conj}\label{conj:Kron-redced}
The reduced Kronecker coefficients \ts $\{\rg(\al,\be,\ga)\}$ \ts are not in~$\ts\SPr$.
\end{conj}

Depending on your point of view, this conjecture is either the harder to prove
or the easier to disprove, compared to Conjecture~\ref{conj:Kron-main}.

\smallskip

{\small
\subsection{Questioning the motivation}\label{ss:Kron-motivation}
There are several traditional reasons why one should continue pursuing the
multidecade project of finding a ``combinatorial interpretation''
for the Kronecker coefficients.  Let us refute the most important of
these, as we see them, one by one.

\smallskip

\nin
{\small $(1)$} \. Estimating the Kronecker coefficients
is enormously difficult, especially getting the lower bounds.  One might argue:

\smallskip

\begin{center}\begin{minipage}{11cm}%
Having a ``combinatorial interpretation'' would be a bonanza for getting good lower
bounds on \ts $g(\la,\mu,\nu)$.
\end{minipage}\end{center}

\smallskip

\nin
Sure, quite possibly so.  But given the poor state of affairs where in most cases
we do not have \emph{any} nontrivial lower bounds obtained by \emph{any method}
(cf.~\cite{BBS21,PP20a}),
shouldn't that be a reason to \defna{not believe} in the existence
of a ``combinatorial interpretation''?

\medskip

\nin
{\small $(2)$} \. Recall the \defn{saturation property} for LR
coefficients states that \.
$c^{k\la}_{k\mu\. k\nu} >0 \ \Rightarrow \  c^{\la}_{\mu\ts \nu} >0$,
for all integer \ts $k\ge 1$.  The original proof by Knutson and Tao
\cite{KT99} crucially relies on a variation of the LR~rule.
One might argue:\footnote{This argument appears frequently throughout
the literature in different contexts.  See e.g.~\cite{Kir04,Mul11} for many
conjectured variations and generalizations of
the saturation property.}

\smallskip

\begin{center}\begin{minipage}{11cm}%
Having a ``combinatorial interpretation'' could help proving some sort of
saturation property for the Kronecker coefficients.\end{minipage}\end{center}

\smallskip

\nin
No, it will not.  First, the saturation property
fails: \ts $g(1^2,1^2,1^2)=0$ \ts while
\ts $g(2^2,2^2,2^2)=1$.  Second, Mulmuley's natural weakening of the saturation
property in~\cite{Mul11} also fails, already for partitions
with at most two rows \cite{BOR09}.  Third, even for the reduced
Kronecker coefficients, the saturation property fails: \ts $\rg(1^5,1^5,3^2)=0$ \ts while
\ts $\rg(2^5,2^5,6^2)=12$ \ts  \cite{PP20}.\footnote{This was independently
conjectured by Kirillov \cite[Conj.~2.33]{Kir04} and Klyachko \cite[Conj.~6.2.6]{Kly04}. We disprove
the conjecture in~\cite{PP20}, by providing an infinite family of
counterexamples. It is, however, concerning how little computational effort
was made to check the conjecture which fails for relatively small partitions,
yet first disproved by a theoretical argument.  Could there be more conjectures
which are not sufficiently tested?  Perhaps, the ``minor but interesting''
\defng{Foulkes plethysm conjecture} \cite[$\S$3]{Sta00}
is worth another round of computer testing, see \cite{CIM17}, as its fate
may be similar to that of a stronger \defng{Stanley's conjecture} disproved
in~\cite{Pyl04}. }

\medskip

\nin
{\small $(3)$} \. The saturation property for the LR~coefficients easily
implies that their \defng{vanishing} can be decided
in poly-time using Linear Programming \cite{DM06,MNS12}
(see also~\cite{BI13} for a faster algorithm).
One might argue:

\smallskip

\begin{center}\begin{minipage}{11cm}%
Even without the saturation property, perhaps having a
``combinatorial interpretation'' could give a complete
description or possibly even an efficient algorithm
for the vanishing of the Kronecker coefficients.
\end{minipage}\end{center}

\smallskip

\nin
No, it will not (most likely).  Here we are assuming that a
``complete description'' includes both necessary and sufficient
conditions verifiable in poly-time, which puts this problem in \ts
$\NP \cap \coNP$.  We are also assuming that an
``efficient algorithm'' is being in~$\poly$.

Now, it is \emph{already known} that the vanishing problem
for the Kronecker coefficients is $\NP$-hard \cite{IMW17},
so an efficient algorithm implies \ts $\poly=\NP$.
Similarly, if an $\NP$-hard problem is in \ts $\NP \cap \coNP$,
then \ts $\NP = \coNP$.  So unless one expects a major breakthrough
in Computational Complexity, this approach will not work.

\medskip

\nin
{\small $(4)$} \. There are obvious social aspects to problem
solving.  This is an old open problem, perhaps the oldest in
the area.  Famous people worked on it and reiterated its
importance.  One might proclaim:

\smallskip

\begin{center}\begin{minipage}{11cm}%
The victor gets the spoils.
\end{minipage}\end{center}

\smallskip

\nin
Absolutely!  But shouldn't then proving \defna{nonexistence} of a combinatorial
interpretation be just as much a ``victory'' as finding one?\footnote{For
more on this argument, see our
blog post  ``\emph{What if they are all wrong?}'' (Dec.~10, 2020),
available~at \ts \href{https://wp.me/p211iQ-uT}{wp.me/p211iQ-uT}
}

\medskip

\nin
{\small $(5)$} \. Finally, the \defng{intellectual curiosity} \ts is not to be discounted.
The problem is clearly attractive and has led to a lot of nice results
even in small special cases.  One might reasonably argue:

\smallskip

\begin{center}\begin{minipage}{11cm}%
While we may never be able to resolve the problem completely, % working on it is fun and
many interesting results might get established along the way.
\end{minipage}\end{center}

\smallskip

\nin
Sure, of course.  But again, why limit yourselves to working only in the
positive direction?\footnote{As the renowned 19th century British philosopher
the Cheshire Cat once said, ``\emph{it doesn't much matter which
way you go}'', you will definitely get \underline{somewhere}
``\emph{if only you walk long enough}'' \cite[Ch.~VI]{Car65}.}
}

%%%%%%%%%%%%%%%%%%%%%%%%%%%%%%%%%%%%%%%%%%%%%%%%%%%%%%%%%%%%%%%%%%%%%%%%%%
%%%%%%%%%%%%%%%%%%%%%%%%%%%%%%%%%%%%%%%%%%%%%%%%%%%%%%%%%%%%%%%%%%%%%%%%%%
%%%%%%%%%%%%%%%%%%%%%%%%%%%%%%%%%%%%%%%%%%%%%%%%%%%%%%%%%%%%%%%%%%%%%%%%%%
%%%%%%%%%%%%%%%%%%%%%%%%%%%%%%%%%%%%%%%%%%%%%%%%%%%%%%%%%%%%%%%%%%%%%%%%%%
%%%%%%%%%%%%%%%%%%%%%%%%%%%%%%%%%%%%%%%%%%%%%%%%%%%%%%%%%%%%%%%%%%%%%%%%%%
%%%%%%%%%%%%%%%%%%%%%%%%%%%%%%%%%%%%%%%%%%%%%%%%%%%%%%%%%%%%%%%%%%%%%%%%%%
%%%%%%%%%%%%%%%%%%%%%%%%%%%%%%%%%%%%%%%%%%%%%%%%%%%%%%%%%%%%%%%%%%%%%%%%%%
%%%%%%%%%%%%%%%%%%%%%%%%%%%%%%%%%%%%%%%%%%%%%%%%%%%%%%%%%%%%%%%%%%%%%%%%%%

\bigskip

\section{Schubert coefficients}\label{s:Schubert}
For notation and standard results on Schubert polynomials, see \cite{Mac91}
and \cite{Man01}.  An accessible introduction to combinatorics of reduced
factorizations is given in~\cite{Gar02}, and the geometric background is
given in \cite[$\S$10]{Ful97}.  A friendly modern introduction is given
in~\cite{Gil19}.  The presentation below is self-contained but omits
the background.

\subsection{RC-graphs} \label{ss:Schu-def}
For a permutation \ts $w\in S_n$\ts, denote by $\RC(w)$ is the set
of \ts \defn{RC-graphs} (also called \defn{pipe dreams}), defined as tilings
of a staircase shape with \defn{crosses} and \defn{elbows} as in Figure~\ref{f:RC}
which satisfy two conditions:

{\small $(i)$} \, curves start in row $k$ on the left and end in column $w(k)$ on top, for all \ts $1\le k \le n$, and

{\small $(ii)$} \. no two curves intersect twice.

\nin
It follows from these conditions that every \ts $G \in \RC(w)$ \ts
has exactly \ts $\inv(w)$ \ts crosses.

For \ts $G \in \RC(w)$, denote by \ts $\bx^G$ \ts the product of \ts $x_i$'s
over all crosses \ts $(i,j)\in G$, see Figure~\ref{f:RC}.
Define
the \defn{Schubert polynomial} \ts $\Sch_w\in \nn[x_1,x_2,\ldots]$ \ts
% corresponding to \ts $w\in S_n$ \ts
as\footnote{The usual definition of Schubert polynomials is
algebraic, making this definition a crucial result in
the area, see~\cite{BB93}. Let us mention other combinatorial
models of Schubert polynomials: \defng{compatible sequences} \cite{BJS93}
and \defng{bumpless pipe dreams} \cite{LLS21}.  See also~\cite{GH21} for
the bijections between them.
}
$$
\Sch_w(\bx) \, := \, \sum_{G\ts\in\ts \RC(w)} \. \bx^G\ts.
$$
For example, \. $\Sch_{1432} = x_1x_2x_3+x_1^2x_3+ x_1x_2^2+x_2^2x_3+x_1^2x_2$ \.
as in the figure.  Note that Schubert polynomials stabilize when fixed
points are added at the end, e.g.\ \ts $\Sch_{1432} = \Sch_{14325}$.  Thus we
can pass to the limit \ts $\Sch_w$\ts, where \ts $w\in S_\infty$ \ts is a
permutation \ts $\nn \to \nn$ \ts with finitely many nonfixed points.

% Denote \ts $rc(w,\al) := [x^\al]\Sch_w$.

\begin{figure}[hbt]
\begin{center}
	\includegraphics[height=2.5cm]{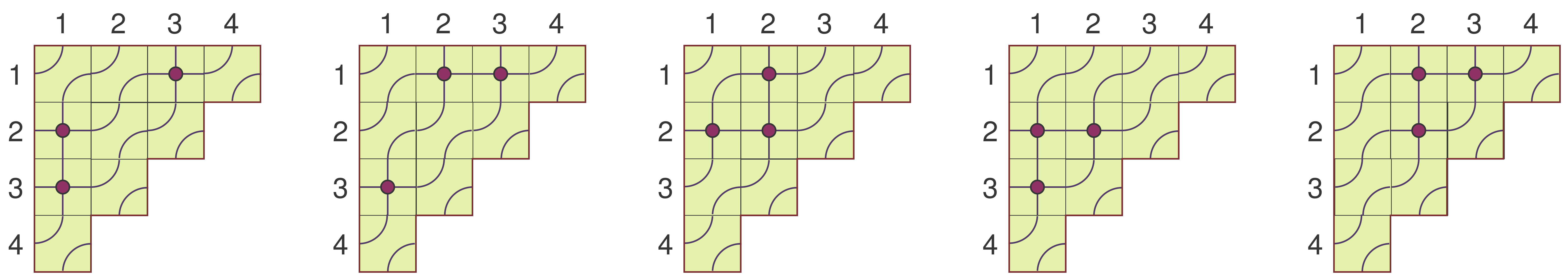}
%\hskip-6.1cm
%\epsfig{file=RC-graphs.eps,height=2.7cm}
\caption{Graphs in \ts $\RC(1432)$ \ts corresponding to the monomials \vfill
$x_1x_2x_3$\ts, \. $x_1^2x_3$\ts, \. $x_1x_2^2$\ts, \. $x_2^2x_3$ \. and \. $x_1^2x_2$\ts,
in this order.}
\end{center}
\label{f:RC}
\end{figure}

Polynomials \ts $\{\Sch_w \. : \. w\in S_\infty\}$ \ts
are known to form a basis in the ring \ts $\zz[x_1,x_2,\ldots]\ts.$
\defn{Schubert coefficients} \ts are defined as structure constants:
$$
\Sch_u \cdot \Sch_v \, = \, \sum_{w \in S_\infty} \. c^w_{uv} \. \Sch_w\ts.
$$
It is known that \ts $c^w_{uv} \in \nn$ \ts for all \ts $u,v,w\in S_\infty$\ts.

\begin{conj}\label{conj:Schu-main}
Schubert coefficients \. $\{c^w_{uv} \. : \. u,v,w\in S_\infty\}$ \ts
are not in $\ts \SPr$.
\end{conj}

For \defn{Grassmannian permutations} (permutations with one descent),
Schubert polynomial coincide with Schur polynomials, so Schubert coefficients
generalize LR~coefficients.  This is a starting point of a number of further
generalizations, see \cite{Knu16,Kog01,MPP14}.

\smallskip

\subsection{Schubert--Kostka numbers} \label{ss:Schu-Kostka}
For a permutation \ts $w\in S_n$ \ts and an integer vector \ts $a\in \nn^n$,
the \defn{Schubert--Kostka number} \ts $K_{w\ts a} := [\bx^a] \Sch_w$ \ts is
the coefficient of a monomial in the Schubert polynomial.  By definition,
\ts $\{K_{w\ts a}\} \in \SP$.

\begin{prop}[{\rm Morales\footnote{Alejandro H.~Morales, personal communication (2016).}}{}]
Schubert coefficients
$\big\{c^w_{uv}\.: \.  u,v,w\in S_\infty\big\}$ \ts are in \ts $\GapPr$.
\end{prop}

\begin{proof} Let \ts $\si\in S_n$ \ts and let \ts $\rho_n:=(n-1,\ldots,1,0)\in \nn^n$.
Define
$$\Om(\si) \. := \. \big\{(a,b,c)\in (\nn^n)^3\.: \.  a+b+c=\si\rho_n\big\}.
$$
It was shown by Postnikov and Stanley in \cite[Cor.~17.13]{PS09}, that
$$c^w_{uv} \, = \, \sum_{\si\in S_n} \, \sum_{(a,b,c)\.\in\. \Om(\si)} \,
\sign(\si) \. K_{u\ts a} \. K_{v\ts b}\. K_{w\ts c}\..
$$
Separating positive and negative signs shows that \ts $\{c^w_{uv}\}\in \SP-\SP$,
as desired.\footnote{One can also use \defng{Monk's rule} (see e.g.~\cite{Gil19}),
to obtain the same result. We thank Alex Yong for this observation. }
\end{proof}

Following \cite[$\S$17]{PS09}, define the
\defn{Schubert--Kostka matrix} \ts $\wh\bK=\bigl(K_{u\ts a}\bigr)$,
which naturally generalizes the Kostka matrix \ts $\bK=\bigl(K_{\la\mu}\bigr)$.
Similarly, define the inverse \ts ${\wh\bK}^{-1}=\bigl(K^{-1}_{a \ts u}\bigr)$ \ts which
generalizes the inverse Kostka matrix \ts $\bK^{-1}=\bigl(K^{-1}_{\la\mu}\bigr)$.

\begin{prop}%[{\rm cf.\ \cite[Prop.~17.3]{PS09}}{}]
The inverse Schubert--Kostka numbers
$\big\{K^{-1}_{a \ts u}\.: \. a\in \nn^\infty, \. u\in S_\infty\big\}$ \ts are in \ts $\GapPr$.
\end{prop}

Similarly to the previous proposition,
the result follows directly from the identity in \cite[Prop.~17.3]{PS09}.
One can ask if the absolute values $\big\{\big|K^{-1}_{a \ts u}\big|\big\}$ \ts
are in~$\SP$.  Conjecture~\ref{conj:tab-inverse-Kostka} implies that the answer is negative,
but perhaps this more general problem is an easier place to start.

\smallskip

{\small

\subsection{Has the problem been resolved?}\label{ss:Schu-solved}
There are two issues around Conjecture~\ref{conj:Schu-main} worth mentioning,
as both, in different ways, suggest that the conjecture has already been
resolved in the negative (i.e.\ a combinatorial interpretation has already
been found).

\smallskip

\underline{First}, Izzet Coskun in~\cite{Cos+} claimed to have completely
resolved the problem of finding combinatorial interpretations for
Schubert coefficients\footnote{This paper is undated, but cited already in~\cite{CV09}.}
using the technology of \defng{Mondrian tableaux}.\footnote{These are aptly
named after a Dutch painter Piet Mondrian (1872--1944), who developed his
signature style in the ``tableau'' series in 1920s, and did not live to see his work's
influence in Schubert calculus.}
Earlier, he used Mondrian tableaux to give a combinatorial interpretation for
\defng{step-two Schubert coefficients} (corresponding to permutations with at
most two descents) in~\cite{Cos09} extending Vakil's earlier work~\cite{Vak06},
see a discussion in~\cite{CV09}.

Unfortunately, paper~\cite{Cos+} has not been peer reviewed and has been
largely ignored by the community (see \cite{Bil21} for a notable
exception).\footnote{We are baffled by the
author's continuing claim that the paper is ``currently under revision''.
We are equally baffled by unwillingness of the experts in the area to go on record stating
whether this work is incorrect, and to provide a counterexample if available.  }
We should mention that the state of art recent work \cite{KZ17} gives a tiling
combinatorial interpretation for the \defng{step-three Schubert coefficients}.
It seems, we are nowhere close to resolving Conjecture~\ref{conj:Schu-main}
in full generality.

\smallskip

\underline{Second}, Sara Billey suggested in~\cite{Bil21}, that Schubert coefficients
\emph{already have} a ``combinatorial interpretation'', since by definition they
are equal to the number of irreducible components in certain intersections of
three Schubert varieties,\footnote{Equivalently,
this is the number of points in a generic intersection of three Schubert varieties.}
and thus ``they already count something''.  Can one create a $\SP$ function out of this
definition?
% Let us address this issue upfront since it is rather standard.

While it is true that Schubert coefficients count the number of certain points in~$\cc^n$,
these points are not necessarily rational.  In fact, they are usually roots of a large
system of rational polynomials. On the other hand, Billey and Vakil prove in~\cite[$\S$4]{BV08},
that there are some remarkable pathologies for these intersections related to realizability and
stretchability of pseudoline arrangements.  It follows from the
\defng{Mn\"ev universality theorem}~\cite{Mnev88} (see also~\cite{Shor91}),
that these problems are \defng{\textsf{$\exists\rr$}-complete}.\footnote{See e.g.~\cite{Scha10}
for a computational complexity overview of the \defng{existential theory of the reals} (\textsf{$\exists\rr$}),
and connections to Mn\"ev's theorem.}

The \ts \textsf{$\exists\rr$} \ts complexity class is in \ts {\sf PSPACE} \ts
and not expected to have polynomial size verifiers.  This suggests
that in the worst case, Billey's approach needs a superexponential precision with
which one would want to compute the intersection points (i.e.\ the floating point
computation needs superpolynomially many digits), implying that it is unlikely
that there exists a poly-time verifier in this case.

The issue is not too different from that of Kronecker
coefficients, which can be written as \.
 $g(\la,\mu,\nu)=\dim\big(\SS^\la \otimes \SS^\mu \otimes \SS^\nu\big)^{S_n}$.
 Thus, one can argue that they \emph{count basis vectors} in this space of invariants.
 But since the Kronecker coefficients can be exponential in~$n$, see e.g.~\cite{PPY19},
 how would one present such a basis so it can be verified (or at least \emph{read})
 in poly-time?\footnote{We thank Greta Panova for suggesting this comparison (personal
 communication).}
}

\smallskip

{\small

\subsection{Elements of style}\label{ss:Schu-style}
There are both cultural and mathematical reasons why the type of issues
we discuss in~$\S$\ref{ss:Kron-motivation} do not apply to the study
of Schubert coefficients,\footnote{For the generalized saturation
problem in this case, see~\cite{ARY13} (see also~\cite[$\S$7]{Buch02}).}
so let us quickly address the differences.

Schubert polynomials were originally introduced by Lascoux and
Sch\"{u}tzenberger in 1980s in the geometric context, see \cite{Las95}.
They remain deeply entrenched in the area which offers a melange of
tools nonexistent on the combinatorial and representation theoretic side.
Arguably, this resulted in a deeper study with a long series of
achievements, too long to outline in this survey
(see e.g.\ \cite{Knu22}).

By now, the Schubert theory outgrew the ``combinatorial
interpretation issue'' as its raison d'\^{e}tre.  Arguably, the area was
always about developing the theory to understand the geometry of flag varieties
from a combinatorial point of view (cf.~\cite{LS85}).

\smallskip

This brings us to the curious case of \defng{puzzles}, a type
of tilings on a triangular lattice with labeled edges (see e.g.\
\cite{Knu16}).\footnote{For other models and combinatorial interpretations
in this context, see an overview in \cite[$\S$1.3]{RYY22}.}  Aesthetically
pleasing for sure, is there some additional value such combinatorial
interpretation bring to the study?

To us, the answer is a mixture.  On the one hand, when the set of tiles
is finite they clearly define a combinatorial
interpretation.\footnote{Occasional infinite sets tiles can be reduced
to a finite set by breaking them apart and adding colors.} Having a nice
or highly symmetric set of tiles can be convenient to prove structural
results about vanishing and other properties of the numbers, cf.~\cite{KT99}.

On the other hand, there is nothing surprising in the labeled tiles
from the computational point of view.  By adding new labels, one can
always break them into single triangles, which can then be viewed
as a variation on \defn{Wang tiles} \ts \cite{Wang61}.  The latter
can also be broken into triangles leading to the same model of
computation.  Wang's celebrated insight is that such tilings are
just as powerful as Turing machines.  In fact, the example of
Wang tilings with the boundary was one of Levin's six original
$\NP$-complete problems \cite{Lev73}.

And yet, there is a clear computational difference between
various types of increasing tableaux (such as
\defng{set-valued tableaux}) vs.\ puzzles.  While the former
need some memory to check which numbers appear, the latter can
be verified by checking only local label conditions (think of
\defng{finite state automata}), making the puzzles (slightly)
less powerful as a computational model.\footnote{Note that
giving a combinatorial interpretation in a weaker computational
model is a \emph{stronger} result.}  Part of the weakness here comes
from the fixed triangular region space requirement, further constraining
the puzzle tiling model.\footnote{In \cite{GP14}, we exploited this model by
constructing sets of Wang tiling which simulated several classical
sequences such as the Catalan numbers. Our model uses a large
number of square tiles, just like \cite{KZ17} which uses 151 triangular
pieces.}
}

\bigskip

%%%%%%%%%%%%%%%%%%%%%%%%%%%%%%%%%%%%%%%%%%%%%%%%%%%%%%%%%%%%%%%%%%%%%%%%%%
%%%%%%%%%%%%%%%%%%%%%%%%%%%%%%%%%%%%%%%%%%%%%%%%%%%%%%%%%%%%%%%%%%%%%%%%%%
%%%%%%%%%%%%%%%%%%%%%%%%%%%%%%%%%%%%%%%%%%%%%%%%%%%%%%%%%%%%%%%%%%%%%%%%%%
%%%%%%%%%%%%%%%%%%%%%%%%%%%%%%%%%%%%%%%%%%%%%%%%%%%%%%%%%%%%%%%%%%%%%%%%%%
%%%%%%%%%%%%%%%%%%%%%%%%%%%%%%%%%%%%%%%%%%%%%%%%%%%%%%%%%%%%%%%%%%%%%%%%%%
%%%%%%%%%%%%%%%%%%%%%%%%%%%%%%%%%%%%%%%%%%%%%%%%%%%%%%%%%%%%%%%%%%%%%%%%%%
%%%%%%%%%%%%%%%%%%%%%%%%%%%%%%%%%%%%%%%%%%%%%%%%%%%%%%%%%%%%%%%%%%%%%%%%%%
%%%%%%%%%%%%%%%%%%%%%%%%%%%%%%%%%%%%%%%%%%%%%%%%%%%%%%%%%%%%%%%%%%%%%%%%%%

\section{The magic of the symmetric group}\label{s:Sym}

Before we can explain why, in our view, so many numbers in Algebraic Combinatorics
are not in $\SP$, % (including Kronecker, plethysm and Schubert coefficients),
we need to deconstruct the wonderful world of the symmetric group.  Only after
reducing all the great features to just one, we can fully appreciate its power
as well as its limitation.

\smallskip

\subsection{Pruning down the list}\label{ss:Sym-basic}
We start with representation theoretic point of view, demystifying
some of the magic.  All these results are routine and well known,
so we restrict ourselves to quick pointers to the literature.

\medskip

\nin
{\small $(a)$} \, \defn{$f^\la\. | \. n!$}  \quad
This is non-specific to $S_n$, since the dimension \ts $\chi(1)$ \ts of an
irreducible character of a finite group~$G$ always divides the order~$|G|$,
see e.g.\ \cite[$\S$6.5]{Ser77}.

\smallskip

\nin
{\small $(b)$} \, \defn{$\{f^\la\} \in \SPr.$} \quad
This is a consequence of \ts $\chi^\la\downarrow^{S_n}_{S_{n-1}}$ \ts having a \defng{simple
spectrum}, i.e.\ multiplicities at most one. Irreducibles in the restriction
(given by the \defng{branching rule}) correspond to removing a corner from
Young diagram~$\la$.  Iterating the restrictions along the subgroup chain \ts
$S_n \sss S_{n-1} \sss \ldots \sss S_1$ \ts  implies
that \ts $f^\la=|\SYT(\la)|$,  cf.\ \cite{OV96}.

\smallskip

\nin
{\small $(c)$} \, \defn{$\sum_{\la\vdash n} \. (f^\la)^2=n!$} \quad This is non-specific to $S_n$.
\defng{Burnside's identity} \. $\sum \chi(1)^2 = |G|$ \. holds for all finite groups,
cf. Remark~\ref{rem:char-identities}.

\smallskip

\nin
{\small $(d)$} \, \defn{$\chi^\la(\mu)\in \zz$.} \quad This is follows from a property of conjugacy classes
of~$S_n$, that \ts $\si^k\in [\mu]$ \ts for every \ts $\si\in [\mu]$ \ts and \ts $(k,\text{ord}(\si))=1$,
see e.g.\ \cite[$\S$13.1]{Ser77}. See also $\S$\ref{ss:char-sums} for references to row and column sums
of character tables of general finite groups.

\smallskip

\nin
{\small $(e)$} \, \defn{$\{$Irreps of $S_n\}$ \. $\longleftrightarrow$ \. \{Irreps of $\GL(N)$\}.} \quad
This is a direct consequence of the \defng{Schur--Weyl duality}: \. $S_n \times \GL(N)$ \ts act
on \ts $(\cc^N)^{\otimes n}$ \ts and the action has simple spectrum \ts $\SS^\la \otimes V_\la$ \ts of
tensor products of corresponding irreps, see e.g.\ \cite[$\S$6.2]{FH99}.

\smallskip

\nin
{\small $(f)$} \, \defn{$\{K_{\la\mu}\} \in \SPr$.} \quad  This is a direct consequence of the
\defng{highest weight theory} for \ts $\GL(N)$.  Indeed, the induced product for $S_n$ reps corresponds to
the tensor product of $\GL(N)$ reps, via~$(e)$.  Now use the definition \.
$K_{\la\mu} = \big\<\chi^\la, 1\uparrow_{S_{\mu_1}\times S_{\mu_2}\times \cdots}^{S_n}\big\>$.

\smallskip

\nin
{\small $(g)$} \, \defn{$\{f^\la\} \in \FPr$.} \quad This extension of {\small $(b)$}
is a consequence
of \ts dimensions \ts $\dim(V_\la)$ \ts of $\GL(N)$ \ts irreps given by a product.
The latter is computed as a ratio of two \defng{Vandermonde determinants}.
The product formula for \ts $f^\la=\dim(\SS^\la)$ \ts follows by taking the limit,
see e.g.\ \cite[$\S$7.21]{Sta99}.\footnote{If you are unsurprised by this result,
try answering if dimensions of \ts Sp$(2n,q)$ \ts irreps are in~$\FP$.  What about
other simple groups of Lie type?  Let me know what you figure out.  }

\medskip

\subsection{Hook-length formula}\label{ss:Sym-accidents}
The results above are both fundamental in the area and have conceptual proofs
explaining away some of the magic.  The \defn{hook-length formula} (HLF) (see~$\S$\ref{ss:tab-SYT})
may also seem fundamental at first, until one realizes that all proofs are a byproduct of
calculations highly specific to symmetric groups.  Here is
a quick overview of the proofs:

\smallskip

$\circ$ \ Direct cancelation proof via the product formula in
$\S$\ref{ss:Sym-basic}{\small \ts $(g)$}, see the original proof

\hskip.45cm in~\cite{FRT54}, see also \cite[$\S$7.21]{Sta99}.

\smallskip

$\circ$ \ \defng{NPS bijection} \cite{NPS97}, an intricate argument seemingly based
on the \defng{jeu-de-taquin},

\hskip.45cm but neither using its properties nor originally
motivated by it  (see e.g.\ \cite[$\S$3.10]{Sag01}).

\smallskip

$\circ$ \ \defng{Hillman--Grassl bijection} \cite{HG76}, see also~\cite{Krat99} (plus a limit argument, see e.g.~\cite{Pak01}),

\hskip.45cm an elegant bijection that is equivalent to RSK, see \cite{Gan81}, \cite[$\S$6.5]{PV10} and \cite[$\S$5]{VW83}.

\smallskip

$\circ$ \ \defng{Geometric bijection} (plus a limit argument) in \cite{Pak01}, obtained as a deconstruction

\hskip.45cm of RSK via
local transformations (cf.\ \cite{Hop14}).

\smallskip

$\circ$ \ \defng{Lagrange interpolation} inductive arguments, see \cite{Ban08,GN04,Kir92,Ver92}.  These

\hskip.45cm are elementary analytic proofs which require the most calculation.

\smallskip

$\circ$ \ \defng{GNW hook walk} \cite{GNW79}, a ingenuous argument which embeds the Lagrange inter-

\hskip.45cm polation inductive argument in probabilistic disguise as shown by Kerov~\cite{Ker93}.

\smallskip

$\circ$ \ Inductive proof via bijection of the \defng{branching rule} in~\cite{CKP11}, obtained as a

\hskip.45cm combinatorial deconstruction of the GNW hook walk argument, cf.\ \cite{Zei84}.

\smallskip

$\circ$ \ \defng{Naruse HLF} \ts \cite{MPP17,MPP18a}, an advanced generalization of the HLF which
easily

\hskip.45cm
implies it (see also \cite{Kon20}).  All currently known proofs are technically involved.

\smallskip

\nin
Neither of these proofs explains the HLF on a deeper level, since each of them either has
a substantive computational part, or outsources the computation to RSK and its relatives.

\medskip

\subsection{Robinson--Schensted--Knuth correspondence}\label{ss:Sym-RSK}
We argue that the \defn{RSK} \ts is the one true miracle in the area.
Let us count some of the ways it emerges, in historical order:

\smallskip

$\circ$ \ The (original) \defng{Robinson--Schensted algorithm}, later extended by Knuth~\cite{Knu70}.

\smallskip

$\circ$ \ The \defng{Burge correspondence} \cite{Bur74} (see also \cite[$\S$A.4.1]{Ful97}).

\smallskip

$\circ$ \ The \defng{Hillman--Grassl correspondence} \cite{HG76} (see above).

\smallskip

$\circ$ \ Sch\"{u}tzenberger's \defng{jeu-de-taquin} \cite{Sch77} (see also \cite[$\S$3.7]{Sag01}).

\smallskip

$\circ$ \ Viennot's \defng{geometric construction} \cite{Vie77} (see also \cite[$\S$3.6]{Sag01}).

\smallskip

$\circ$ \ Quantum version of D\'esarm\'enien's \defng{straightening algorithm} \cite{Des80} given in \cite{LT96}.

\smallskip

$\circ$ \ Steinberg's \defng{unipotent variety} approach \cite{Ste88} (see also \cite{vL00}).

\smallskip

$\circ$ \ Fomin's \defng{growth diagrams} approach \cite{Fom95}  (see also \cite[$\S$7.13]{Sta99}).

\smallskip

$\circ$ \ Benkart--Sottile--Stroomer \defng{tableau switching} \cite{BSS96}.

\smallskip

$\circ$ \ Our \defng{geometric bijection} \cite{Pak01} (see above).

\smallskip

$\circ$ \ Lascoux's \defng{double crystal graphs} version \cite{Las03}.

\smallskip

$\circ$ \ The \defng{octahedral map} \cite{KTW04,DK05b}
(see also \cite{DK08,HK06,PV10,Spe07}).

\smallskip

\nin
There are many more versions of RSK --- these are just the ones we find most
interesting.  There are also numerous extensions, generalizations and applications
both in and outside of the area.  Let us now clarify some items on this list.

Some of the bijections above, such as the geometric construction by Viennot
and the double crystal graphs by Lascoux,
are restatements of RSK into a different language.
Several others require a serious proof that they coincide with RSK, e.g.\
the Hillman--Grassl correspondence and the jeu-de-taquin.  In a difficult case,
D\'esarm\'enien's straightening algorithm is similar but not equal to RSK, but the
quantum versions do coincide.\footnote{This similarity puzzled Gian-Carlo Rota,
who asked for years to find
a connection, see e.g.\ \cite{BT00}.  The mystery was eventually resolved by
Leclerc and Thibon in \cite{LT96}, who noticed that sometimes the leading
coefficient in the straightening is ``incorrect'' because of the cancellation
which disappear when \ts $q \ne 1$.  We recall Rota's great joy upon learning
of this discovery.  See also \cite{GL21} for another closely related
connection. }

In \cite{PV10}, we set up a technology which allow one to prove that various
maps in the area are in fact linear time reducible to and from the RSK.
This allows us to put the \defng{Sch\"utzenberger involution}, the tableau switching
and the octahedral recurrence on the list.

\smallskip

Now is the time to ask a key question:  \defna{Given so many different approaches to so many
different problems, why do they result in the same bijection?}  There is a formal mathematical reason to
explain coincidences between any two RSK appearances.  The explanations could be algebraic or
combinatorial, but neither would resolve a question given the multitude of instances.

The answer is simple and unambiguous.  It does not really matter what is the deep reason
behind the RSK, whether it is the \defng{highest weight theory}, the straightening algorithm or
something else.  What matters is that each RSK appearance is a shadow of one
fundamental result that is yet to be formalized. This suggests we treat RSK as
\defna{the one true miracle} \ts which causes much of what we consider magical
about the symmetric group.

\smallskip

\subsection{LR~rule}\label{ss:Sym-LR}
Consider the \defng{Littlewood--Richardson coefficients}
and its many combinatorial interpretations (see e.g.\ \cite{vL01} for an extensive albeit dated survey):

\smallskip

$\di$ \ The \defng{original LR rule}: \.  $c^\la_{\mu\nu} = |\LR(\la/\mu,\nu)|$, see \cite{LR34}.

\smallskip

$\di$ \ The \defng{LR variation}: \. $c^\la_{\mu\nu} = |\LR(\mu\circ \nu,\la)|$, see e.g.\ \cite{RW84}.

\smallskip

$\di$ \ James--Peel \defng{pictures}  \cite{JP79}, see also~\cite{CS84,Zel81}.

\smallskip

$\di$ \ Gelfand--Zelevinsky interpretation using \defng{Gelfand--Tsetlin patterns}  \cite{GZ85}.

\smallskip

$\di$ \ Leaves of the \defng{Lascoux--Sch\"utzenberger % {\em (LS)}
tree} \cite{LS85}.

\smallskip

$\di$ \ Kirillov--Reshetikhin \defng{rigged configurations} \cite{KR88} (see also \cite{KSS02}).

\smallskip

$\di$ \ Berenstein--Zelevinsky \defng{triangles} \cite{BZ92}.  %{\em (BZ)}

\smallskip

$\di$ \ Fomin--Greene \defng{good maps} \cite{FG93}.

\smallskip

$\di$ \ Nakashima's interpretation using \defng{crystal graphs} \cite{Nak93} (see also \cite[$\S$9]{BS17}).

\smallskip

$\di$ \ Littelmann's \defng{paths} \cite{Lit94}.

\smallskip

$\di$ \ Knutson--Tao \defng{hives} \cite{KT99}, see also~\cite{GP00}.

\smallskip

$\di$ \ Kogan's interpretation using \defng{RC-graphs} \cite{Kog01}.

\smallskip

$\di$ \ Buch's \defng{set-valued tableaux} \cite{Buch02}.

\smallskip

$\di$ \ Knutson--Tao--Woodward \defng{puzzles} \cite{KTW04}.

\smallskip

$\di$ \ Danilov--Koshevoy \defng{arrays} \cite{DK05a}.

\smallskip

$\di$ \ Vakil's \defng{chessgames} \cite{Vak06}.

\smallskip

$\di$ \ Thomas--Yong \defng{$S_3$-symmetric LR rule} \cite{TY08}.

\smallskip

$\di$ \ Purbhoo's \defng{mosaics} \cite{Pur08} (see also~\cite{Zin09}).

\smallskip

$\di$ \ Coskun's \defng{Mondrian tableaux}  \cite{Cos09}.

\smallskip

$\di$ \ Nadeau's \defng{fully packed loop configurations in a triangle} \cite{Nad13} (see also~\cite{FN15}).

\smallskip

\nin
The list above is so lengthy, it is worth examining carefully. Most
of these LR~rules are byproducts of (often but not always, successful)
efforts to find a combinatorial interpretation of more general numbers.
Some of these are closely related to each other, while others seem
quite different, both visually and mathematically.

Now, on the surface the RSK is nowhere in the picture.  We already
mentioned \cite{Ker84,Whi81,Zel81} which make the connection explicit,
but here is a quick outline of how to get the original LR~rule.

Start with a skew shape \ts $\la/\mu$,
and run the jeu-de-taquin on \ts $\SYT(\la/\mu)$.  The number of
times each \ts $A\in \SYT(\nu)$ \ts is the image \ts $\jdt\bigl(\SYT(\la/\mu)\bigr)$ \ts
is exactly~$c^\la_{\mu\nu}$, for all $A$ and~$\nu$.\footnote{Nothing in this
claim is obvious: from the fact that jeu-de-taquin is well defined (independent
on the order of moves), to the fact that preimage sizes are equinumerous
and equal to~$c^\la_{\mu\nu}$, see e.g.\ \cite[App.~1 to Ch.~7]{Sta99}.}
Taking the lex-smallest
such \ts $A$ (obtained by reading squares of~$\nu$ left-to-right as commonly done), gives
a combinatorial interpretation \ts $c^\la_{\mu\nu}=\big|\big\{B\in \SYT(\la/\mu)\.:\.\jdt(B)=A\big\}\big|$.

While the above combinatorial interpretation suffices to show that $\{c^\la_{\mu\nu}\}\in \SP$,
minor adjustments can be made to beautify the resulting rule.  First,
note that preimage of squares in every row of $A$ cannot be in the same column of
\ts $B=\jdt^{-1}(A)$, so we can relabel $A$ by placing $i$ in the squares of $i$-th row.
We get a unique \ts $A_0\in \SSYT(\nu,\nu)$.  Now, the preimage in \ts $\LR(\la/\mu,\nu):=\jdt^{-1}(A_0)\subseteq
\SSYT(\la/\mu,\nu)$ \ts can be described using the \defng{ballot condition},\footnote{The
ballot condition is often called ``lattice'' or ``Yamanouchi'' depending on the
context and how the tableau is being read. The differences between these are
inconsequential.}  giving the usual
description of \ts $\LR(\la/\mu,\nu)$.\footnote{The
LR~variation can be obtained in the same way, by doing
jeu-de-taquin to \ts $\SYT(\mu\circ \nu)$ \ts and looking at a preimage
of \ts $A\in \SYT(\la)$.  The difference is that after relabeling, all tableaux in \ts
$\jdt^{-1}(A_0)\subseteq\SSYT(\mu\circ\nu,\la)$ have the same filling of~$\nu$, which can
then be omitted from the description.  The result if the subset of $\SSYT(\mu,\la/\nu)$ with
a $\nu$-ballot condition.
}

\smallskip

We emphasize that RSK is omnipresent in the LR~study \cite{vL01}.  It helps to
approach this historically (we will try to be brief).
The LR~rule was introduced in~\cite{LR34}.
Soon after, Robinson introduced the first version of RSK in his
(incomplete) effort to prove the original LR~rule \cite{Rob38}.
As James describes in~\cite{Jam87}, the LR~rule ``is much harder
to prove than was at first suspected.''\footnote{See also his famous
``get men on the moon'' sentence \cite[p.~117]{Jam87}
(also quoted in~\cite[p.~147]{Mac95}).}

Macdonald~\cite[$\S$I.9]{Mac95} credits Sch\"utzenberger \cite{Sch77}
and Thomas's thesis \cite{Tho74} with first complete proofs.  Both were
obtained in the context of RSK and its relatives.  Since then,
many proofs of the LR~rule were discovered, too many
to be cited here, all related to RSK relatives (ibid.)  We single out
the proof in \cite{KTW04}  based on the associativity property given
by the octahedral map, and a geometry inspired proof in~\cite{BKT04}
based on the jeu-de-taquin.
Note that both use properties of RSK relatives.\footnote{
Let us also mention the \emph{involutive proofs} of variations on the
LR~rule: \cite{BZ88,Gas98,RS98,Ste02}.
As far as we can tell, these are essentially the same ``verification type proofs''
stated in different languages. }

Given that the role of RSK is often invisible without carefully
examining the proofs of the LR~interpretations above, this brings
us to the following question:  \defna{Are these
``combinatorial interpretations'' of LR~coefficients
equivalent in some formal sense?}

{\small
\smallskip

\subsection{Little boxes all the same}{}\hskip-.2cm{}\footnote{See
 \emph{Malvina Reynolds Sings the Truth}, Columbia Records, 1967, CS9414.}
\label{ss:Sym-little}
There are several ways to formalize the question above.
First, note that all ``combinatorial interpretations''
above are naturally in $\SP$, with the exception of crystal graphs which
can have exponential size;\footnote{Arguably,
moving from the LR~tableaux to crystal graphs trades conciseness for elegance,
in roughly the same way as moving from standard Young tableaux to vectors in
the \defng{Young basis}.  While crystal graphs can be inspirational and
amenable to generalizations,
ultimately all such results can be obtained in the language of LR~tableaux
(cf.~\cite{Gal17}).
}
the issue can be fixed if one follows the bijection in~\cite{NS11}.

Second, all of these combinatorial interpretations are related to the
original LR~rule via a sequence of explicit poly-time bijections.
For example, the LR~tableaux are in bijection with:  hives and BZ
triangles \cite{PV05}, crystal graphs \cite{NS11}, Mondrian tableaux~\cite{Liu17}, etc.
This is unsurprising, perhaps, compared with parsimonious reductions
between many $\SP$-complete problems such as the number of $3$-colorings
and the number of Hamiltonian cycles, see e.g.\ \cite[$\S$18]{Pap}.

More surprising is that with few notable exceptions these bijections have
linear time complexity.  For example, when there is a natural presentation
as integer points in polytopes, these polytopes are essentially the same
and the bijection is given by a special linear map, see~\cite{PV05}.  This holds for LR~tableaux
interpreted as GT~patterns, as well as for hives, BZ~triangles and DK~arrays,
where the natural presentation is binary, see \cite{DK16,PV05}.

It is thus most surprising, that RSK is behind so many other
combinatorial interpretations, the bijection \.
$\LR(\la/\mu,\nu) \to \LR(\mu\circ \nu,\la)$ \. being the most
natural such example. Recall the (algebraically obvious)
\defng{fundamental symmetry} \ts $c^\la_{\mu\nu}=c^\la_{\nu\mu}$, which
is not transparent on the LR~tableaux or the highly symmetric BZ~triangles.
Nor does it follow from integer points in polytopes of GT~patterns since
the GZ~polytopes are asymmetric.  In other words,
proving \ts $|\LR(\la/\mu,\nu)| = |\LR(\la/\nu,\mu)|$ \ts bijectively
is rather nontrivial, and indeed is linear time equivalent to computing
the RSK map as proved in \cite{PV10}.

Applying the fundamental symmetry can double the number of
combinatorial interpretations, each time leading to a nontrivial
bijections, see e.g.\ \cite{ACM09} for puzzles and \cite{TKA18} for
hives.\footnote{Another way to  double  the number of
combinatorial interpretations is to use the \defng{conjugation symmetry} \ts
$c^\la_{\mu\nu} = c^{\la'}_{\mu'\nu'}\ts.$ This symmetry is easy to prove
algebraically.  It was proved
bijectively in~\cite{HS92}, and the authors attributed to Dennis White
a connection to the jeu-de-taquin.  It makes sense only for the unary encoding.}
It follows from~\cite{PV10}, that the RSK is behind them all. One can argue that
the $S_3$-symmetric LR~rule already has the fundamental symmetry ``embedded''
into the rule.  But the way the rule is constructed, to \emph{verify} that the
combinatorial interpretation is valid one needs to \emph{perform}
the~RSK.

Let us mention an experiment we made in \cite[$\S$7]{PV10},
where we used ingredients from different bijections to cook up four$\.(!)$
 bijections proving the fundamental symmetry for LR~tableaux.
In the scientific method style, we conjectured all four to be identical
without much of any checking \cite[Conj.~1]{PV10}. This conjecture is now
completely proved by a combination of results in \cite{DK05a} and \cite{TKA18}.
We followed the same approach to conjecture that two versions of the octahedral
map coincide \cite[Conj.~3]{PV10}; this was later proved in~\cite{HK06}.

Finally, we note a negative sort of evidence: as soon as one needs a different
kind of combinatorial interpretation for the LR~coefficients which does not involve
the RSK, nothing emerges.  This is why both Open Problems~\ref{op:tab-PPY}
and~\ref{op:tab-LPP}  are so challenging, cf.~$\S$\ref{ss:Kron-where-to-look}$\ts{}(\ddag)$.

}
\smallskip

{\small

\subsection{What gives?}  \label{ss:Sym-gives}
Combinatorics of the symmetric group is so vast, it is easy to get lost.
There are thousands of papers, hundreds of bijections, and dozens of
combinatorial interpretations which we cannot possibly mention here
for the lack of space and limited lifespan.  And yet, we claim that
there is a unifying principle for a large part of the field.

Fundamentally, the Algebraic Combinatorics is the study of bases in
symmetric spaces via combinatorics of transition matrices between them.
The dimensions of these symmetric spaces tends to be exponential
(think $S_n$ irreps, tensor powers, cohomology ring of Grassmannians, etc.)
Thus, we need to be able to handle not only the exponential size bases, but also
the exponential size of vectors in these bases.

In fact, most \emph{natural bases} of these spaces do have vectors with exponential
size support (think Young bases, Schur functions, Schubert polynomial, etc.)
Fortunately, the components of the whole vector can often be computed from
\emph{name of the vector} (think of semistandard Young tableaux from
Young diagrams for Schur functions, or RC-graphs from permutations).
The coefficients are not necessarily positive (e.g.\ tabloids in the
Young basis of $S_n$ irreps can have alternating signs), which is
why given a choice it is best to use positive bases, e.g.\ work with
Schur functions rather than Young bases.

Now, when applying \emph{operations} to our symmetric spaces, one needs to
be able to extract the standard symmetric bases out of these new large spaces.
Since one cannot work with exponential size vectors, the symmetry must be
traded for a concise presentation of the lex-smallest vector (in principle,
any orbit representative of the underlying group of symmetries).  For example,
this easily leads to standard Young tableaux as lex-smallest vectors among
tabloids of a given shape, see e.g.\ \cite[$\S$1.6]{Sag01}.  We refer
to \cite{KR84} for more on this philosophy from the Invariant Theory point
of view.

To see a more interesting example, consider the left-right action of
\ts $G\times G$ \ts on \ts $\cc[G]$, which has a simple spectrum as the
sum of \ts $\pi \otimes \pi$ \ts over all irreducibles.  For \ts $G=S_n$ \ts
one is then tempted to look for how lex-smallest components of vectors
in the invariant subspaces of \ts $\cc[S_n]$, leading to the
\defng{straightening} that Rota liked to much.  The details behind the
RSK correspondence and the LR~rule are more technical, but the underlying
story is similar and not much more complicated.  And as we emphasize
earlier, this all comes down to the miracle of RSK and its relatives.

What is amazing here is not that the resulting algorithm is nice and
interrelated, but that it is correct.  The former is a property of
the underlying algebra, while the latter is a combinatorial miracle
behind~RSK.  We don't need to work hard to imagine a world where
RSK does not exists -- the (original) straightening is not
in poly-time, or at least not obviously so (this is due to unexpected
sign cancellations).  If not for the Schur--Weyl duality we could be
stuck there for a long time, at least until a quantum version was
discovered.

For a more recent example, consider a interesting story of three
papers by Thomas--Yong \cite{TY18} and Pechenik--Yong \cite{PY17a,PY17b},
where the authors first introduce a generalization of jeu-de-taquin
to obtain \emph{some} combinatorial interpretation for the theory they
wanted, and then developed \emph{another} generalization of jeu-de-taquin
to prove \emph{the desired} combinatorial interpretation (conjectured earlier
by Knutson and Vakil, see~\cite[$\S$5]{CV09}).

\smallskip

We finish this section on a positive note.  If one \emph{wants} to
find a combinatorial interpretation for Kronecker coefficients, in my
opinion one would need to find an appropriate generalization of~RSK.
It is even clear how to start --- the map should be from 3-dim contingency
arrays into triples of semistandard Young tableaux, so that the number of
elements in preimage is always \ts $g(\lambda,\mu,\nu)$, which would then
have a combinatorial interpretation as the number of certain contingency
arrays.

The idea would be to
exploit the \defn{generalized Cauchy identity}
%\begin{align}\label{extended_identity}
$$
\prod_{i,\ts j,\ts k} \. \frac{1}{1\. - \. x_i \ts y_j \ts z_{k}} \ = \
\sum_{\lambda, \ts \mu, \ts \nu} \. g(\lambda,\mu,\nu)\. s_{\lambda}(\bx)\ts s_{\mu}(\by)\ts s_{\nu}(\bz)
$$
%\end{align}
(see e.g.~\cite[Exc.~I.7.10]{Mac95} and~\cite[Exc.~7.78]{Sta99}).
Unfortunately, the structure of such arrays is more complicated than in the
2-dim case of contingency tables, so finding lex-smallest
(under the action of the triple products of symmetric groups) does not appear
to be feasible in full generality.\footnote{A closely related issue is well
known in Algebraic Statistics, see e.g. \cite[$\S$10]{Sul18}.
}
Thus, unsurprisingly, until now this approach is worked out only in a
few special cases where lex-smallest contingency arrays are easy to distinguish,
see e.g.\ \cite{Val00} and \cite[$\S$2]{IMW17}.\footnote{Although we personally
don't expect this can be done in full generality (otherwise this would have
been done by now), we believe in telling both sides of the story.}

% When it comes to Schubert coefficients, there is even less hope, in my opinion.
% Again, it is clear how to start again (consider pairs of RC-graphs!), but the
% reduced factorizations that the RC-graphs encode appear to be rather cumbersome,
% cf.~\cite{Gar02}.  The remarkable \defng{Edelman--Greene bijection} extends
% RSK only so far \cite{HY14}, and there is no natural group action to use,
% suggesting that lex-smallest RC-graph of the product could be hard to compute.

}

\bigskip

%\newpage

%%%%%%%%%%%%%%%%%%%%%%%%%%%%%%%%%%%%%%%%%%%%%%%%%%%%%%%%%%%%%%%%%%%%%%%%%%
%%%%%%%%%%%%%%%%%%%%%%%%%%%%%%%%%%%%%%%%%%%%%%%%%%%%%%%%%%%%%%%%%%%%%%%%%%
%%%%%%%%%%%%%%%%%%%%%%%%%%%%%%%%%%%%%%%%%%%%%%%%%%%%%%%%%%%%%%%%%%%%%%%%%%
%%%%%%%%%%%%%%%%%%%%%%%%%%%%%%%%%%%%%%%%%%%%%%%%%%%%%%%%%%%%%%%%%%%%%%%%%%
%%%%%%%%%%%%%%%%%%%%%%%%%%%%%%%%%%%%%%%%%%%%%%%%%%%%%%%%%%%%%%%%%%%%%%%%%%
%%%%%%%%%%%%%%%%%%%%%%%%%%%%%%%%%%%%%%%%%%%%%%%%%%%%%%%%%%%%%%%%%%%%%%%%%%
%%%%%%%%%%%%%%%%%%%%%%%%%%%%%%%%%%%%%%%%%%%%%%%%%%%%%%%%%%%%%%%%%%%%%%%%%%
%%%%%%%%%%%%%%%%%%%%%%%%%%%%%%%%%%%%%%%%%%%%%%%%%%%%%%%%%%%%%%%%%%%%%%%%%%

\section{How to prove a negative?}\label{s:no}

%As we near the end of our journey, let us speculate on why some functions
%are in \ts $\SP$ \ts while others are not, and how can we prove the latter.

%\subsection{Assumptions}\label{ss:no-ass}
By the title of this section we mean: \defna{How to prove that a given function does not have a
combinatorial interpretation?}  Unfortunately, we really can't, at least
not unconditionally.  For all we know, it could be that \ts $\poly=\NP=\PSPACE$ \ts
and  \ts $\FP=\SP$.   In that case, poly-time verifiers can do magic, not just give
combinatorial interpretations.\footnote{See e.g.\ \cite{For97} for a quick review
of Counting Complexity.}  Now that we accepted the need for \emph{some}
complexity assumptions, we can proceed to discussing special cases.

\smallskip

\subsection{Nothing comes from nothing}\label{ss:no-nothing}
Let \. $f: \{0,1\}^\ast\to \nn$ \. be a function.
There is always a \defna{mathematical reason} \ts why we have \ts $f(x)\in \nn$ \ts for all \ts $x\in \{0,1\}^\ast$.
This reason could be an easy consequence from the definition, an observation,
a routine calculation, a standard result in the area, or a technically difficult
theorem.  Whether this reason can be replaced by a \defng{counting argument} \ts
is exactly the same as asking if \ts $f\in \SP$.

\smallskip

Consider some examples.  For \ts $f(x)=2^{|x|}$, the combinatorial interpretation
is ``all subsets of $[n]$,'' where $n=|x|$ is the length of the word~$x$.
For \ts $f(x) = x(x-1)(x-2)/3$, the combinatorial interpretation is
``all $3$-subsets of $[x]$ counted twice''.  Here the (easy) number
theoretic result  \ts ``$3\ts|\ts m(m-1)(m-2)$ \ts for all $m\in \nn$'' \ts
is proved by a counting argument.

In the other direction, for \ts $f(x) = 2^x$, there is no combinatorial
interpretation since the function is doubly exponential in the size of
the input.  There is still a (trivial) \defng{counting argument} \ts
here, placing this function in a counting class \ts $\SEXP$, but that
goes outside the scope of this survey.

Next, consider a polynomial \ts $\vp(x) = (x-1)^2$ \ts
and let \ts $f\in \SP$.   The function \ts $\vp(f)$ \ts is trivially
nonnegative and in \ts $\GapP=\SP-\SP$. However, there is no natural
combinatorial interpretation in this case, as we already discussed
in~$\S$\ref{ss:Main-examples}$\ts{}(4)$.  There, we used a
substitution\footnote{We elaborate below on the meaning of this substitution.}
\ts $x\gets h(G):=\#$Hamiltonian cycles in graph~$G$.  Note that if we
used a different substitution \ts $x\gets e(P):=\#$linear extensions of
a poset~$P$, \ts then $(x-1)^2\in \SP$, see~$\S$\ref{ss:Main-examples}$(1)$.
This mean that the inequality \ts $(x-1)^2\ge 0$ \ts may be
trivial analytically, but cannot be proved by a counting argument in the
worst case even if it can be proved in special cases,
as this example shows.

\smallskip

Back to the LR and Kronecker coefficients, there is a very clear
algebraic reason why we have \. $c^\la_{\mu\nu}\ge 0$ \ts  and  \ts
$g(\la,\mu,\nu)\ge 0$.
Fundamentally, it is reduced to the (easy) inequality \. $\<\zeta,\xi\>\ge 0$ \.
for every two characters \ts $\zeta,\xi$ \ts of a finite group~$G$.
It seems unlikely that this general inequality would have a
proof based on a counting argument.

Now note that we have a really good understanding why
LR~coefficients are in~$\SP$, and a very poor
understanding what why Kronecker coefficients are not
(thus, Conjecture~\ref{conj:Kron-main}).
As we mentioned above, it is neither unusual nor surprising that an
inequality can be in $\SP$ is some special cases and not in others.
Unfortunately, at the moment, the Kronecker coefficients are much
too unapproachable to admit a resolution of the problem.

\smallskip

In summary, we are concerned not so much with
whether any particular function is in~$\SP$, although we do
care about that.  Instead, we examine what proof ingredients
(of general results) can be shown to be in~$\SP$, and
refute those which are not.  This is the subject of our long
and technical paper~\cite{IP22}.  In this section we present
a brief and non-technical introduction.

\smallskip

\subsection{Polynomials}\label{ss:no-first}
A rational polynomial \ts $\vp\in \qqq[x]$ \ts is called
\defn{integer-valued} \. if \. $\vp(x)\in \zz$ \. for all \ts $t\in \zz$.
It is well known and easy to see that \ts $\vp$ \ts is integer-valued
if and only if \ts $\vp$ \ts  is an integer linear combination of binomials: \.
$\vp \in \zz\big\<1,x,\tbinom{x}{2},\tbinom{x}{3},\ldots\big\>$.
See e.g.\ \cite{CC16} for the background and many related results.

Polynomials \ts $\vp$ \ts in the semigroup \ts
$\nn\big\<1,x,\tbinom{x}{2},\tbinom{x}{3},\ldots\big\>$ \ts
are called \defn{binomial-good} (\defn{binomial-bad}, otherwise).
Binomial-good polynomials have a combinatorial interpretation for~$\vp(x)$.
Formally, for every function \ts $f\in \SP$, the function \ts $\vp(f)$ \ts
is also in~$\SP$.  For example, let \ts $f\gets h(G)$ \ts be the number of
Hamiltonian cycles in~$G$, and let \ts $\vp(x)=\frac13\ts x(x-1)(x-2)$.
Then we have \ts $\vp(f) = 2 \binom{f}{3}\in \SP$ \ts as discussed above,
since \ts $\vp(f)$ \ts counts twice the number of triples of
distinct Hamiltonian cycles in~$G$.

Let us emphasize that the combinatorial interpretation of $\vp(f)$ is
\defng{oblivious}, i.e.\ it works with the verifier for the function~$f$ as a \defng{black box} without
ever looking at the graph~$G$ or the notion of Hamiltonicity.  Note that the black
box is not allowed to \emph{compute}~$f$, and all it can do is give a combinatorial
interpretation for~$f$.  In other words, the black box verifier for~$f$ looks at \ts $x\in \{0,1\}^\ast$ \ts and
says Yes/No in poly-time, and the verifier for $\vp(f)$ works the same way by calling on the verifier for~$f$.\footnote{This notion
is completely formal. In computational complexity terminology, this says that
$\vp$ \defng{relativizes} with respect to \defng{oracle}~$f$.
We find the algorithmic notions more transparent. }

Denote by \ts $\BG$ \ts the set of binomially-good polynomials, and by
\ts $\Obli$ \ts the set of polynomials which have an oblivious
combinatorial interpretation.  From above, $\BG\subseteq \Obli$.
It was proved in \cite[Thm~3.13]{HVW95} (see also \cite[$\S$4.3]{IP22}),
that \ts $\BG = \Obli$, i.e.\ we show that
binomial-bad polynomials~$\vp$ cannot have oblivious combinatorial
interpretation.\footnote{To repeat
ourselves, we use ``oblivious'' to restrict combinatorial interpretations
of $\vp(f)$ to only those whose verifier ignores the nature of~$f$.
If you have never seen \defng{oblivious algorithms} and this notion seem
confusing, just think of \ts ``oblivious''=``nice''.}
For example,
\ts $\vp=(x-1)^2 = 2\binom{x}{2} - x +1$ \ts is binomial-bad, and for
\ts $f\gets h(G)$ \ts we do not expect a combinatorial interpretation.

We also generalize this result to multivariate polynomials, where \ts
$\vp\in \qqq[x_1,\ldots,x_k]$ \ts is integer-valued if \.
$\vp \in \zz\big\<\tbinom{x_1}{d_1}\cdots\tbinom{x_k}{d_k}\big\>$,
see \cite{Nag19}.  By analogy with the univariate case,
we say that $\vp$ is \defn{binomial-good} if \.
$\vp \in \nn\big\<\tbinom{x_1}{d_1}\cdots\tbinom{x_k}{d_k}\big\>$,
and prove \ts \ts $\BG = \Obli$ \ts for this generalization \cite[$\S$4.3]{IP22}.

\smallskip

Denote by \ts $\CI$ \ts the set of polynomials \ts
$\vp \in \qqq[x_1,\ldots,x_k]$ \ts such that \ts
$\vp(f_1,\ldots,f_k)\in \SP$ \ts for all \ts $f_1,\ldots,f_k \in \SP$.
Clearly, $\Obli \subseteq\CI$.
Our \defn{Binomial Basis Conjecture {\em (BBC)}} states that \ts $\Obli=\CI$, see \cite[$\S$4.4]{IP22}.
Since this conjecture implies \ts $\poly\ne \NP$ (ibid.), there is little hope
to resolve it even in the univariate case.  Thus, from the computational
complexity point of view, the \emph{``combinatorial interpretation problem''}
for polynomials is completely resolved.

\begin{ex}%[{\rm {\em Non-example}}{}]
\label{ex:no-non}
Let  $\vp := (x-y)^2$.  Clearly,
\ts $\vp(x,y)=2\binom{x}{2} + 2\binom{y}{2} -2xy +x +y \notin \BG$,
which implies that \ts $\vp\notin \CI$.
Consider the classical geometric proof of \ts $a^2+b^2\ge 2ab$ \ts
obtained by reflection of triangles as in Figure~\ref{f:obli}.
Since reflected triangles cover both $a\times b$
rectangles, the inequality follows.  The reader might want to ponder
(before reading the footnote), why does this construction not give an oblivious
combinatorial interpretation?\footnote{The difference \ts $a^2+b^2-2ab$ \ts
is the area of two triangles on the sides of the square.  To decide whether
point $(i,j)$ is in one of these triangles, you would need to determine the sign and
absolute value of $(a-b)$, which oblivious algorithms are not allowed to do.  }
\end{ex}

\vskip-.3cm
\begin{figure}[hbt]
\begin{center}
	\includegraphics[height=3.5cm]{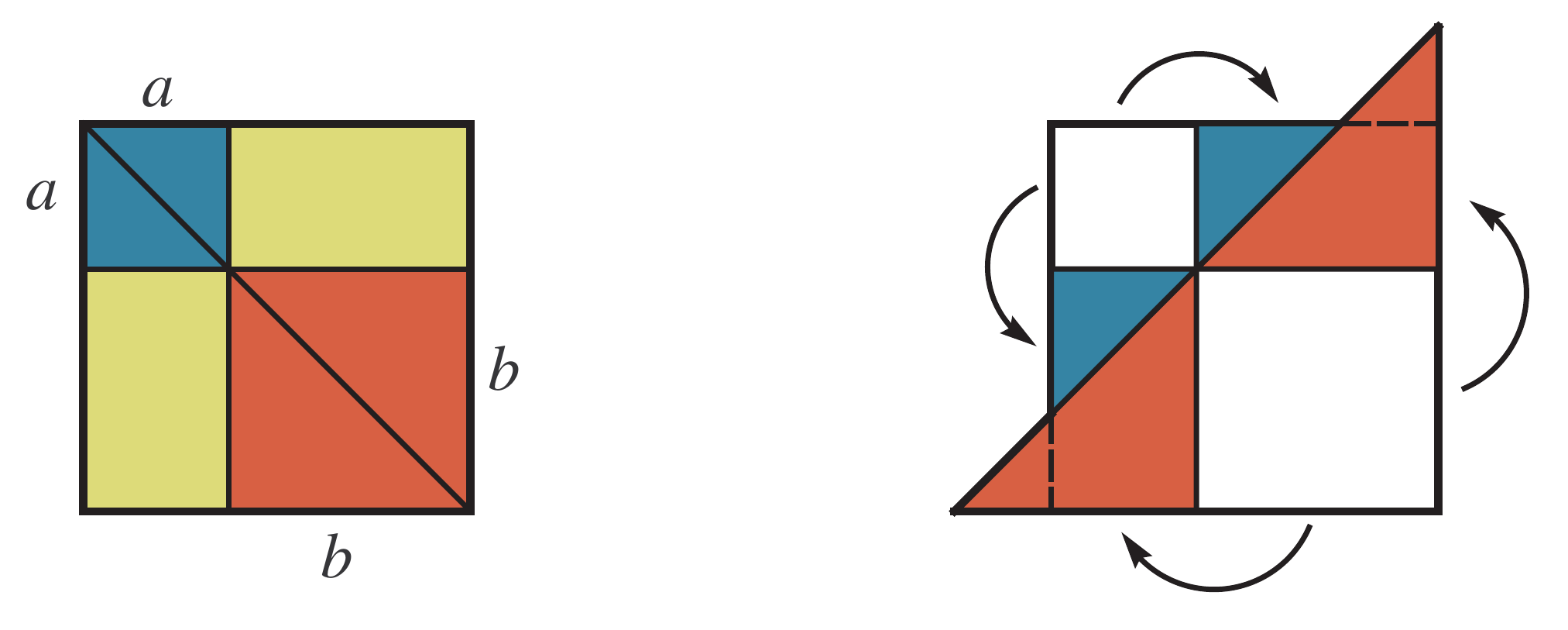}
\vskip-.3cm
\caption{Geometric proof of the inequality \ts $a^2+b^2 \ge 2ab$. Here the blue
and red triangles are reflected to completely cover yellow rectangles.}
\end{center}
\label{f:obli}
\vskip-.3cm
{}
\end{figure}

\smallskip

\subsection{Making weaker assumptions}\label{ss:no-weaker}
For some polynomials and some $\SP$ functions, the BBC
can be replaced by a weaker complexity theoretic assumption.
To explain what is going on, we need a few definitions.

Let $f, g: \{0,1\}^\ast \to \zz$ \ts be two $\SP$ functions.
We say that \ts $f$ \ts has a
\defn{parsimonious reduction} to~$\ts g$, if there is
a polynomial reduction \ts $\{0,1\}^\ast \to \{0,1\}^\ast$ \ts
which maps $f$ into~$g$.
% We write \ts $f\mapsto g$ \ts in this case.
Typically, if $f$ and $g$ count the number of solutions to
problem $A$ and $B$, the parsimonious reduction preserves the
number of solutions.

It is known, for example, that \ts {\sc \#SAT} \ts has a parsimonious
reduction to \. {\sc \#HamiltonianCycles} (and vise versa).
This means that for every {\sc SAT} formula $\Phi$, there
is a polynomially constructed graph~$G$, such that the number
of satisfying assignments of $\Phi$ is equal to $h(G)$, see e.g.\
\cite[$\S$18]{Pap}. % The converse parsimonious reduction also exits.
The same holds for \ts {\sc \#{}\ts{}3Colorings}, \ts {\sc \#WangTilings},
and many other $\SP$-complete functions.  For the lack of
a better name, we call such $\SP$ functions \defn{counting-complete},
and use \ts $\CCF$ \ts to denote the set of such functions.\footnote{If you
like \ts {\sf NSF}, this name is both appropriate and rewarding.}

On the other hand, unless $\poly=\NP$, the function
{\sc \#PerfectMatchings} is not in~$\CCF$, since the corresponding
vanishing problem  \ts $\PM(G)=^?0$ \ts
is in~$\poly$.  The same holds for  {\sc \#LinearExtensions}.
 So while these two functions are $\SP$-complete, they are not
counting-complete.\footnote{This is also why these two problems are more
interesting, and the proof of their $\SP$-completeness is more challenging
as they cannot have a parsimonious reduction to \ts {\sc \#SAT}.}

A map \ts $\vp: \nn^k\to \qqq$ \ts is called \defn{monotone} if \ts
$\vp(a_1,\ldots,a_k)\le \vp(a_1',\ldots,a_k')$ \ts for all integer \ts $a_1\le a_1'$,
\ldots, $a_k\le a_k'$. For example, polynomials \ts $x/2$, \ts $x-1$ \ts and \ts
$x+y$ \ts are monotone, but \ts $x^2-2x$ \ts and \ts $(x-y)^2$ \ts are not.
Denote by \ts $\Mon$ \ts the set of monotone polynomials, and note that \ts
$\BG\subset \Mon$.

\smallskip

We can now state two results which allow weakening of
the ``oblivious'' assumption:

\smallskip

\nin
(1) \ Let  $\vp(x,y) = (x-y)^2 \notin \Obli$.  We prove a stronger claim in
this case, that \ts $\vp \notin \CI$ \ts unless a standard complexity
assumption fails.  Formally, we prove in  \cite[$\S$2.3]{IP22}, that
for every two (independent) functions \ts $f,g\in \CCF$,
we have \ts $\vp(f,g)\notin\SP$ \ts unless \ts $\Sigma_2^{\textsc{p}}=\PH$.

This is the strategy used in \cite{IPP22}
to prove Theorem~\ref{t:char-abs}.  Recall that by MN~rule, the character \ts
$\chi^\la(\mu)\in \GapP$,
so it can be written as \ts $(f-g)$ \ts for some (usually, not independent)
functions \ts $f,g\in \SP$.  We found some instances of \ts $(\la,\mu)$,
for which the corresponding $(f,g)$ are both independent and
counting-complete.  This implied the result.  We believe the
same approach could potentially work for Conjecture~\ref{conj:tab-inverse-Kostka}.

\smallskip

\nin
(2) \ Let \ts $\vp\in \qqq[x_1,\ldots,x_k]$ \ts be a non-monotone polynomial,
so \ts $\vp \notin \BG$.  We prove a stronger claim in this case,
that \ts $\vp \notin \CI$ \ts unless a standard complexity assumption fails.
Formally, let \ts $f_1,\ldots,f_k\in \CCF$ \ts be independent counting-complete
\ts $\SP$ \ts functions.  We prove in  \cite[$\S$2.3]{IP22}, that
\ts $\vp(f_1,\ldots,f_k)\notin\SP$ \ts unless \ts $\UP=\coUP$,
see \cite[$\S$2.4]{IP22}.  In particular,  \ts $(f-1)^2\notin\SP$ \ts
unless \ts $\UP=\coUP$, cf.~$\S$\ref{ss:Main-examples}{\small $\ts{}(4)$}.

For another  example, recall the \defn{Motzkin polynomial} \ts
$M(x,y):=x^2y^4 + x^4y^2-3x^2y^2+1$.
It follows from the AM-GM inequality applied to positive terms,
that \ts $M(x,y)\ge 0$ \ts for all $x,y\in \rr$.
On the other hand, this polynomial is famously not a \defng{sum of squares}, and is a
fundamental example in \defng{Semidefinite Optimization}, see e.g.~\cite{Ble13,Mar08}.
Now, observe that \ts $M(x,y)$ \ts is not monotone: \ts $M(0,1)=1$ \ts and \ts $M(1,1) =0$.
This gives \ts $M\notin \CI$ \ts unless \ts $\UP=\coUP$.

Note that non-monotone polynomial \ts $(x-y)^2$ \ts
has a stronger property: % its multiples
\ts $(x-y)^2\psi \notin \Mon$,
for all \ts $\psi\in \zz[x,y]$.  By contrast,
the multiple \ts $(x-1)^2\cdot x=6\tbinom{x}{3}+2\tbinom{x}{2}\in \BG$ \ts
is binomial-good, and thus monotone.

\smallskip

\subsection{Algebraic inequalities}\label{ss:no-alg}
Let \ts $\vp,\ts \psi \in \qqq[t_1,\ldots,t_k]$, so that \ts $\vp \le \psi$ \ts
for all $(t_1,\ldots,t_k)\in \nn^k$.  Suppose that polynomial \ts $(\psi-\vp)\in \SP$ \ts
and has an oblivious combinatorial interpretation.  From above, we have
\ts $(\psi-\vp)$ \ts is binomial-good.  Since inequalities are
routine building blocks across mathematics, it is worth examining
which of them are binomial good.

First, note that the inequality \ts
$a^2 + b^2 \ge 2  a b$ \ts is equivalent to the complete square case
$(1)$ above.  In \cite[$\S$7.1]{IP22}, we show that a number of standard
inequalities are also not in~$\SP$ (in a sense of substitutions),
including the \defng{Cauchy inequality},
the \defng{Minkowski inequality}, and the \defng{Alexandrov--Fenchel inequality}.
All these proofs are routine and similar to our proof of Proposition~\ref{prop:sub-PM}.

Let us single out the \defn{Hadamard inequality} for real $d\times d$  matrices:
%\begin{equation}\label{eq:Had-3x3}
$$
\det\begin{pmatrix}
      a_{11} & \cdots & a_{1d}\\
      \vdots & \ddots & \vdots \\
      a_{d1} & \cdots & a_{dd}
    \end{pmatrix}^2 \ \le \ \prod_{i=1}^{d}\.\bigl(a_{i1}^2\ts + \ts \ldots \ts + \ts a_{id}^2\bigr)\ts.
$$
%\end{equation}
Geometrically, it says that the volume of a parallelepiped in~$\rr^d$ is at most the
product of its basis edge lengths, with equality when these edges are orthogonal.
Note that the standard proof use eigenvalues,
see e.g.\ \cite[$\S$2.13]{HLP52} and~\cite[$\S$2.11]{BB61},
suggesting that translation into combinatorial language would be difficult.

Denote by \ts $H_d\big(a_{11},\ldots,a_{dd}\big)$ \ts the polynomial
defined by the Hadamard inequality.  Note that $H_2(a,b,c,d)\in \CI$, since
$$H_2(a,b,c,d)\, = \,
(a^2\ts +\ts b^2)(c^2\ts + \ts d^2) \. - \.
\det\begin{pmatrix}
      a & b\\
      c & d
    \end{pmatrix}^2   \, = \, (ac \ts + \ts bd)^2 .
$$
On the other hand, \ts
$H_2(a,b,c,-d) \ge 0$ \ts has different properties from our point of view.
Indeed, we have \ts $H_2(a,b,c,-d) =(ac - bd)^2 \notin \CI$ \ts unless  \ts $\Sigma_2^{\textsc{p}}=\PH$.
Finally, observe that
$$
H_3{\begin{pmatrix}
x & \binom{x}{3} & 0 \\
0 & 1 & 1 \\
1 & 0 & 1
\end{pmatrix} }
\ = \
3 \ts \tbinom{x}{1}
\. + \. 6  \ts \tbinom{x}{2}
\. - \. 3  \ts\tbinom{x}{3}
\. + \. 28 \ts \tbinom{x}{4}
\. + \. 90 \ts \tbinom{x}{5}
\. + \. 60 \ts \tbinom{x}{6} \, \notin \BG\ts.
$$
By the BBC, we have \ts $H_3 \notin \CI$.  Since \ts $H_3 \in \Mon$, this
is the best we can prove with our tools.

\smallskip

\subsection{Algebraic inequalities restricted to semialgebraic sets}\label{ss:no-semi-alg}
In many cases, one is interested in polynomial inequalities where the
variables are themselves constrained by a system of polynomial inequalities
and equations.  Our \defn{Diagonalization Theorem} \cite[Thm~6.2.1]{IP22}
gives a complete algebraic characterization of such systems with an
oblivious combinatorial interpretation.

The results in \cite{IP22} are both technical to state, difficult to prove
and hard to apply. Instead of presenting or even outlining them here,
we discuss two examples (one easy and one difficult) which give a glimpse
at our arguments.

\medskip

\nin
(1) \ The \defn{Ahlswede--Daykin {\em (AD)} inequality},
see e.g.\ \cite[$\S$6.1]{AS16}, for \ts $n=1$ \ts states:
$$
\left\{ \ts
\aligned
& x_0 \ts y_0 \. \le \. u_0  \ts v_0 \ , \ x_0 \ts y_1 \. \le \. u_0  \ts v_1 \\
& x_1 \ts y_0 \. \le \. u_0  \ts v_1  \ , \ x_1 \ts y_1 \. \le \. u_1  \ts v_1 \endaligned
\right. \quad \Longrightarrow \quad
 (x_0+x_1)(y_0+y_1) \. \le \. (u_0+u_1)(v_0+v_1)\ts,
$$
for all \ts $x_i,y_i, u_i, v_i \ge 0$, where \ts $i\in \{0,1\}$.
We prove in \cite[Prop.~2.5.1]{IP22}, that the AD~inequality does
not have an oblivious combinatorial interpretation. In other words,
we prove that \ts $\vp \notin\Obli$, where \.
$\vp\ts :=\ts (u_0+u_1)(v_0+v_1) \ts - \ts (x_0+x_1)(y_0+y_1)$ \. restricted
to the above system of four inequalities.
Below is the outline of the proof.

% i.e.\ $\vp \notin\Obli$.

Following \cite[$\S$7.4]{IP22}, we
rewrite the inequality in terms of $\SP$ functions as follows:
$$
\left\{ \ts
\aligned
& \al_0  \be_0 + \xi_1 = \ga_0 \de_0 \, , \ \al_0 \be_1 + \xi_2 = \ga_0  \de_1 \\
& \al_1  \be_0 + \xi_3 = \ga_0  \de_1  \, , \ \al_1 \be_1 + \xi_4 = \ga_1   \de_1 \endaligned
\right. \ \  \Longrightarrow \ \
(\ga_0+\ga_1)(\de_0+\de_1) \ts - \ts (\al_0+\al_1)(\be_0+\be_1) \in \SP\ts,
$$
for all \. $(\al_0,\al_1,\be_0, \be_1, \ga_0, \ga_1, \de_0, \de_1, \xi_1,\xi_2,\xi_3,\xi_4) \in (\SP)^{12}$.
\ts Second, make the following substitution into these 12 functions:  \.
$\bigl(1,1,x,x,x,1,1,x,0,2\binom{x}{2},2\binom{x}{2},0\bigr)$ \ts and check that
the above system of inequalities is satisfied. Now let
\ts
$x\gets f$, where \ts $f\in \SP$.  This substitution gives \.
$\vp(f)=  (f-1)^2 \notin\BG$, and thus \ts $\vp\notin \Obli$.

\medskip

\nin
(2) \ Let \ts $\ba,\bb \in \rr^n$ be weakly decreasing,
such that \ts $\ba \trianglerighteq \bb$, see $\S$\ref{ss:tab-SSYT}.
The \defn{Karamata inequality} says
for every \defng{convex function} \ts $F: \rr^n \to \rr$, we have \ts $F(\ba) \ge F(\bb)$,
see e.g.~\cite[$\S$3.17]{HLP52} and~\cite[$\S$28, $\S$30]{BB61}.  We refer
to \cite{BP21,PPS20} for some recent applications to linear extensions
and Young tableaux, and to \cite{MOA11} for numerous generalizations
and further references.

Following \cite[$\S$7.5]{IP22}, we
rewrite the inequality in the language of $\SP$ functions.
Suppose
$$f_1 \. + \. \ldots \. + \. f_i \, = \, g_1 \. + \. \ldots \. + \. g_i \. + \. h_i \quad
\text{for all \ \ $1\le i \le n$,}
$$
where \ts $f_i,g_i, h_i\in \SP$, and \ts $h_n=0$. Suppose also
that
$$f_i \, = \, f_{i+1} \. + \. d_i\., \quad g_i \, = \, g_{i+1} \. + \. e_i\., \quad
\text{for all \ \ $1\le i < n$,}
$$
where \ts $d_i,e_i \in \SP$. Finally, let \ts $\ga: \zz\to \nn$ \ts be a
nonnegative convex function.  The \defn{Karamata function} \. $K^{n}_\ga$ \.  is defined as
$$
\rK^{\<n\>}_\ga(f_1,\dots,f_n, g_1,\ldots,g_n) \, := \,
\sum_{i=1}^{n} \. \ga(f_i) \, - \, \sum_{i=1}^{n} \. \ga(g_i)\..
$$
%Here and below we write \ts $K_\ga$ \ts when~$n$ is clear.
Clearly, \ts $\rK^{\<n\>}_\ga \in \GapP$ \ts and nonnegative by the Karamata theorem.
It is thus an interesting question if \ts $\rK^{\<n\>}_\ga \in \Obli$, i.e.\
if \ts $\rK^{\<n\>}_\ga(f_1,\ldots,g_n)\in \Obli$, i.e.\ in~$\SP$ \ts
for all \ts $f_i, g_i, h_i, d_i,e_i\in \SP$.

For example, for \ts $\ga(t) = \al t+\be$, we have \ts $\rK^{\<n\>}_\ga = 0$.  Similarly,
$$
\aligned
\rK^{\<2\>}_\ga(f_1,f_2,g_1,g_2) \, & = \, (e_1+h_1) \ts h_1 \. \in \. \SP \quad \text{for}  \ \ \ga(t)=\tbinom{t}{2}, \quad \text{and}
\\
\rK^{\<3\>}_\ga(f_1,f_2,f_3,g_1,g_2,g_3) \, & = \, (d_1+e_1) \ts h_1 \. + \. (d_2+e_2) \ts h_2 \. \in \. \SP \quad \text{for}  \ \ \ga(t)=t^2.
\endaligned
$$
It follows from here that \. $2\ts \rK^{\<3\>}_{\ga}\in \Obli$ \ts where \ts $\ga(t)=\tbinom{t}{2}$, since
linear terms cancel out.  We prove in \cite[Prop.~7.5.5]{IP22}, that \.
$\rK^{\<3\>}_\ga\notin \Obli$ \. for \ts $\ga(t) = \binom{t}{2}$.  The proof requires a computation
of lattice points in a $12$-dimensional polytope defined by linear equations and inequalities
corresponding to constraints on \ts $f_i,g_i,h_i,d_i,e_i$.

\smallskip

\subsection{What's next?}\label{ss:no-next}
By now the (exhausted) reader knows what kind of results we want to
prove --- the many \ts ``not in $\SP$'' \ts conjectures throughout the
paper. And they know how we imagine the plan of attack --- by simulating
the proofs of positivity and integrality of these combinatorial
functions with polynomial equations.

There are two main obstacles on the way.  First, the proof of
positivity and integrality can be rather involved, so distilling a
single reason and expressing it as a polynomial inequality can
be difficult.\footnote{It is now well understood how to translate
general mathematical proofs in a formal language of low degree
polynomials which can then be ``checked'' with few queries,
see \cite{A+98}. The connection is somewhat superficial as the
latter uses polynomials over finite field, while in \cite{IP22}
we work over~$\cc$.  Still, this suggests commonality of the ideas,
keeping alive the hope that such translation can be made in special cases.
}
Naturally, one would want to start with counting graphs
and linear extensions rather than Kronecker and Schubert coefficients,
as the former seem much more generic and less involved.
It may take a long time before this project can reach the latter.

Second, the family of polynomials for which we know that they don't
have an oblivious combinatorial interpretations is rather large and
seems satisfactory for applications.   But to make the final results
more accessible and convincing to the general audience, it is important to weaken
the assumptions (see~$\S$\ref{ss:no-weaker}).   This direction is
certainly worth exploring in the context of Computational Complexity.

\bigskip

%%%%%%%%%%%%%%%%%%%%%%%%%%%%%%%%%%%%%%%%%%%%%%%%%%%%%%%%%%%%%%%%%%%%%%%%%%
%%%%%%%%%%%%%%%%%%%%%%%%%%%%%%%%%%%%%%%%%%%%%%%%%%%%%%%%%%%%%%%%%%%%%%%%%%
%%%%%%%%%%%%%%%%%%%%%%%%%%%%%%%%%%%%%%%%%%%%%%%%%%%%%%%%%%%%%%%%%%%%%%%%%%
%%%%%%%%%%%%%%%%%%%%%%%%%%%%%%%%%%%%%%%%%%%%%%%%%%%%%%%%%%%%%%%%%%%%%%%%%%
%%%%%%%%%%%%%%%%%%%%%%%%%%%%%%%%%%%%%%%%%%%%%%%%%%%%%%%%%%%%%%%%%%%%%%%%%%
%%%%%%%%%%%%%%%%%%%%%%%%%%%%%%%%%%%%%%%%%%%%%%%%%%%%%%%%%%%%%%%%%%%%%%%%%%
%%%%%%%%%%%%%%%%%%%%%%%%%%%%%%%%%%%%%%%%%%%%%%%%%%%%%%%%%%%%%%%%%%%%%%%%%%
%%%%%%%%%%%%%%%%%%%%%%%%%%%%%%%%%%%%%%%%%%%%%%%%%%%%%%%%%%%%%%%%%%%%%%%%%%

{\small
\section{Counting complexity addendum}\label{s:addendum}

\subsection{We need a list} \label{ss:addendum-def}
Throughout the whole survey we tried to mention \ts \defng{$\SPr$-completeness} \ts and
\ts \defng{$\SPr$-hardness} \ts as little as possible.  There are two reasons for that:
we wanted not to distract the reader from the main problem (membership in~$\SP$),
and we wanted to minimize the confusion that invariably arises.

There is, however, a direct connection to these complexity classes.  As we mention in
$\S$\ref{ss:basic-classes}, we have \ts $\FP \subseteq \SP$, which makes
combinatorial interpretation of problems in \ts $\FP$ \ts trivial.  Naturally,
we are thus interested in problems that are not
in~$\FP$.  The best evidence that a function is not in
\ts $\FP$ \ts yet potentially in \ts $\SP$, is if a function is $\SP$-hard.
This is why it is worth checking $\SP$-hardness of functions in every conjecture
and open problem that we pose (cf.~$\S$\ref{ss:finrem-dichotomy}).

If this was about $\NP$-completeness, we would stop here and refer to
\cite{GJ79} along with some recent comprehensive list of $\NP$-complete problems
(such as \href{https://en.wikipedia.org/wiki/List_of_NP-complete_problems}{this one}
on {\tt Wikipedia}).  Unfortunately, there does not seem to be such comprehensive sources
about $\SP$-completeness.  Thus we present an annotated short list of such results,
restricted only to functions which we consider relevant to the survey.

\smallskip

\subsection{Graph theory problems}\label{ss:addendum-graphs}
We proceed roughly according to the sections in this survey.

\smallskip

\begin{enumerate}
  \item {\sc \#Monotone 2SAT} is $\SP$-complete but $\notin\CCF$.
This implies that {\sc \#VertexCover} is $\SP$-complete \cite{Val79a}.
  \smallskip
  \item {\sc \#HamiltonianCycles} is $\SP$-complete \cite{Val79a}.
  Moreover {\sc \#HC} is in $\CCF$, since the proof is based on a parsimonious bijection.
  In the context of Smiths's theorem (see~$\S$\ref{ss:Main-examples}), this remains
  true for {\sc \#HC} in cubic planar graphs \cite{LOT03}.
  \smallskip
  \item {\sc \#3Colorings} is $\SP$-complete \cite{Val79a}.
  Moreover {\sc \#3C} is in~$\CCF$,
  since the proof is based on a parsimonious bijection.
  \smallskip
  \item {\sc \#PerfectMatchings} is~$\SP$-complete via reduction to {\sc Permanent} \cite{Val79c}.
  This implies that {\sc \#$k$-Matchings}
  \ts $\{p(G,k)\}$ is $\SP$-complete (see~$\S$\ref{ss:sub-match}).
  Thus, \ts $\{f(G,k)\}$ is $\SP$-hard by telescoping.
The problem remains $\SP$-complete for subgraphs in $\zz^3$ \cite{Val79b},
and even for the number of 3-dim domino tilings \cite{PY13}.
  For planar graphs {\sc \#PerfectMathings} is in~$\FP$ by the \defng{Kasteleyn formula},
  see e.g.\ \cite[$\S$8.3]{LP86}.\\
  The \emph{vanishing problem} \ts $\{\PM(G)=^?0\}$ \ts is in~$\poly$,
  see e.g.\ \cite[$\S$9.1]{LP86}.
\smallskip
\item {\sc \#SpanningForests} \ts
$\{F(G,k)\}$ is $\SP$-complete (see~$\S$\ref{ss:sub-forests}).
This holds because the total number of
spanning forests \. $1+F(G,1)+\ldots + F(G,n-1)=T(G; 2,1)$ \. is an
evaluation of the \defng{Tutte polynomial} \ts known to be $\SP$-complete,
see e.g.~\cite[Thm~6.2.9]{Wel93}.  This implies that $\{f(G,k)\}$ is $\SP$-hard.
For a fixed $k\ge 1$, both \ts $\{F(G,k)\}$ \ts
and \ts $\{F(G,n-k)\}$ \ts are in~$\FP$~\cite{Myr92}.
\smallskip
\item {\sc \#SpanningSubgraphs} \ts is \ts $\SP$-complete \cite{PB83}.
This is an evaluation of the Tutte polynomial (see e.g.\ \cite[Ch.~X]{Bol98}).
We conjecture that the function in Conjecture~\ref{conj:sub-Bunkbed} is \ts $\SP$-hard.
\smallskip
\item {\sc IsingModelStatisticalSum} \ts is \ts $\SP$-complete \cite[Thm~15]{JS93}.
In notation of~$\S$\ref{ss:sub-Ising}, this and Proposition~\ref{p:Ising-Cor-functions}
implies that \ts $\{\Cor(v,w)\}$ \ts is $\SP$-complete.
For planar graphs, the problem is in~$\FP$ \ts by the \defng{Kasteleyn--Fisher
determinant formula} \ts \cite{Kas63,Fis66}.

\smallskip
\item
For general rational hyperplane
arrangements, the problem of counting the number of regions in the
complement is $\SP$-complete.  For membership in $\SP$, see $\S$\ref{ss:seq-Cat}.
The hardness follows from
graphical arrangements where the number of regions equal
to the evaluation \ts $|\chi_G(-1)|$ \ts of the
\defng{chromatic polynomial} (see e.g.\ \cite[$\S$3.11, Exc.~94-95]{Sta99}).
This evaluation is equal to the \defng{number of acyclic orientations} of~$G$,
known to be $\SP$-complete, see e.g.\ \cite[Thm~6.2.9]{Wel93}.
\smallskip
\item
{\sc \#LinearExtensions} \ts $\{e(P)\}$ \ts are $\SP$-complete \cite{BW91}.
Thus, the function defined by the Bj\"orner--Wachs inequality is $\SP$-complete
(see~$\S$\ref{ss:LE-BW}).  By telescoping, the function defined by the
Stanley inequality is $\SP$-hard (see~$\S$\ref{ss:LE-Stanley}).
Computing \ts $e(P)$ \ts remains $\SP$-complete for $P$ of height two,
and of width two~\cite{DP18}.  The function defined by the Sidorenko
inequality (see~$\S$\ref{ss:LE-Sid}) is conjectured to be $\SP$-complete
in~\cite[$\S$9.6]{CPP22b}.
\smallskip
\item
{\sc \#WangTilings} \ts of a square is $\SP$-complete; this follows e.g.\ from
the proof of Thm~3 in \cite{DD07} that the decision problem in $\NP$-complete.
\smallskip
\item
{\sc Volume} is $\SP$-hard via reduction to {\sc \#LinearExtensions} of \defng{order polytope} \cite{BW91}, \\
remains $\SP$-hard for \defn{zonotopes} \cite{DGH98}. \\
{\sc MixedVolume} is $\SP$-hard for boxes via reduction to {\sc Permanent} (ibid.)
\end{enumerate}

\smallskip

\subsection{Algebraic combinatorics problems with binary input}\label{ss:addendum-bin}
The type of input makes so much difference for problems in
Algebraic Combinatorics, we decided to separate them altogether
and make a clear indication in the name, so they would be impossible
to confuse.

Note that if the problem is in $\FP$ in the binary
input, then this is also true in the unary input.  Vice versa, if the problem
is $\SP$-complete or $\SP$-hard in the unary
input, then this is also true in the binary input.  Similar claims hold
for the decision problems as well.

\smallskip

\begin{enumerate}
  \item {\sc ContingencyTablesBinaryInput} \ts $\{\rT(\la,\mu)\}$ \ts  is  \ts $\SP$-complete even for \ts $\ell(\la)=2$ \ts \cite{DKM97}.  \\
  The \emph{vanishing problem} \ts $\{\rT(\la,\mu)>^?0\} \in \poly$, since it is equivalent to \ts $\{|\la|=^?|\mu|\}$.\\
  When \ts $\ell(\la)$ \ts is fixed, the problem has FPTAS \cite{G+11}. \\
  When \ts $\ell(\la), \ell(\mu)$ \ts are fixed, the problem is in $\FP$ by \cite{Bar93}.

\smallskip

\item {\sc KostkaBinaryInput} \ts $\{K_{\la\mu}\}$  \ts is  \ts  $\SP$-complete \cite{Nar06}. \\
  The \emph{vanishing problem} \ts  $\{K_{\la\mu}>^?0\} \in \poly$, since it
  is equivalent to \ts $\{\la \trianglerighteq^?\mu\}$, see $\S$\ref{ss:tab-SSYT}. \\
  The \emph{uniqueness problem} \ts
  $\{K_{\la\mu}=^?1\} \in \poly$, see a complete characterization in \cite{BZ90}.

\smallskip

\item {\sc LittlewoodRichardsonBinaryInput} \ts $\{c^\la_{\mu\nu}\}$  \ts is  \ts
$\SP$-complete \cite{Nar06}. \\
  The \emph{vanishing problem} \ts  $\{c^\la_{\mu\nu}>^?0\} \in \poly$,
  see \cite{DM06,MNS12}.\footnote{This is based on the \defng{saturation property},
  see~$\S$\ref{ss:Kron-motivation}$\ts{}(2)$, which fails for
  other root systems. Notably, for the \ts $B$--$C$--$D$ \ts Lie types,
  it holds up to a factor of two~\cite{Sam12}.  It is open whether
  the vanishing problem is in~$\poly$ in these case; this would follow from \cite[Conj.~4.7]{DM06}
  (cf.\ \cite[$\S$5.2]{GOY21} and \cite[$\S$7.4]{RYY22}).}
  \\
  The \emph{uniqueness problem} \ts  $\{c^\la_{\mu\nu}=^?1\} \in \poly$.
  More generally, $\{c^\la_{\mu\nu}=^?t\} \in \poly$ \ts \\
  for every fixed $t\ge 0$, see \cite[Thm~11.3.2]{Ike12} and \cite{Ike16}.

\smallskip

\item {\sc KroneckerBinaryInput} \ts $g(\la,\mu,\nu)$ \ts is \ts $\SP$-hard.
This follows easily from the result that \ts {\sc Littlewood-RichardsonBinaryInput}
is  \ts $\SP$-complete, and trivially from the unary case.\\
The \emph{vanishing problem} \ts  $\{g(\la,\mu,\nu)>^?0\}$ \ts is $\NP$-hard; this follows
trivially from the unary case.

\smallskip

\item {\sc ReducedKroneckerBinaryInput} \ts $\{\rg(\al,\be,\ga)\}$  \ts is  \ts $\SP$-hard, same reason as above.
\\
We conjecture that the \emph{vanishing problem} \ts  $\{\rg(\al,\be,\ga)>^?0\}$ \ts is in
$\NP$-hard.

\smallskip

\item {\sc ExcitedDiagrams} \ts $\{|\ED(\la/\mu)|\}\in \FP$ \ts via reduction to
\defng{flagged tableaux} \ts \cite[Cor.~3.7]{MPP18a}. \\
This is the number of terms of the summation in the NHLF.  The \emph{vanishing problem} is trivial.

\smallskip

\item {\sc CharacterSquaredBinaryInput} \ts $\big\{(\chi^\la(\mu))^2\big\}$ \ts is doubly exponential
and thus \\ not in~$\PSPACE$.  For example, $\chi^\la(1)=\Cat(m)=e^{\Omega(n)}$, where \ts $\la=(m,m)$ \ts
and \ts $n=2m$.  \\
The \emph{vanishing problem} \ts  $\{\chi^\la(\mu)\ne^?0\}$ \ts is \ts $\NP$-hard \cite[$\S7$]{PP17};
a stronger result follows \\ from the unary case.

\smallskip

\item {\sc HurwitzBinaryInput} \ts $\{h_{g\mu}\}$ \ts is doubly exponential
and thus not in~$\PSPACE$. \\ For example, $h_{0 \ts (n)} = n^{n-2}$.

\end{enumerate}

\smallskip

\subsection{Algebraic combinatorics problems with unary input}\label{ss:addendum-unary}
This is the most interesting case, and the one we discuss throughout
the paper.\footnote{To indicate unary input, the literature often refers
to ``strong'' $\NP$- and $\SP$-completeness, see \cite{GJ78} and \cite[$\S$8]{Vaz01}.
We find this terminology misleading and best to be avoided,
as some results become weaker while others stronger when the input size changes. }
To minimize the overlap, we don't include here some of the
poly-time result which hold already for the binary case.

\smallskip

\begin{enumerate}
\item {\sc ContingencyTablesUnaryInput} \ts $\{\rT(\la,\mu)\}$ \ts
is not known to be  \ts $\SP$-complete, \\  see \cite[$\S$1.1]{DO04}
and \cite[$\S$8.1]{PP17}.  We conjecture this to be true. \\
When \ts $\ell(\la)$ \ts is fixed the problem is not \ts $\SP$-complete
unless \ts $\poly=\NP$.\footnote{This follows from having FPTAS in the binary input,
see e.g.~\cite[$\S$8.3]{Vaz01}.}

\smallskip

\item {\sc KostkaUnaryInput} \ts $\{K_{\la\mu}\}$  \ts is not known to be  \ts $\SP$-complete
\cite[$\S$8.1]{PP17}.  \\ We conjecture this to be true.
This would follow from the conjecture that \\ {\sc \#ContingencyTablesUnaryInput}
is $\SP$-complete, via the reduction in~\cite{Nar06}. \\
Moreover, we conjecture that \ts $\big\{K_{\la\ts (2^a1^b)}\big\}$ \ts is $\SP$-complete.

\smallskip

\item {\sc LittlewoodRichardsonUnaryInput} \ts $\{c^\la_{\mu\nu}\}$  \ts  is conjectured to be  \ts $\SP$-complete
\cite[$\S$8.1]{PP17}.  This would follow from the conjecture that  \ts
{\sc KostkaUnaryInput} is \ts  $\SP$-complete, via the reduction in \cite{Nar06}.
The \emph{vanishing problem} \ts has a \defng{recursive description}, see \cite[Prop.~9]{Zel99}. \\
The \emph{vanishing} \ts of \defng{generalized LR~coefficients} (tileability using Knutson--Tao~puzzles \\
regions in the triangular lattice with given boundary), is $\NP$-complete \cite{PY14}.

\smallskip

\item {\sc KroneckerUnaryInput} \ts $g(\la,\mu,\nu)$ \ts is \ts $\SP$-hard.  This follows from  \cite{IMW17}. \\
% It follows from the proof that this problem is in~$\CCF$.
%
The \emph{vanishing problem} \ts  $\{g(\la,\mu,\nu)>^?0\}$ \ts is $\NP$-hard, ibid. \\
We conjecture that \ts $\{g(\la,\la,\la) \. : \. \la=\la'\}$ \ts is $\SP$-hard, cf.\
Conjecture~\ref{conj:Kron-triple} and~\cite{PP22},  \\
and that \ts $\{g(\la,\la,\mu) \. : \. \mu=(n-k,k)\}$ \ts is also $\SP$-hard, cf.\
Remark~\ref{r:two-row} and~\cite{PP14}.\footnote{The last conjecture was suggested
by Greta Panova (personal communication, Sep.~2022). }

\smallskip

\item {\sc ReducedKroneckerUnaryInput} \ts $\rg(\al,\be,\ga)$  \ts is  \ts $\SP$-hard \cite{PP20}. \\
% We conjecture that this problem is in~$\CCF$.
%
The \emph{vanishing problem} \ts  $\{\rg(\al,\be,\ga)>^?0\}$ \ts is conjectured to be \ts
$\NP$-hard in \cite[$\S$4.4]{PP20}.

\smallskip

\item {\sc CharacterSquaredUnaryInput} \ts $\big\{(\chi^\la(\mu))^2\big\}$ \ts is $\SP$-hard \cite{Hep94}, \\
and \ts $(\chi^\la(\mu))^2\notin\SP$ \ts unless \ts $\Sigma_2^{\textsc{p}}=\PH$,
see Theorem~\ref{t:char-abs}.  \\
The \emph{vanishing problem} \ts  $\{\chi^\la(\mu)=^?0\}$ \ts is \ts $\CeqP$-complete, and thus $\NP$-hard \cite{IPP22}. \\
The \emph{positivity problem} \ts  $\{\chi^\la(\mu)\ge^?0\}$ \ts is \ts $\PP$-complete, and thus $\PH$-hard, ibid.

\smallskip

\item {\sc InverseKostkaUnaryInput} \ts $\big\{K^{-1}_{\la\mu}\big\}$  \ts has not been studied.
We conjecture it is $\SP$-hard. \\ For the \emph{vanishing problem} \ts  $\{K^{-1}_{\la\mu}=^?0\}$,
we conjecture it is \ts $\CeqP$-complete (cf.\ Conjecture~\ref{conj:tab-inverse-Kostka}).

\smallskip

\item {\sc BruhatOrderIdeal} \ts $\{B(\si)\}$, where \ts $B(\si):=|\{\om\in S_n \.:\. \om \preccurlyeq \si\}|$,
is \ts $\SP$-complete~\cite{DP18}. \\ This problem is equivalent to \ts {\sc \#LinearExtensions} \ts
 for permutation posets: \ts $B(\si)=e(P_\si)$.

\smallskip

\item {\sc ReducedFactorizations} \ts $\{r(w)\.:\. w \in S_n\}$ \cite{Sta84}, are defined as
$$
\qquad r(w)\. := \. \#\big\{\ts(i_1,\ldots,i_\ell) \,: \, (i_1,i_1+1)\cdots (i_\ell,i_\ell+1)=w,
\. 1\le i_j< n, \. \ell=\inv(\om)\ts \big\}.
$$
The problem is conjectured to be \ts $\SP$-complete \cite[$\S$8.5]{DP18},
cf.~\cite[$\S$6]{MPY22}.\footnote{\cite[$\S$6]{MPY22} observes that \ts $\{r(w)\}\notin\SP$ \ts
when permutations are presented in binary via the \defng{Lehmer code}.
}  When \ts $w\in S_n$ \\ is
\defng{vexillary} ($2143$-avoiding), we have \ts $\{r(w)\}\in \FP$ \ts from
\cite[Cor.~2.8.2]{Man01} and the HLF.

\smallskip

\item {\sc SchubertCoefficient} \ts $c^{w}_{uv}$ \ts is not known to be $\SP$-hard.\footnote{The argument
in~\cite[p.~885]{MQ17} claiming that \ts $\{c^{w}_{uv}\}$ is $\SP$-complete via reduction to
\ts $\{c^\la_{\mu\nu}\}$ \ts is erroneous as it conflates the input sizes.
The authors acknowledge the mistake (personal communication, 2022).} We conjecture this to be true. \\
The \emph{vanishing problem} \ts $\{c^{w}_{uv}>^?0\}$ \ts is not knows to be \ts $\NP$-hard
\cite[$\S$4]{ARY13}.

\smallskip

\item {\sc SchubertKostka}  \ts $\{K_{ua}\}$ \ts and \ts {\sc \#RC-graphs} \ts
$\{\Sch_u(1)=\sum_{a} K_{ua}\}$ \ts have not been studied. \\ We conjecture that both are \ts $\SP$-complete.

\smallskip

\item {\sc IncreasingTableaux} \ts $\big\{g^\la := |\IT(\la)|\big\}$, where an
\defn{increasing tableau} \ts $A\in \IT(\la)$ \ts is a plane partition
strictly increasing in rows and columns, and no gaps between $1$ and maximal
entry of~$A$, see \cite{TY09}.  Note that \ts $g^\la \ge f^\la$.
This problem \ts $\{g^\la\}$ \ts
was proposed in \cite[$\S$1.3]{TY11},
and conjectured to be $\SP$-complete in \cite[$\S$7.10]{MPP22}.

\smallskip

\item {\sc SetValuedTableaux} \ts $\big\{s(\la,k):=|\SVT(\la,k)|\big\}$, where
a \defn{set-valued tableau} \ts $A\in \SVT(\la,k)$ \ts is a surjection from \ts
$[k]$ \ts to squares of Young diagram~$\la$, s.t.\ the numbers increase in rows
and columns of the image, see \cite{Buch02}.  Complexity of \ts $\{s(\la,k)\}$
\ts was asked in \cite[$\S$5.2]{MPY22} and \cite[$\S$5.7]{HMMS22}.  We conjecture
that \ts $\{s(\la,k)\}$ \ts is $\SP$-complete.

\smallskip

\item {\sc HurwitzNumberUnaryInput} \ts $\{h_{g\mu}\}$ \ts is $\SP$-complete
\cite{PP22+}.

\end{enumerate}

\smallskip

\subsection{Related problems}\label{ss:addendum-extra}
Here are a few additional problems we think are worth solving (all in unary).

\smallskip

\begin{enumerate}
\item  Let \ts $\La \ssu \nn^3$ \ts be a \defng{3-dim Young diagram}.
% (sometimes called \defng{solid partition}).
Denote by \ts $P_\La$ \ts
the corresponding poset.  We conjecture that \ts $\{e(P_\La)\}$ \ts
is \ts $\SP$-complete. In fact, we conjecture this holds for \ts $\La$ \ts
of height two.

\smallskip

\item  Let \ts $\La \ssu \nn^3$ \ts be a \defng{3-dim Young diagram}.
We conjecture that the number of domino tilings of~$\La$ is \ts
$\SP$-complete. Again, we conjecture this holds for \ts $\La$ \ts
of height two.

\smallskip

\item  Let \ts $\cA=(A_1,\ldots,A_n)$ \ts and \ts $\cB=(B_1,\ldots,B_n)$ \ts
be two collections of points in~$\nn^3$.  Denote by \ts $f(\cA,\cB)$ \ts
the number of collections \ts $(\ga_1,\ldots,\ga_n)$ \ts of
nonintersecting shortest paths \ts $\ga_i: A_i \to B_i$, \ts $1\le i \le n$.
We conjecture that \ts  $\big\{f(\cA,\cB)\big\}$ \ts is $\SP$-complete.

\smallskip

\item  Let \ts $\rG(d_1,\ldots,d_n)$ \ts be the number of simple graphs with
given degree sequence.  We conjecture that \ts $\{\rG(d_1,\ldots,d_n)\}$ \ts is
$\SP$-complete.
For many related results and further references, see e.g.~\cite{Wor18}.
The \emph{vanishing problem} \ts $\{\rG(d_1,\ldots,d_n)>^?0\}$ \ts
is in $\poly$ even in binary, by the \defng{Erd\H{o}s--Gallai theorem} (see e.g.\ \cite{AL05}).
\smallskip

\item  Let \ts $G\ssu S_n$ \ts be a permutation group given by its generators.
Recall that the size \ts $|G|$, the size of the commutator \ts $|[G,G]|$ \ts
and many other functions are in~$\FP$, see e.g.~\cite{Ser03}.
What about the number \ts $c(G)$ \ts of \defng{conjugacy classes}?

\smallskip

\item  We conjecture that tileability of simply-connected regions in $\rr^2$
with the unit square and the unit edge equilateral triangle is $\NP$-complete
(rotations are allowed, see~\cite[App.~1]{Zin09}), and that the number of such
tilings is $\SP$-complete.

\smallskip

\item  Let \ts $Q\ssu \rr^2$ \ts be a centrally-symmetric polygon with integer
side lengths, and let $T$ is a fixed set of rhombi tiles with unit sides.
We conjecture that the number of tilings of $Q$ with $T$ is
$\SP$-complete.\footnote{For the connection to
reduced factorizations, see~\cite{Eln97}.}  Note that the existence of such
tilings is in~$\poly$, see e.g.~\cite{KS92}. We refer to \cite[$\S$5]{Ken93}
and \cite[$\S$5.3, $\S$7]{Pak03} for some background.

\smallskip

\item  We conjecture that for \defng{binary matroids} \ts represented by vectors in $\ff_2^n$,
the number of bases is
$\SP$-complete.\footnote{Despite claims in \cite{Ver98},
this problem is unresolved since Vertigan's proof remains unwritten, and even
the proof idea is unavailable (Dirk Vertigan, personal communication, April~2010).}
Note that for \defng{paving matroids} \ts represented by cycles, and for
bicircular matroids, the number of bases
is $\SP$-complete, see \cite[$\S$3]{Jer06} and \cite[$\S$3]{GN06}.

\end{enumerate}

}
\bigskip

{\small
\section{Proofs}\label{s:proofs}

\subsection{Proof of Proposition~\ref{p:seq-tri}}\label{ss:proofs-Prop-tri}
Let $\cT_n$ be the set of \defn{rooted plane triangulations} on~$n$
vertices, and let $b_n:=|\cT_n|$.  Here the root in a triangulation
$G=(V,E)$ is a flag $(v,e,F)$, where $v\in V$, $e=(v,w)\in E$ and $F$
is a face in~$G$ containing~$e$.
\defng{Tutte's product formula} \ts shows that $\{b_n\}$ can be computed in
poly$(n)$ time, see e.g.~\cite{Sch15}.
Moreover, Poulalhon and Schaeffer~\cite{PS06} gave a bijective proof
of Tutte's formula, by constructing a bijection $\Phi: \cT_n \to \cB_n$
between rooted plane triangulations on~$n$ vertices
and certain \emph{balanced plane trees} with two additional markings.
It follows from the construction that both $\Phi$ and $\Phi^{-1}$ can
be computed in $O(n)$ time.

Let $\Ga=\Aut(G)$ be the group of automorphisms of graph~$G$.
It is easy to see that the stabilizer subgroup $\St_G(v,e,F)=1$,
because rooted triangulations have a unique topological embedding
into a sphere.  In particular, this implies that \ts $|\Ga|\le 4|E| = O(n)$.
% \footnote{It
% follows from Schramm's version of the circle packing theorem that $\Ga\ssu \rO(3,\rr)$,
% which in turn implies that $|\Ga|=O(n)$.  We refer to~\cite[$\S$11]{Pak10}
% for an accessible introduction.}
Recall that for all planar graphs, the effective
graph isomorphism can be done in linear time~\cite{HW74}.  In summary,
for each $e'=(v',w')\in E$, it can be decided in poly$(n)$ time whether
there is an automorphism $g\in \Ga$ s.t.\ $g\cdot v=v'$ and $g\cdot w=w'$.
Moreover, when it exists such $g$ can be computed explicitly.

Now, a combinatorial interpretation for $\{a_n\}$ can be constructed as follows.
Let $t \in \cB_n$.  Use the \defng{Depth First Search} (DFS)
around $t$ to obtain a unique labeling of vertices of~$t$.  Use $\Phi^{-1}$ to
transferred this labeling onto $\tau:=\Phi^{-1}(t)\in \cT_n$. Use the argument above
to compute all $O(n)$ relabelings $\tau'$ of~$\tau$.
% If $\Phi^{-1}(\tau')$
% is not a DFS labeling of some balanced tree $t'\in \cB_n$, discard~$\tau'$.
For all such $\tau'$, check if they are isomorphic to~$\tau$, and if
not discard such~$\tau'$.  Compute $\Phi(\tau')$ for all such rooted
triangulations.  We obtain exactly $|\Ga \cdot \tau|=O(n^2)$ balanced plane trees.
Accept~$t$ if it is lex-smallest, and reject otherwise.  The details are
straightforward.  \qed

\subsection{Proof of Proposition~\ref{prop:char-dist}}\label{ss:proofs-char-dist}
Recall that \. $\rho^{(n)} =  1\uparrow_{C_n}^{S_n}\ts$, where  $C_n\simeq \zz_n$
is the usual cyclic subgroup of~$S_n$.
Let \ts $\mu=(\mu_1,\ldots,\mu_\ell)\vdash n$ \ts be a partition into
distinct parts: \ts $\mu_1> \ts\ldots \ts > \mu_\ell>0$.
We have:
$$\aligned
\rho^\mu \ & = \ 1\uparrow_{C_{\mu_1} \times \. \cdots \. \times C_{\mu_\ell}}^{S_n}  \ = \
\rho^{(\mu_1)}\otimes \. \cdots \. \otimes \rho^{(\mu_\ell)} \uparrow_{S_{\mu_1} \times\. \cdots \. \times S_{\mu_\ell}}^{S_n}\\
& = \ \sum_{\la\vdash n} \. \Bigg[ \sum_{\nu^{(1)}\vdash \mu_1} \.
\. \ldots \. \sum_{\nu^{(\ell)}\vdash \mu_\ell}  \, c^{\la}_{\nu^{(1)}, \. \ldots \., \. \nu^{(\ell)}} \,
a_{\nu^{(1)}\ts (\mu_1)} \. \cdots \. a_{\nu^{(\ell)}\ts (\mu_\ell)} \Bigg]\, \chi^\la\,,
\endaligned
$$
where the \defn{generalized LR~coefficient}
$$
c^{\la}_{\nu^{(1)}, \. \ldots \., \. \nu^{(\ell)}} \ := \ \sum_{\tau^{(1)},\ldots,\tau^{(\ell-2)}} \. c^{\la}_{\nu^{(1)} \tau^{(1)}}
c^{\tau^{(1)}}_{\nu^{(2)} \tau^{(2)}} \. \cdots \. c^{\tau^{(\ell-2)}}_{\nu^{(\ell-1)} \nu^{(\ell)}}
$$
are written as sums of products of LR~coefficients.  Combining~\eqref{eq:KW} and the LR~rule, this 
gives a (rather cumbersome) combinatorial interpretation of the
multiplicity \ts $a_{\la\mu}$ \ts
of~$\chi^\la$\ts.  This is clearly in~$\ts\SP$, which proves the first part.
The second part follows from the same argument, with
rows $(\mu_i)$ replaced by rectangles with distinct lengths.\qed

}

\bigskip

{\small

\section{Final remarks}\label{s:finrem}

\subsection{} \label{ss:finrem-hist}
The term \defn{combinatorial interpretation} \ts seems to be
relatively recent and was first used by Hardy \cite[$\S$6.9]{Har40},
in connection with the \defng{Rogers--Ramanujan identities}.  There, the
meaning was literal, to say that RR~identities can be restated in a
combinatorial language.  The first modern usage was by Kaplansky and
Riordan \cite[p.~262]{KR46}, to say that \defng{Stirling numbers of second kind} \ts
have a ``combinatorial interpretation'' as the number of \defng{rook placements}
on the staircase shape.\footnote{This paper was extremely influential, and the
result can be found e.g.\ in \cite[Cor.~2.4.2]{Sta99}. }

Part of the reason is linguistic.  For example, MacMahon used plenty of
``interpretations'' in his celebrated \emph{Combinatory Analysis}~\cite{Mac15}, so
the term ``combinatory interpretation'' can be found in several papers,
e.g.\ in \cite{Ing26}.

\subsection{} \label{ss:finrem-dichotomy}
There is a reason we are so cavalier with many $\SP$-hardness
conjectures in Section~\ref{s:addendum}.  Roughly speaking, this is because
the universe of interesting combinatorial $\FP$ functions is quite small,
and is reduced to:

\smallskip

\begin{itemize}
\item[$\circ$] \defn{\bf dynamic programming} (see e.g.\ \cite[$\S$15]{CLRS09}),

\item[$\circ$]
\defn{\bf determinant formulas}, e.g.\ \eqref{eq:AFDF}, the \defng{matrix-tree theorem} (MTT), the  \defng{Lindstr\"om--Gessel--Viennot theorem}
(LGV, see e.g.\ \cite[$\S$5.4.2]{GJ83}),
and the \defng{Kasteleyn formula} (cf.~$\S$\ref{ss:addendum-graphs}),\footnote{Except for the MTT,
the clue to all of these is \emph{planarity}.  Note that the \eqref{eq:AFDF} follows from the LGV
and the limit argument.  The bijection in~\cite{KPW00} shows
that the MTT implies the Kasteleyn formula.}

\item[$\circ$]
\defn{\bf explicit formulas}, e.g.\ \eqref{eq:HLF} and \defng{MacMahon's box formula}
(see e.g.\ Eq.~(7.109) in~\cite{Sta99}).\footnote{As we mentioned earlier (see~$\S$\ref{ss:Sym-accidents}),
both of these formulas are derived from determinant formulas.}

\end{itemize}

\smallskip

\nin
If your function clearly does not belong to either of these
(not well-defined) classes, there is a
good chance it is not in~$\FP$.  Although in principle there are
are \defng{intermediate} counting classes between $\FP$ and~$\SP$ (unless $\FP=\SP$),
in practice no one has seen them.\footnote{Joshua Grochow
\href{https://cstheory.stackexchange.com/a/25410/16704}{proposed}
the number of graph isomorphisms as
an intermediate function; that was before Babai's breakthrough \cite{Bab18}.
Also, a well-known expert once suggested to us (personal communication),
that computing the number \ts $\rT(\la,\mu)$ \ts of contingency tables might be intermediate (in unary); we disagree.}
Various dichotomy $\NP$-completeness results only reinforce this belief,
see e.g.\ classic paper~\cite{HN90} and the most recent breakthrough~\cite{Zhuk20}.
In summary: \emph{If you don't see how to compute a function using standard approaches,
then most likely it is provably hard to do.}

\subsection{} \label{ss:finem-TLDR}
We have used a lot of space presenting our arguments in support of our
conjectures that many combinatorial functions in Algebraic Combinatorics
are (probably) not in~$\SP$.  Here is a quick summary of our arguments, TLDR style:

\smallskip

\begin{itemize}
\item[$\circ$]
\defna{\bf Are you smarter than everyone?} \ts
So many people worked on these problems, if there was a combinatorial interpretation
it would have been discovered by now.

\smallskip

\item[$\circ$]
\defna{\bf You can't get there if you are going in the wrong direction.} \ts If everyone in
a community thinks a problem has a positive solution and they can't even agree what
would it mean to have a negative solution, there is a chance the problem never gets solved.
At the very least the community should hedge and pursue both directions.

\smallskip

\item[$\circ$]
\defna{\bf There is already one miracle. Why are you expecting another one?} \ts
The LR~rule and all its variations are magical.  But they are caused by one true miracle:
the~RSK.  All the variations on the theme (jeu-de-taquin, octahedral map, etc.),
are equivalent in a formal sense, and thus simply the RSK in disguise.  Given the
scarcity of miracles, it seems unreasonable to hope for a positive solution without
using the RSK or its relative in an essential way.

\smallskip

\item[$\circ$]
\defna{\bf Nothing comes from nothing.} \ts The proof that a function is positive or integral
is based on a sequence of arguments.  For certain type of arguments that are given by
algebraic inequalities, we can show they do not have oblivious combinatorial
interpretations.
% So far, this approach is limited, of course.
So either your favorite function is very special and its proof avoids
all such arguments, or you need another proof.
\end{itemize}

\smallskip

\subsection{}  \label{ss:finem-california}
While writing this survey, we had an unmistakable feeling of
mapping the \defng{terra incognita}.  This reminded us of Nicolas Sanson's 1650 map
of  North America, where the author made a logical leap and extend shore lines
to create the \defng{Island of California}, see below.\footnote{This image is in public domain.
High resolution version is available from \ts
{\sf \href{https://commons.wikimedia.org/wiki/File:Insel_Kalifornien_1650.jpg}{Wikimedia Commons}.}}
You can imagine why he did it, but it's still an important error which took about fifty years
to correct.  We can't wait to find out if it's us who are making unjustified
logical leaps, or it's others who have been using a wrong map.  We just hope
it will take less than fifty years.

 \newpage

\ {}

\vskip.2cm

\begin{figure}[hbt]
		\includegraphics[height=4.cm]{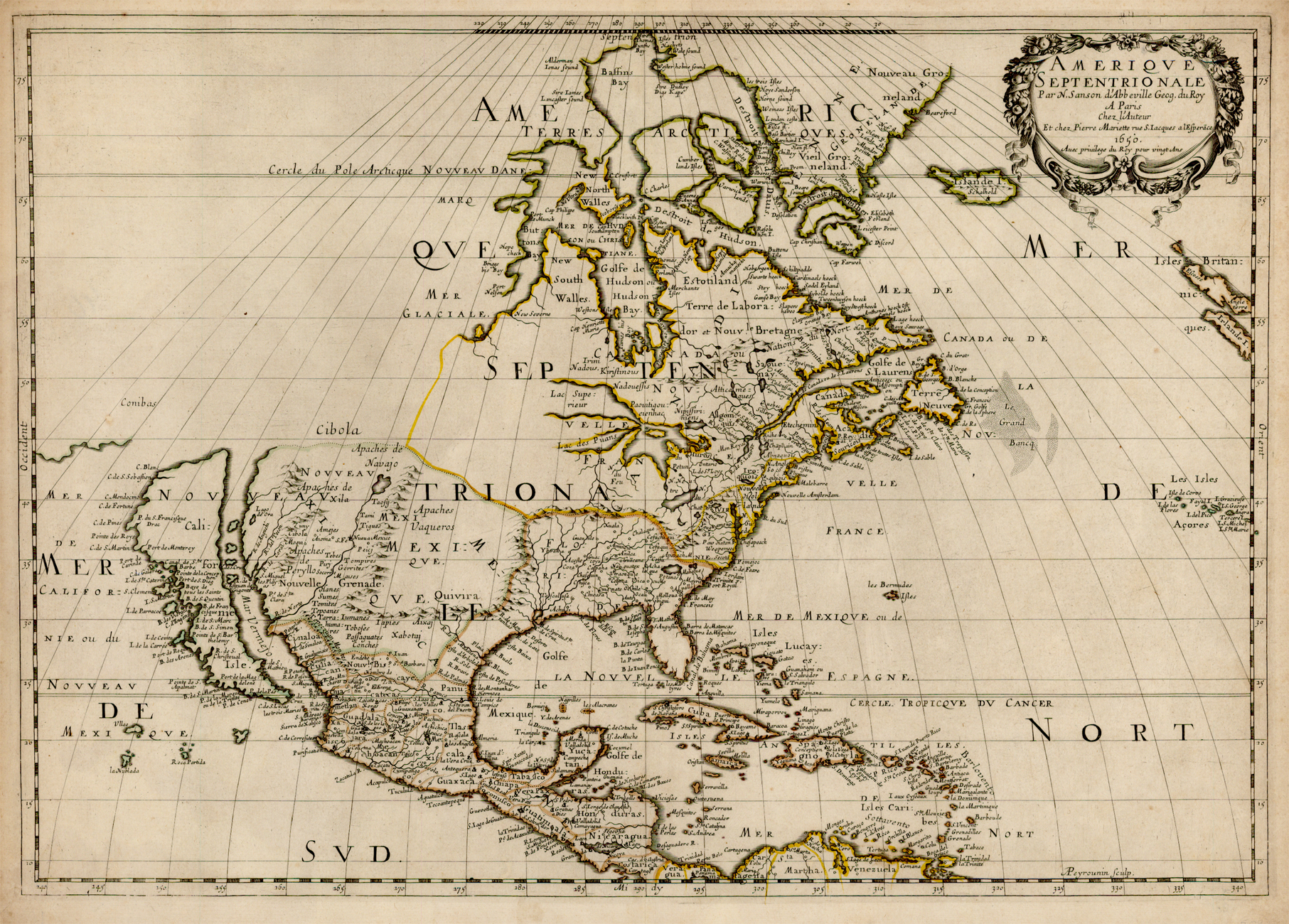}
		\label{f:map}
\end{figure}

\vskip.9cm

\subsection*{Acknowledgements}
We dedicate this survey to the memory of Gian-Carlo Rota whose teaching
and philosophy influenced both our approach to the subject and our
view of the world.  We are embarrassed that it took so long for us to
recognize his influence.

This paper would not exist without our recent work with Christian
Ikenmeyer and our extensive collaboration with Greta Panova.
We have also benefitted from closed collaboration
with Swee Hong Chan, Sam Dittmer, Scott Garrabrant, Alejandro Morales,
Danny Nguyen, Ernesto Vallejo, Jed Yang and Damir Yeliussizov.
We are deeply grateful for the opportunity to work with them all.

We are also thankful for helpful discussions on the subject with
% Karim Adiprasito,
Scott Aaronson,
Olga Azenhas, Cyril Banderier, Sasha Barvinok, Sara Billey, Art\"em Chernikov,
Jes\'us De Loera, Michael Drmota, \'Eric Fusy, Pavel Galashin, Nikita Gladkov,
Darij Grinberg, Josh Grochow, Leonid Gurvits, Zach Hamaker, Sam Hopkins, 
Mark Jerrum, Allen Knutson, Christian Krattenthaler, Greg~Kuperberg, 
Svante Linusson, Tyrrell McAllister, F\"edor Petrov, Vic Reiner, 
Yair Shenfeld,  Richard Stanley, Josh Swanson, Ramon van Handel, 
Dennis White, Nathan Williams, Alex Yong and Paul Zinn-Justin.
We thank Vince Vatter for graciously sharing with us~\cite{Vat18}.

Some of this material was accumulated over the years based
on numerous conversations.  We sincerely apologize to anyone
we forgot to mention.  The author was partially supported by the~NSF.

}

% \newpage

%%%%%%%%%%%%%%%%%%%%%%%%%%%%%%%%%%%%%%%%%%%%%%%%%%%%%%%%%%%%%%%%%%%%%%%%

%\vskip.7cm	
\end{document}